\documentclass[11pt]{article}

\usepackage{amsmath,amsthm,amsfonts,amssymb}
\usepackage[margin=2.2cm]{geometry}
\usepackage{thmtools}
\usepackage{hyperref}
\usepackage{enumitem}
\usepackage{authblk}
\usepackage{orcidlink}
\usepackage{booktabs}
\usepackage{mathrsfs}
\usepackage{subfigure}
\usepackage{float}
\usepackage[T1]{fontenc}
\usepackage{lmodern}
\usepackage{algorithm}
\usepackage{algpseudocode}
\usepackage{caption}
\usepackage{lipsum}
\hypersetup{
	colorlinks=true,
	linkcolor=blue,
	citecolor=blue,
	urlcolor=blue
}
\usepackage[mathscr]{euscript}
\allowdisplaybreaks
\newtheorem{thm}{Theorem}[section]
\newtheorem{lem}[thm]{Lemma}
\newtheorem{prop}[thm]{Proposition}

\theoremstyle{definition}
\newtheorem{defin}[thm]{Definition}
\newtheorem{obs}[thm]{Observation}

\usepackage{tikz}
\usetikzlibrary{positioning}
\usepackage{xfrac}
\usepackage{pdflscape}
\usepackage{tikz-cd}
\usetikzlibrary{patterns}
\usepackage{enumitem}
\usepackage{graphicx}
\usepackage{adjustbox}
\newcommand{\Crit}{\operatorname{Crit}}
\newcommand{\Ker}{\operatorname{Ker}}

\newcommand{\GT}{\operatorname{GT}}
\newcommand{\id}{\operatorname{id}}
\newcommand{\V}{\mathcal{V}}
\newcommand{\Va}{\V_{\alpha}}
\newcommand{\N}{\mathcal{N}(X)}
\newcommand{\VN}{\V_{\mathcal{N}}}
\newcommand{\wX}{\widetilde{X}}
\newcommand{\CrV}{\Crit_q^{\V}(\wX)}
\newcommand{\CrW}{\Crit_q^{\mathcal{W}}(\wX)}

\newcommand{\CrWq}{\Crit_{q+1}^{\mathcal{W}}(\wX)}
\newcommand{\CrqZ}{\Crit_q^{\mathcal{Z}}(L)}
\newcommand{\W}{\mathcal{W}}
\newcommand{\p}{\mathcal{Z}}
\newcommand{\CW}{C^{\W}}
\newcommand{\Wo}{\overline{\mathcal{W}}}
\newcommand{\Wa}{\mathcal{W}_\alpha}
\newcommand{\oW}{\overline{\W}_\alpha}

\newcommand{\Aa}{\bar{A}_{\alpha}}
\newcommand{\wA}{\widetilde{A}_\alpha}
\newcommand{\Z}{\mathbb{Z}}
\newcommand{\Lq}{\mathcal{L}}
\newcommand{\Lz}{L^{\mathcal{Z}}}
\newcommand{\h}{H}
\newcommand{\pL}{\partial^{\Lq}}
\newcommand{\pW}{\partial^{\W}}
\newcommand{\Tf}{\GT(f_{q+1}((\tau,\beta)), f_q((\sigma,\alpha)))}
\newcommand{\Gba}{\Gamma((\tau,\beta), (\sigma,\alpha))}
\newcommand{\Tba}{\GT(\tau_\beta, \sigma_\alpha)}
\newcommand{\z}{\mathcal{Z}}
\newcommand{\az}{\alpha_0}
\newcommand{\ao}{\alpha_1}

\newcommand{\s}{\sigma'_{\alpha_1}}
\newcommand{\sa}{\sigma_{\alpha_0}}
\newcommand{\tb}{(\tau_1)_{\beta}}
\newcommand{\Pa}{P_{\sa, \alpha_1}}

\newcommand{\Crv}{\Crit^{\V}}
\newcommand{\CrP}{\Crit^{\p}_q(C)}
\newcommand{\CrPq}{\Crit^{\p}_{q+1}(C)}
\newcommand{\CP}{(\mathcal{C}^{\p}_{\#}, \partial^{\p}_{\#})}
\newcommand{\C}{(\mathcal{C}_{\#}, \partial_{\#})}
\newcommand{\Lc}{(\Lq_{\#}(X), \partial^{\Lq}_{\#})}
\newcommand{\M}{\mathcal{M}}
\newcommand{\K}{K_{5,5}}
\newcommand{\Nm}{\mathcal{N}(\M(\K))}
\newcommand{\LM}{(\Lq_{\#}(\M(\K)), \pL_{\#})}
\newcommand{\LMZ}{(\Lq^{\z}_{\#}(\M(\K)), \partial^{\Lq^{\z}}_{\#})}
\newcommand{\F}{(F_{\#}(\M(\K)), \partial^{F}_{\#})}
\newcommand{\LMz}{(F^{\z'}_{\#}(\M(\K)), \partial^{F^{\z'}}_{\#})}
\newcommand{\kk}{K_{4,3}}
\newcommand{\Kk}{K_{4,2}}
\newcommand{\vv}{\mathscr{V}}
\newcommand{\U}{\mathcal{U}}
\newcommand{\si}{\sigma_{\alpha}}
\newcommand{\ax}{[\alpha;x]}
\newcommand{\by}{[\beta;y]}
\newcommand{\bz}{\beta_0}

\title{A combinatorial nerve theorem for effective homology computation}

\author[a]{Sucharita Barik\thanks{\texttt{sucharita.barik.126@tcgcrest.org}}}

\author[b]{Anupam Mondal\thanks{\texttt{anupam.mondal@tcgcrest.org}}}

\author[c]{Sajal Mukherjee~\thanks{\texttt{sajal.mukherjee@tcgcrest.org}}}

\author[d]{Pritam Chandra Pramanik~\thanks{\texttt{pritam.pramanik.80@tcgcrest.org}}}

\author[e]{Arundhati Rakshit~\thanks{\texttt{arundhati.rakshit.124@tcgcrest.org}}}

\affil[a,b,c,d,e]{\small Institute for Advancing Intelligence (IAI), TCG CREST, Kolkata--700091, India}
\affil[a,b,c,e]{\small Academy of Scientific and Innovative Research (AcSIR), Ghaziabad--201002, India}
\date{}
\begin{document}
	\maketitle
	\begin{abstract}
		The celebrated (homological) nerve theorem makes use of spectral sequences to determine the homology of a simplicial complex. However, this theorem cannot effectively compute the homology in every circumstance. In this paper, we develop an effective version of the nerve theorem, yielding a new and powerful tool for homology computation.
		
		The essence of our theorem can be formulated in the following manner. Suppose, $X$ is a simplicial complex with covering subcomplexes $A_1, \dots ,A_k$, that is, $X= \cup_{i=1}^k A_i$ and $\N$ is the nerve of $X$ with respect to its covering. Let $\Wa$ be a given gradient vector field on $A_{\alpha}(=\cap_{i \in \alpha} A_i)$ for each $\alpha \in \N$. Then, we use the mere information of the gradient trajectories in $A_{\alpha}$ for each $\alpha \in \N$ to explicitly compute the homology groups of $X$. Furthermore, we point out here, that these gradient vector fields do not need to be coherent, that is, they do not need to coincide on the intersections, which gives us ample flexibility to apply our theorem. Moreover, we can further simplify the computation of the homology groups using a gradient vector field on the nerve of $X$. Our approach is purely combinatorial, in the sense that it does not involve any notions of geometric realisation, continuity or homotopy, which makes it more amenable to computation and coding.
	\end{abstract}
	
	
	
	\section{Introduction}\label{intro}
	Efficient homology computation plays an important role not only in pure mathematical problems (see \cite{Curry2015, Shareshian2007, Shukla}) but also in applied and scientific problems of varied kinds, for example, in the analysis of electromagnetic boundary value problems in computational electromagnetism and numerical solutions of three dimensional magnetostatic problems (see \cite{ DlotkoSpecogna2010, Rodrguez2013, Kotiuga1984, Suuriniemi2004}). Discrete Morse theory, introduced by Forman (see \cite{Forman1, Forman2}) is an exceptionally powerful tool used for efficient homology computation (see \cite{Harker2013, Matching}). Given a huge simplicial complex, it considerably reduces the number of generators of the chain groups, which makes the computation of homology more convenient. However, the efficiency of the computation of Morse homology relies heavily on assigning a good gradient vector field (one which reduces the number of generators significantly) on the complex, which is often challenging. As a matter of fact, it is an NP-hard problem to assign an optimal gradient vector field on a large simplicial complex (see \cite{NP1, NP2, NP3}). Another notable tool used for homology computation (used mostly for theoretical purposes rather than explicit computation) is the homological nerve theorem, which is applicable in the following setup.
	Suppose $X$ is an (abstract) simplicial complex with subcomplexes $A_1, \dots ,A_k$ such that $\cup_{i=1}^kA_i=X$. Then, the (homological) nerve theorem can be used to determine the homology of $X$, under certain strong conditions on the subcomplexes and their possible intersections (e.g., they are acyclic up to a certain dimension), using Leray spectral sequence (see \cite{Borsuk1948, Leray1945}). However, in general cases, that is, when the subcomplexes and their intersections are not necessarily acyclic, the spectral sequence approach does not always offer an explicit way to compute the homology of a simplicial complex.
	
	In our paper, we have devised a tool for efficient homology computation which aptly addresses the limitations of both the aforementioned tools. The idea is to \emph{break} the complex into smaller manageable subcomplexes and assign efficient gradient vector fields on these subcomplexes and their possible intersections, which is comparatively easier than assigning a sufficiently good gradient vector field on the whole complex. It is worth mentioning here that these gradient vector fields do not need to be \emph{coherent}, that is, they do not need to coincide on the intersections, which offers us ample flexibility. In such a setup, our theorem provides an algorithmic procedure to explicitly determine the homology groups of $X$ using the combinatorial information of the gradient trajectories in the subcomplexes and their possible intersections. Moreover, our theorem works without any additional conditions or restrictions on the subcomplexes or on their intersections. We mention here that our approach throughout this paper is purely combinatorial, that is, we do not use any notions of geometric realisation, continuity or homotopy anywhere, so that it is more amenable to coding.

	In order to proceed to our main result, we need to introduce a few terminologies.
	Let $X$ be an (abstract) simplicial complex. We denote the \emph{dimension} of a simplex $\sigma$ in $X$ as $\dim(\sigma)$ and a $q$-dimensional simplex $\sigma$ as $\sigma^{(q)}$. Suppose, $\sigma, \tau \in X$, such that $\sigma \subseteq \tau$ and $\dim(\tau)= \dim(\sigma) + 1$. Then $\sigma$ is a \emph{facet} of $\tau$. We denote this as $\sigma <_f \tau$.

	A \emph{discrete vector field} $\V$ is a collection of pairs of simplices in $X$, of the form $(\sigma, \tau)$, such that, $\sigma <_f \tau$, and each simplex of $X$ is contained in at most one pair of $\V$. A \emph{Forman $\V$-trajectory} (or simply, a trajectory) $P$ is a sequence of simplices of the following form,
	$$ P:~ \tau_0^{(q)}, \sigma_1^{(q-1)}, \tau_1^{(q)}, \dots ,\sigma_k^{(q-1)}, \tau_k^{(q)}, \sigma_{k+1}^{(q-1)},$$
	where  $(\sigma_{i}, \tau_{i}) \in \mathcal{V}$ for each $i \in [k]$ ($:=\{1, \dots,k\}$),  $ \sigma_{i} <_f \tau_{i-1}$ and $(\sigma_i, \tau_{i-1}) \notin \mathcal{V}$ for all $i\in [k+1]$. $P$ is called \emph{closed} if $\tau_0=\tau_r$ for some $r > 1$. A discrete vector field $\V$ is a \emph{gradient vector field} if there are no closed $\V$-trajectories.
	
	A simplex which does not appear in $\V$ is \emph{$\V$-critical} (or simply \emph{critical}). $\Crv_q(X)$ denotes the set of all $q$-dimensional $\V$-critical simplices in $X$ while $\Crit^{\V}(X)$ denotes the set of all $\V$-critical simplices of $X$. A simplicial complex $X$ is \emph{collapsible} if it admits a gradient vector field $\V$ such that $\Crv(X)$ consists of only a single $0$-dimensional simplex. For a general introduction to discrete Morse theory, we refer to \cite{Scoville2019}.
	
	Next, we introduce the notion of \emph{nerve complex} of a simplicial complex $X$. Let $X$ be a simplicial complex and $A_1, \dots ,A_k$ be subcomplexes of $X$ such that $X= \bigcup_{i=1}^{k}A_i$. Then the \emph{nerve complex}, or simply \emph{nerve} of $X$ (with respect to $A_1, \dots,A_k$) is a simplicial complex $\mathcal{N}(X)$ as follows,
	$$ \N:=\{\alpha \subseteq [k] \mid \bigcap_{i \in \alpha}A_i \ne \emptyset\}. $$
	We denote $\bigcap_{i\in \alpha}A_i$ as $A_{\alpha}$.
	
	Now, we build the setup for stating our main result. Let $X$ be a simplicial complex with subcomplexes $A_1, \dots , A_k$ such that $ X= \bigcup_{i=1}^k A_i.$ Let $\N$ be the nerve of $X$. For each $\alpha \in \N$, we consider a copy of $A_{\alpha}$, namely $\Aa$, where each simplex $\sigma \in A_{\alpha}$ is relabelled as $\sigma_{\alpha}$, that is,
	$$ \Aa := \{\sigma_\alpha \mid \sigma \in A_{\alpha}\}.$$
	We note here that for each $\sigma \in A_{\alpha} \cap A_{\beta}$, $\sigma_{\alpha} \ne \sigma_{\beta}$, if $\alpha \ne \beta$, $\alpha, \beta \in \N$.
	
	Let $\overline{\W}_{\alpha}$ be a given gradient vector field on $\Aa$ for each $\alpha \in \N$. Now, for each $q \geq 0$,
	
	$$ L_q(X):= \bigcup_{i=0}^q \bigcup_{\alpha^{(q-i)} \in \N} \Crit_i^{\overline{\W}_{\alpha}}(\bar{A}_{\alpha}).$$
	We denote $L_q(X)$ as simply $L_q$ whenever the object of reference is implicit in the context.
	
	We next define a new combinatorial object, which will be of primary interest subsequently, namely, the \emph{generalised trajectories}.
	
	\begin{defin}\label{gt}
		Let $\gamma \in L_{q+1}(X)$, $\delta \in L_q(X)$, $q \geq 0$. Then we define a \emph{generalised trajectory} from $\gamma$ to $\delta$ for the following two cases.
		\begin{enumerate}[label=(\alph*)]
			\item Let $\gamma= (\tau_0)^{(i)}_{\alpha} \in L_{q+1}$, $\delta= (\sigma_{k+1})^{(i-1)}_{\alpha} \in L_q$, $\alpha^{(q+1-i)} \in \N$, where $1 \le i \le (q+1)$. A generalised trajectory from $\gamma$ to $\delta$ is a sequence of simplices of the form,
			
			$$ P: (\gamma=)~(\tau_0)^{(i)}_{\alpha}, (\sigma_1)^{(i-1)}_{\alpha}, (\tau_1)^{(i)}_{\alpha}, \dots, (\sigma_k)^{(i-1)}_{\alpha}, (\tau_k)^{(i)}_{\alpha}, (\sigma_{k+1})^{(i-1)}_{\alpha} ~(=\delta) $$
			where, $((\sigma_j)_{\alpha}, (\tau_j)_\alpha) \in \oW$ for each $j \in [k]$, $(\sigma_j)_\alpha <_f (\tau_{j-1})_\alpha$, $((\sigma_j)_\alpha, (\tau_{j-1})_\alpha) \notin \oW$ for each $j \in [k+1]$.
			
			A generalised trajectory of this kind can also be represented in the following way,
			$$ (\gamma=)~(\tau_0)^{(i)}_{\alpha} \rightarrow (\sigma_1)^{(i-1)}_{\alpha} \rightarrowtail (\tau_1)^{(i)}_{\alpha} \rightarrow \dots \rightarrow (\sigma_k)^{(i-1)}_{\alpha} \rightarrowtail (\tau_k)^{(i)}_{\alpha} \rightarrow (\sigma_{k+1})^{(i-1)}_{\alpha} ~(=\delta), $$
			where `$\tau \rightarrow \sigma$' denotes that $\sigma <_f \tau$ and `$\sigma_\alpha \rightarrowtail \tau_\alpha$' denotes that $(\sigma, \tau) \in  \oW$.
			
			The \emph{weight} of $P$,
			
			$$ w_G(P):= \left(\prod_{j=0}^{k-1}(-1)\langle \tau_j, \sigma_{j+1} \rangle \langle \tau_{j+1}, \sigma_{j+1} \rangle\right)\langle \tau_k, \sigma_{k+1}\rangle.$$
			(Intuitively, this is a Forman-trajectory, beginning from a $\oW$-critical simplex in $\Aa$, traversing entirely through $\Aa$, following $\oW$ and terminating at a $\oW$-critical simplex of $\Aa$.)\label{1st}
			
			\item Let $\alpha_t^{(q+1-i-t)} <_f \alpha_{t-1} <_f \dots <_f\alpha_0^{(q+1-i)} \in \N$, where $0 \le i < (q+1)$ and $1 \leq t \leq (q+1-i)$. Let $\gamma= (\tau_{1,0})^{(i)}_{\alpha_0} \in L_{q+1}$, $\delta= (\tau_{t,j_t+r})^{(i+t-1)}_{\alpha_t} \in L_q$. In this case, a generalised trajectory from $\gamma$ to $\delta$ is a sequence of simplices of the form (we refer to \autoref{gen} for an illustration),\label{second}
			$$ \begin{aligned}
				P:~ &(\tau_{1,0})^{(i)}_{\alpha_0}, (\tau_{0,1})^{(i-1)}_{\alpha_0}, (\tau_{1,1})^{(i)}_{\alpha_0}, \dots , (\tau_{0,j_1})^{(i-1)}_{\alpha_0},(\tau_{1,j_1})^{(i)}_{\alpha_0},(\tau_{1,j_1})^{(i)}_{\alpha_1}, (\tau_{2,j_1+1})^{(i+1)}_{\alpha_1}, \\ &(\tau_{1,j_1+1})^{(i)}_{\alpha_1}, (\tau_{2,j_1+2})^{(i+1)}_{\alpha_1}, \dots ,(\tau_{1, j_2-1})_{\alpha_1}^{(i)}, (\tau_{2,j_2})_{\alpha_1}^{(i+1)}, (\tau_{2,j_2})_{\alpha_2}^{(i+1)}, \dots, (\tau_{t-1,j_t-1})^{(i+t-2)}_{\alpha_{t-1}},\\ & (\tau_{t,j_t})^{(i+t-1)}_{\alpha_{t-1}}, (\tau_{t,j_t})^{(i+t-1)}_{\alpha_t}, (\tau_{t+1,j_t+1})^{(i+t)}_{\alpha_t}, (\tau_{t,j_t+1})^{(i+t-1)}_{\alpha_t}, \dots, (\tau_{t+1,j_t+r})^{(i+t)}_{\alpha_t}, (\tau_{t,j_t+r})^{(i+t-1)}_{\alpha_t},
			\end{aligned}  $$
			
			where,
			\begin{enumerate}[label=(\roman*)]
				\item $0 \le j_1 < j_2 < \dots < j_t$ and $r \ge 0$,
				\item $((\tau_{0,m})_{\alpha_0}, (\tau_{1,m})_{\alpha_0}) \in \Wo_{\alpha_0}$ for each $1 \le m \le j_1$,
				\item $((\tau_{l,m})_{\alpha_l}, (\tau_{l+1,m+1})_{\alpha_l}) \in \Wo_{\alpha_l}$, for each $j_l \le m \le j_{l+1} -1$ and $1\le l < t$,
				\item $((\tau_{t,m})_{\alpha_t}, (\tau_{t+1,m+1})_{\alpha_t}) \in \Wo_{\alpha_t}$ for each $j_t \le m \le j_t+r-1$,
				\item $(\tau_{0,m})_{\alpha_0} <_f (\tau_{1,m-1})_{\alpha_0}$ and $((\tau_{0,m})_{\alpha_0}, (\tau_{1,m-1})_{\alpha_0}) \notin \Wo_{\alpha_0}$ for each $1 \le m \le j_1$,
				\item $  (\tau_{l,m})_{\alpha_l} <_f (\tau_{l+1,m})_{\alpha_l}$, and $((\tau_{l,m})_{\alpha_l}, (\tau_{l+1,m})_{\alpha_l}) \notin \Wo_{\alpha_l}$, for each $j_l < m \le j_{l+1} -1$ and $1\le l \le t$, where, $j_{t+1}-1=j_t+r$.
			\end{enumerate}
			
			A generalised trajectory of this kind can be depicted as,
			$$ \begin{aligned}
				&(\tau_{1,0})^{(i)}_{\alpha_0} \rightarrow (\tau_{0,1})^{(i-1)}_{\alpha_0} \rightarrowtail (\tau_{1,1})^{(i)}_{\alpha_0} \rightarrow \dots  \rightarrow(\tau_{0,j_1})^{(i-1)}_{\alpha_0} \rightarrowtail (\tau_{1,j_1})^{(i)}_{\alpha_0} \Rightarrow (\tau_{1,j_1})^{(i)}_{\alpha_1} \rightarrowtail (\tau_{2,j_1+1})^{(i+1)}_{\alpha_1}\\ & \rightarrow (\tau_{1,j_1+1})^{(i)}_{\alpha_1} \rightarrowtail (\tau_{2,j_1+2})^{(i+1)}_{\alpha_1} \rightarrow \dots \rightarrow (\tau_{1, j_2-1})_{\alpha_1}^{(i)} \rightarrowtail (\tau_{2,j_2})_{\alpha_1}^{(i+1)} \Rightarrow (\tau_{2,j_2})_{\alpha_2}^{(i+1)} \rightarrowtail  \dots \rightarrow \\ &(\tau_{t-1,j_t-1})^{(i+t-2)}_{\alpha_{t-1}} \rightarrowtail (\tau_{t,j_t})^{(i+t-1)}_{\alpha_{t-1}} \Rightarrow (\tau_{t,j_t})^{(i+t-1)}_{\alpha_t} \rightarrowtail (\tau_{t+1,j_t+1})^{(i+t)}_{\alpha_t} \rightarrow (\tau_{t,j_t+1})^{(i+t-1)}_{\alpha_t}\rightarrowtail \dots \\ & \rightarrowtail (\tau_{t+1,j_t+r})^{(i+t)}_{\alpha_t} \rightarrow (\tau_{t,j_t+r})^{(i+t-1)}_{\alpha_t},
			\end{aligned}  $$
			where `$\sigma_\alpha\Rightarrow \sigma_\beta$' denotes that $\beta <_f \alpha$ and $\sigma \in A_\alpha \cap A_\beta$.
			
			The \emph{weight} of this trajectory $P$,
			$$\begin{aligned}
				w_G(P) := & (-1)^{\phi(i,t)} \prod_{l=1}^{t} \langle \alpha_{l-1}, \alpha_{l}\rangle \left(\prod_{k=0}^{j_1-1}(-1) \langle \tau_{1,k}, \tau_{0,k+1} \rangle \langle \tau_{1,k+1}, \tau_{0,k+1} \rangle \right) \\ &\prod_{l=1}^{t-1}\left(\left( \prod_{k=j_l}^{j_{l+1}-2}(-1)\langle \tau_{l+1,k+1}, \tau_{l,k} \rangle \langle \tau_{l+1,k+1}, \tau_{l,k+1} \rangle\right) (-1)\langle \tau_{l+1,j_{l+1}}, \tau_{l,j_{l+1}-1} \rangle \right)\\ &\left( \prod_{k=j_t}^{j_t+r-1} (-1)\langle \tau_{t+1,k+1}, \tau_{t,k} \rangle \langle \tau_{t+1,k+1}, \tau_{t,k+1} \rangle\right),
			\end{aligned}$$
			where $\phi(i,t)= it + \frac{t(t-1)}{2}$.
		
			(Intuitively, these are trajectories commencing from a $\Wo_{\az}$-critical simplex in $\bar{A}_{\az}$, traversing through $\bar{A}_{\az}$, following $\Wo_{\az}$ for some time, then entering $\bar{A}_{\ao}$, where $\ao <_f \az$, traversing through $\bar{A}_{\ao}$ for some time, following $\Wo_{\ao}$ and then entering $\bar{A}_{\alpha_2}$, where $\alpha_2 <_f \ao$, and so on, and finally terminates at a $\Wo_{\alpha_t}$-critical simplex in $\bar{A}_{\alpha_t}$ for some $\alpha_t$, where $\az <_f \ao <_f \dots <_f \alpha_t$. )
		\end{enumerate}
	\end{defin}

	From here on, we will refer to these two kinds of generalised trajectories as generalised trajectories of the first (\autoref{gt}\ref{1st}) and second kind (\autoref{gt}\ref{second}) respectively. \autoref{gen} illustrates a generalised trajectory of the second kind.
	\begin{figure}[h]
		\centering
		\begin{tikzpicture}[
			vertex/.style={draw, circle, fill=black, inner sep=0.05pt, minimum size=0.7mm},
			baseline, scale=1.3]
			
			\node (a) at (0,0){$(\tau_{1,0})_{\az}^{(1)}$};
			\node (b) at (0.8,-1){$(\tau_{0,1})_{\az}^{(0)}$};
			\node (c) at (1.6,0){$(\tau_{1,1})_{\az}^{(1)}$};
			\node (d) at (2.4,-1){$\cdots$};
			\node (n) at (3.2, -1){$(\tau_{0,j_1})_{\az}^{(0)}$};
			\node (e) at (3.8,0){$(\tau_{1,j_1})_{\az}^{(1)}$};
			\node (o) at (4.4,0) {$\hspace{0.2em}\Rightarrow$};
			\node (f) at (5.1,0){$(\tau_{1,j_1})_{\ao}^{(1)}$};
			\node (g) at (5.8,1){$(\tau_{1,j_1+1})_{\ao}^{(2)}$};
			\node (h) at (6.6,0){$\hspace{0.2em} \cdots$};
			\node (p) at (6.7,0){};
			\node (i) at (7.4,1){$(\tau_{2,j_2})_{\ao}^{(2)}$};
			\node (j) at (8,1){$\hspace{0.2em}\Rightarrow$};
			\node (k) at (8.7,1){$(\tau_{2,j_2})_{\alpha_2}^{(2)}$};
			\node (q) at (9.4,2){$(\tau_{3,j_2+1})_{\alpha_2}^{(3)}$};
			\node (r) at (10.1,1){$\hspace{0.2em}\cdots$};
			\node (s) at (10.2,1) {};
			\node (l) at (10.9,2){$(\tau_{3,j_2+r})_{\alpha_2}^{(3)}$};
			\node (m) at (11.6,1){$(\tau_{2,j_2+r})_{\alpha_2}^{(2)}$};
			
			\draw[->] (a) -- (b);
			\draw[>->] (b) -- (c);
			\draw[->] (c) -- (d);
			\draw[>->] (n) -- (e);
			\draw[>->] (f) -- (g);
			\draw[->] (g) -- (h);
			\draw[>->] (p) -- (i);
			\draw[>->] (k) -- (q);
			\draw[->] (q) -- (r);
			\draw[>->] (s) -- (l);
			\draw [->] (l) -- (m);
			
		\end{tikzpicture}
		\caption{A generalised trajectory of the second kind, initiating from $(\tau_{1,0})_{\az} \in L_{q+1}$ and terminating at $(\tau_{2,j_2+r})_{\alpha_2} \in L_q$. For this particular trajectory, $i=1$, $t=2$ and $\alpha_2<_f \ao <_f \az$. The weight of an arrow `$\sigma \rightarrow \tau$' contributes $\langle \tau, \sigma\rangle$, the weight of an arrow `$\sigma\rightarrowtail\tau$' contributes $-\langle \tau, \sigma\rangle$ while the weight of an arrow `$\sigma_\beta \Rightarrow \sigma_\alpha$' contributes $(-1)^{\dim(\sigma)}\langle\beta, \alpha\rangle$. Thus the weight of such a generalised trajectory turns out to be the product of the weights of all the arrows in the generalised trajectory, which is precisely as defined in \autoref{gt}\ref{second}. }\label{gen}
	\end{figure}
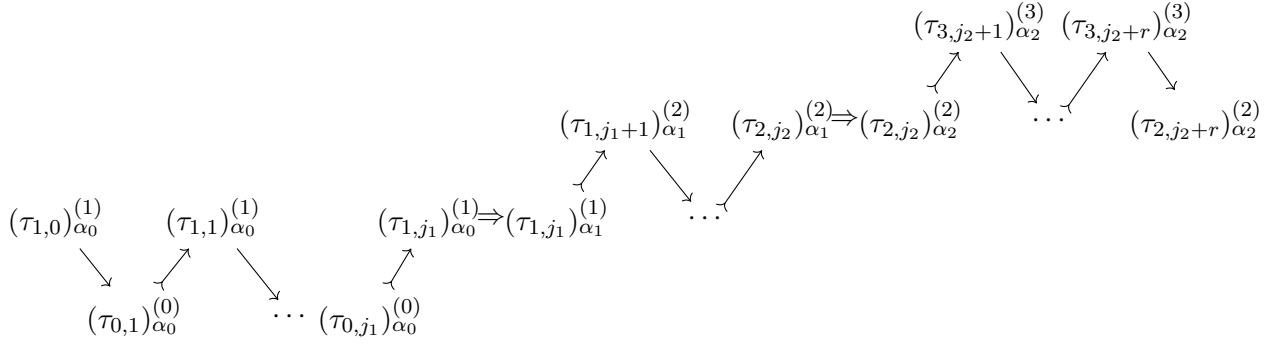

	Next, for each $q \ge 0$, $\Lq_q(X, \Z)$ is the free module generated by $L_q(X)$ over $\Z$. Here onwards, we will be working with the coefficient ring $\Z$ for all our purposes. So, for the sake of simplicity, we will write any chain group or homology group $D_{\#}(X, \Z)$ as simply $D_{\#}(X)$.
	We denote the set of all generalised trajectories from $\gamma \in  L_{q+1}$ to $\delta \in L_{q}(X)$ as $\GT(\gamma, \delta)$. Next, for each $q \ge 0$, we define a map $\pL_{q+1}: \Lq_{q+1}(X) \rightarrow \Lq_q(X) $, as follows.
	
	Let $\tau^{(i)}_{\alpha_0} \in L_{q+1}(X)$ and $\alpha_t^{(q+1-i-t)} <_f \alpha_{t-1} <_f \dots <_f\alpha_0^{(q+1-i)} \in \N$ where $0 \le i \le (q+1)$ and $0 \le t \le (q+1-i)$. Then,
	$$
	\pL_{q+1}(\tau^{(i)}_{\alpha_0}):=
	\sum \limits_{t=0}^{q+1-i} \sum\limits_{\sigma^{(i+t-1)}_{\alpha_t} \in L_q(X)} \left(\sum_{P \in \GT(\tau_{\alpha_0},\sigma_{\alpha_t})} w_G(P)\right) \sigma_{\alpha_t}.
	$$

	We extend this linearly to all of $\Lq_{q+1}(X)$.
	
	Now, we state our main result in terms of the above setup.
	\begin{thm}\label{main}
		$(\Lq_q(X), \pL_q)_{q \ge 0}$ is a chain complex and $H^{\Lq}_q(X) \cong H_q(X)$ for each $q = 0, \dots, d$, where $H_q(X)$ is the $q^{th}$ simplicial homology group of $X$ while $H^{\Lq}_q(X)$ is the $q^{th}$ homology group of $\Lc$.
	\end{thm}
	Another notable aspect of our theorem is that, here we do not even need to know the individual homology groups of the subcomplexes or of their intersections, which reduces significant number of computational steps. It is also noteworthy here, that if we choose the subcomplexes wisely, and assign \emph{good} gradient vector fields (with lower number of critical simplices), then it enhances our computational efficiency. However in principle, this theorem can always compute the homology of the simplicial complex, irrespective of our choices.
	
	Our paper is organised in the following manner. In Section~\ref{prelim}, we introduce some basic prerequisites needed to present our main result. In Section~\ref{dmt}, we introduce the concept of discrete Morse theory in simplotopal complex, which is the central tool used in our proof. Finally, the main theorem of our paper is proved in Section~\ref{thm}. In Section~\ref{app}, we prove a version of the usual homological nerve theorem, using our result (\autoref{main}). Towards the end, in Section~\ref{red}, we show how this chain complex $\Lc$ can be further reduced using a given gradient vector field on the nerve complex. Section~\ref{algo} demonstrates the computational aspects of our combinatorial nerve theorem. Here, we provide algorithms for the computation of homology groups of a simplicial complex using our theorems and finally conclude by computing the homology groups of the $5 \times 5$ chessboard complex (i.e., the matching complex of the complete bipartite graph $\K$) using the aforementioned algorithms, in Python $3.12.7$ and SageMath $9.3$.
	
	\section{Preliminaries}\label{prelim}
	\subsection{Simplicial complex}
	An \emph{(abstract) simplicial complex} $X$ is a collection of finite  sets with the property that for each $\tau \in X$, if $\sigma \subseteq \tau$, then $\sigma \in X$. A \emph{simplex} is an element of $X$. The \emph{dimension of a simplex} $\sigma$,  $\dim(\sigma):= |\sigma|-1$. An $i$-dimensional simplex is also referred to as an $i$-simplex. The \emph{dimension of $X$}, $\dim(X):= \max_{\sigma \in X} \dim(\sigma)$. A simplicial complex $Y$ is a \emph{subcomplex} of $X$ if $Y \subseteq X$.
	
	\begin{defin}{(Orientation on a simplex)}
		Let $X$ be a simplicial complex and $\sigma= \{v_0, \dots ,v_n\} \in X$. We define an equivalence relation on the set of all possible orderings of $v_0, \dots ,v_n$ such that for any two orderings $\gamma, \delta$, $\gamma \sim \delta$ if and only if $\gamma$ and $\delta$ differ by an even permutation. An \emph{orientation} on $\sigma$ is an ordering which represents an equivalence class under this relation. Thus, there are two distinct orientations of a simplex. A simplex $\{v_0, \dots ,v_n\}$ with an orientation represented by the ordering, $v_0, \dots ,v_n$ is an \emph{oriented} simplex, which we denote as $[v_0, \dots ,v_n]$.
	\end{defin}
	A simplex $\{v_0, \dots, v_n\} \setminus \{v_i\}$ equipped with the ordering $v_0, \dots ,v_{i-1}, v_{i+1}, \dots ,v_n$ is denoted as $[v_0, \dots , \widehat{v}_i, \dots v_n]$.
	\begin{defin}{(Incidence number)}
		Let $\tau, \sigma \in X$ such that $\tau=[v_0, \dots,v_n]$ and $\dim(\sigma)= \dim(\tau)-1$. Then, the \emph{incidence number} of $\sigma$ with respect to $\tau$,
		$$ \langle \tau, \sigma\rangle:= \begin{cases}
			0 & \text{ if } \sigma \nsubseteq \tau,\\
			(-1)^{i} & \text{ if } \sigma = [v_0, \dots, \widehat{v}_i, \dots ,v_n]\\
		\end{cases} $$
	\end{defin}
	Let $S_q(X)$ represent the set of all oriented $q$-simplices of $X$.
	Now, for each $q \ge 0$, the $q^{th}$ \emph{simplicial chain group} $C_q(X)$ is the free module generated by $S_q(X)$ over $\Z$, that is, $C_q(X)= \Z\langle S_q(X) \rangle$. The $(q+1)^{th}$ boundary map $\partial_{q+1}: C_{q+1}(X) \rightarrow C_q(X)$,
	$$ \partial_{q+1}(\tau):= \sum_{\sigma \in S_{q}(X)} \langle \tau, \sigma \rangle \sigma, \text{    for each    } \text{ } \tau \in S_{q+1}(X),$$

	We now extend it linearly to $C_{q+1}(X)$. 
	It can be shown that $\partial \circ \partial=0$, that is, $(C_{\#},\partial_{\#})$ is a chain complex. The \emph{$q^{th}$ simplicial homology group},
	$$ H_{q}(X):= \frac{\Ker(\partial_{q})}{ \operatorname{Im}(\partial_{q+1})}.$$
	
	\subsection{Isomorphism and homotopy equivalence of chain complexes}
	Suppose, we have two chain complexes $(P_{\#}, \partial^P_{\#})$, $(Q_{\#}, \partial^Q_{\#})$. A \emph{chain map} between these two chain complexes is a sequence of module homomorphisms $f_{\#}= \{f_n\}_{n \ge 0}$, $f_n: P_n \rightarrow Q_n$, such that,
	
	$$ f_{n}\circ \partial^{P}_{n+1}=\partial^Q_{n+1}\circ f_{n+1}  \text{ for each } n \geq 0.$$
	
	\[\begin{tikzcd}[sep=large]
		\cdots & { P_{n+1}} && {P_{n}} && {P_{n-1}} & \cdots & \\
		\\
		\cdots & {Q_{n+1}} && {Q_{n}} && {Q_{n-1}} & \cdots \\
		\\
		&&&&&&& {}
		\arrow[from=1-1, to=1-2]
		\arrow["{\partial_{n+1}^{P}}", from=1-2, to=1-4]
		\arrow["{f_{n+1}}"', from=1-2, to=3-2]
		\arrow["{g_{n+1}}", shift left=3, from=1-2, to=3-2]
		\arrow["{\partial_{n}^{P}}", from=1-4, to=1-6]
		\arrow["{\lambda_{n}}"', from=1-4, to=3-2]
		\arrow["{f_{n}}"', from=1-4, to=3-4]
		\arrow["{g_n}", shift left=3, from=1-4, to=3-4]
		\arrow[from=1-6, to=1-7]
		\arrow["{\lambda_{n-1}}"', from=1-6, to=3-4]
		\arrow["{f_{n-1}}"', from=1-6, to=3-6]
		\arrow["{g_{n-1}}", shift left=3, from=1-6, to=3-6]
		\arrow[from=3-1, to=3-2]
		\arrow["{\partial_{n+1}^{Q}}"', from=3-2, to=3-4]
		\arrow["{\partial_{n}^{Q}}"', from=3-4, to=3-6]
		\arrow[from=3-6, to=3-7]
	\end{tikzcd}\]
	Let $f_{\#}$ and $g_{\#}$ be two chain maps between the chain complexes $(P_{\#}, \partial^P_{\#})$ and $(Q_{\#}, \partial^Q_{\#})$. Then $f_{\#}$ and $g_{\#}$ are \emph{homotopic}, denoted by $f_{\#} \simeq g_{\#}$, if there exists a sequence of module homomorphisms $\lambda_{n}: P_n \rightarrow Q_{n+1}$ such that the following holds.
	$$ \lambda_{n-1}\circ \partial_{n}^P + \partial_{n+1}^Q\circ \lambda_{n} = f_{n} - g_{n}, \text{  for all } n \geq 1.$$
	
	$(P_{\#}, \partial^P_{\#})$ and $(Q_{\#}, \partial^Q_{\#})$ are called \emph{homotopy equivalent}, if there exists chain maps $f_{\#}:P_{\#} \rightarrow Q_{\#}$ and $h_{\#}: Q_{\#} \rightarrow P_{\#}$ such that $f_{\#} \circ h_{\#} \simeq \id_{Q_{\#}}$ and $h_{\#} \circ f_{\#} \simeq \id_{P_{\#}}$. Homotopy equivalent chain complexes have isomorphic homology groups. If there exists a chain map between two chain complexes, which is also an isomorphism, then the two chain complexes are isomorphic. If two chain complexes are isomorphic, then they are also homotopy equivalent.
	
	\subsection{Basics of algebraic Morse theory}\label{alg}
	We now introduce a generalisation of the idea of discrete Morse theory. The notion of acyclic pairing in this setup, generalises the idea of gradient vector field on simplicial complexes, while the idea of trajectories in discrete Morse theory translates to an analogous notion of paths in this setup.
	
	Let $\C$ be a free chain complex generated by a basis $C_{\#}$, over $\Z$. For any $y \in C_q$, $x \in C_{q-1}$, the \emph{incidence} of $x$ with respect to $y$ is the coefficient of $x$ in $\partial_q(y)$, denoted as $\langle y,x \rangle_{C}$.
	
	Now, we represent $\C$ as a ranked poset $(C, \leq)$, consisting of all the basis elements $\{x \in C_q\}_{q \ge 0}$. The \emph{rank} of any element $x$ is $\{q \mid x \in C_q\}$ and the order relation in this poset is given in the following manner. For each $y \in C_{q+1}$, $x \in C_q$, $x \leq y$ if and only if $\langle y, x\rangle_C \ne 0$. Next, a \emph{pairing} $\p$ on $(C, \le)$ is a collection of pairs of the form $\{(x,y) \mid x \in C_q, y \in C_{q+1}, \langle y,x \rangle_C = \pm1\}_{q \ge 0}$. Let $y \in C_{q+1}$, $x \in C_q$. Then a $\p$-\emph{path}, or simply a \emph{path} from $y$ to $x$ is a sequence of the form,
	\begin{equation}\label{z}
		\gamma: (y=)~y_0, x_1, y_1, \dots ,x_k, y_k, x_{k+1}~(=x),
	\end{equation}
	where, $x_{i}\le y_{i-1}$, $(x_{i}, y_{i-1}) \notin \p$ for each $i \in [k+1]$ and $(x_i, y_i) \in \p$ for each $i \in [k]$. If $y= y_r$ for some $r > 1$, then $\gamma$ is said to be \emph{cyclic}. $\p$ is \emph{acyclic} if there does not exist any cyclic $\p$-path. 
	
	Let $\p$ be an acyclic pairing on the chain complex $\C$ (i.e., on the poset $(C, \le)$). An element which does not appear in $\z$ is $\z$-\emph{critical} or simply critical. For each $q \ge 0$, the collection of all $\z$-critical elements of $C_q$ is denoted as $\CrP$. Given any $\p$-path $\gamma$ in $(C, \le)$ of the above form (refer to (\ref{z})), the weight of $\gamma$,
	
	$$ w_C(\gamma):= \left(\prod_{i=0}^{k-1}(-1)\langle y_i, x_{i+1}\rangle_C \langle y_{i+1}, x_{i+1} \rangle_C\right) \langle y_k, x_{k+1}\rangle_C.$$
	
	For $y \in C_{q+1}$, $x \in C_q$, we denote the set of all $\z$-paths from $y$ to $x$ as $\mathcal{P}(y,x)$.
	For each $q \ge 0$, $\mathcal{C}^{\p}_q$ is the free $\Z$-module generated by $\CrP$, that is, $\mathcal{C}^{\p}_q = \Z \langle \CrP\rangle$. For each $q \ge 0$, we define a map $\partial^{\p}_{q+1} : \mathcal{C}^{\p}_{q+1} \rightarrow \mathcal{C}^{\p}_q$ in the following manner.
	
	For each $y \in \CrPq$, $$\partial_{q+1}^{\p}(y):= \sum_{x \in \CrP} \left(\sum_{\gamma \in \mathcal{P}(y,x)}w_C(\gamma)\right)x.$$
	
	The next result (see \cite{Kozlov2005}) states that $\CP$ is indeed a chain complex and lays down the relation between $\C$ and $\CP$.
	
	\begin{thm}{\cite{Kozlov2005}}\label{koz}
		$\CP$ is a chain complex. Further, for each $q \ge 0$, $H^{\mathcal{C}}_q \cong H^{\mathcal{C}^{\p}}_q$, where $H^{\mathcal{C}}_q$ and $H^{\mathcal{C}^{\p}}_q$ represent the $q^{th}$ homology groups of $\C$ and $\CP$ respectively.
	\end{thm}

	\subsection{Simplotopal complex}
	
	Let $V_1, \dots ,V_m$ be a collection of finite mutually disjoint sets. A \emph{simplotopal complex} $K$ (see \cite{Meunier}) is a finite collection of non-empty sets, $\{(\tau_1, \dots, \tau_m) \mid \emptyset \ne \tau_i \subseteq V_i\}$ with the property that for each $i \in [m]$, if $\emptyset \ne \sigma_i \subseteq \tau_i$, then $(\sigma_1,\sigma_2, \dots ,\sigma_m)$ is also a member of this collection. A \emph{simplotope} is an element of $K$. Let $\tau=(\tau_1, \dots,\tau_m)\in K$. For each $\tau_i \subseteq V_i$, the dimension of $\tau_i$, $\dim(\tau_i):=|\tau_i| -1 $.  The \emph{dimension} of $\tau$ is, $$ \dim(\tau):= \sum_{i=1}^{m}\dim(\tau_i).$$
	The dimension of the simplotopal complex $K$, $\dim(K):=\max\{ \dim(\tau) \mid \tau \in K\}$. 	A $q$-dimensional simplotope $\tau$ is denoted as $\tau^{(q)}$. Let $\tau= (\tau_1,\tau_2, \dots ,\tau_m)$, $\sigma= (\sigma_1, \dots ,\sigma_m)  \in K$ such that $\emptyset \ne \sigma_i \subseteq \tau_i$ for each $i \in [m]$ and $\dim(\sigma) < \dim(\tau)$. Then $\sigma$ is a \emph{face} of $\tau$. In addition, if $\dim(\sigma)= \dim(\tau) -1$, then $\sigma$ is a \emph{facet} of $\tau$.
	Let $\tau=( \tau_1, \dots ,\tau_m) \in K$, where $\tau_i=\{v_{i0}, v_{i1}, \dots ,v_{ip_i}\}$ for each $i \in [m]$, and $\sigma$ be a facet of $\tau$. Then we can observe that $\sigma$ is of the form $(\sigma_1, \dots ,\sigma_m)$, where $\sigma_j= \tau_j \setminus \{v_{ji}\}$, $i \in \{0, \dots ,p_j\}$ for a unique $j \in [m]$, and $\sigma_k=\tau_k $ for all $k \ne j$, $k \in [m]$.

	We make a note here, that we will sometimes denote a simplotope $\sigma=(\sigma_1, \dots, \sigma_m)$ as simply $\sigma$, for convenience. This will not cause any notational ambiguity, as the meaning will be made clear from the context.
	
	\begin{defin}{(Orientation on a simplotope)}
		Let $K$ be a simplotopal complex. Suppose $\sigma= (\sigma_1, \dots, \sigma_m) \in K$. Then we define an \emph{ordering} as a permutation in $S_{\sigma_1} \times S_{\sigma_2} \times \dots \times S_{\sigma_m}$, where $S_{\sigma_i}$ denotes the permutation group on $\sigma_i$ for each $i$. Let $\alpha= \cup_{i=1}^m \sigma_i$. Now, we define an equivalence relation on $S_{\sigma_1} \times S_{\sigma_2} \times \dots \times S_{\sigma_m}$ as follows. Let $\gamma, \delta \in S_{\sigma_1} \times S_{\sigma_2} \times \dots \times S_{\sigma_m}$. Then $\gamma \sim \delta$ if and only if $\gamma$ and $\delta$ differ by an even permutation in $S_{\alpha}$. An ordering on $\sigma$ represented by each equivalence class under this relation is an \emph{orientation} on the simplotope $\sigma$. Thus, if $\dim(\sigma) > 0$, then there are precisely two equivalence classes in $S_{\sigma_1} \times S_{\sigma_2} \times \dots \times S_{\sigma_m}$, each of which represent an orientation of $\sigma$.
	
		A simplotope equipped with an orientation is an \emph{oriented simplotope}. We denote the collection of all $q$-dimensional oriented simplotopes of $K$ as $S_q(K)$.	
	\end{defin}
	
	\begin{defin}{(Incidence number)}
		Let $\tau=(\tau_1, \dots , \tau_m) \in S_{q+1}(K)$,  where $\tau_j= \{v_{j0}, v_{j1}, \dots ,v_{jp_j}\}$ for each $j \in [m]$. Let $\sigma= (\sigma_1, \dots, \sigma_m) \in S_q(K)$ such that $\sigma$ is a facet of $\tau$ and $\sigma$ is equipped with the orientation represented by the ordering $v_{10}, v_{11}, \dots ,v_{1p_1}, \dots  ,v_{j0}, \dots, \widehat{v}_{ji}, \dots ,v_{jp_j}, \dots ,v_{m0}, v_{m1}, \dots ,v_{mp_m}$, where $\widehat{v}_{j_i}$ represents the removal of $v_{j_i}$ from the sequence. Then the \emph{incidence number} of $\sigma$ with respect to $\tau$,
		$$ \langle \tau, \sigma\rangle:=
		(-1)^{\left(\sum\limits_{k=1}^{j-1}\dim(\sigma_k)\right)+i}.$$
		However, if $\sigma \in S_q(K)$ is not a facet of $\tau$, then $\langle \tau, \sigma \rangle:=0$.
	\end{defin}
	The next observation follows naturally from the definition given above.
	\begin{obs}\label{inc}
		Let $\tau=(\tau_0, \tau_1) \in K$ and $\sigma=(\sigma_0, \sigma_1)$ be a facet of $\tau$, where $\tau_0, \tau_1 \in X$ for some simplicial complex $X$. Then the following holds.
		\begin{enumerate}
			\item If $\sigma_0 <_f \tau_0$ and $\tau_1=\sigma_1$, then $\langle \tau, \sigma \rangle = \langle \tau_0, \sigma_0 \rangle$.
			\item If $\tau_0=\sigma_0$ and $\sigma_1 <_f \tau_1$, then $\langle \tau, \sigma \rangle = (-1)^{\dim(\tau_0)} \langle \tau_1, \sigma_1 \rangle$.
		\end{enumerate}
	\end{obs}
	For each $q \geq 0$, we define $C_q(K)$ as the free module generated by $S_q(K)$ over $\Z$. Let us now define a map $\partial_{q+1}: C_{q+1}(K) \rightarrow C_q(K)$ as,
	
	$$ \partial_{q+1}(\tau) := \sum_{\sigma \in S_q(K)}\langle \tau, \sigma\rangle \sigma \text{, for each } \tau \in S_{q+1}(K), $$
	and extend this map linearly to $C_{q+1}(K)$.
	The next proposition states that $(C_{\#}(K), \partial_{\#})$ is indeed a chain complex.
	\begin{prop}{\cite{Meunier}}
		$\partial_q \circ \partial_{q+1}=0$ for each $q \geq 0$.
	\end{prop}
	This gives us the \emph{simplotopal chain complex} $(C_{\#}(K), \partial_{\#})$.
	The $q^{th}$ \emph{simplotopal homology} of $K$, $$H_q(K):= \frac{\Ker(\partial_q)}{\operatorname{Im}(\partial_{q+1})}.$$
	\section{Discrete Morse theory of simplotopal complex}\label{dmt}
	
	Let $K$ be a simplotopal complex. A \emph{discrete vector field} $\mathcal{V}$ on $K$ is a collection of pairs of simplotopes $(\sigma, \tau)$ in $K$ such that $\sigma$ is a facet of  $\tau$ and every simplotope in $K$ appears in at most one such pair.
	
	\begin{defin}
		A \emph{$\mathcal{V}$-trajectory} is a sequence of simplotopes of $K$ as follows.
		$$ P: \tau_0^{(q)}, \sigma_1^{(q-1)}, \tau_1^{(q)}, \dots , \sigma_k^{(q-1)}, \tau_k^{(q)}, \sigma_{k+1}^{(q-1)},  $$
		where $(\sigma_i^{(q-1)}, \tau_i^{(q)}) \in \mathcal{V}$ for each $i \in [k]$, and $\sigma_i$ is a facet of $\tau_{i-1}$, $(\sigma_i^{(q-1)}, \tau_{i-1}^{(q)}) \notin \mathcal{V}$ for each $i \in [k+1]$, $q \ge 1$.
	\end{defin}
	We say that $P$ \emph{closed} if $\tau_0= \tau_r$ for some $r > 1$. A discrete vector field $\mathcal{V}$ is a \emph{gradient vector field} if there are no closed $\mathcal{V}$-trajectories. We refer to $\tau_0$ as the \emph{initial simplotope} of $P$ and $\sigma_{k+1}$ as the \emph{terminal simplotope} of $P$.
	We will refer to a $\mathcal{V}$-trajectory as simply a trajectory whenever the gradient vector field is clear from the context. The set of all $\mathcal{V}$-trajectories in $K$ with initial simplotope $\tau$ and terminal simplotope $\sigma$ is denoted by $\Gamma(\tau, \sigma)$.
	Next, for a gradient vector field $\V$ on $K$, we define the weight of a $\V$-trajectory.
	\begin{defin}
		Let $\V$ be a gradient vector field on $K$ and $P$ be the following $\mathcal{V}$-trajectory.
		$$P: \tau_0^{(q)}, \sigma_1^{(q-1)}, \tau_1^{(q)}, \dots , \sigma_k^{(q-1)}, \tau_k^{(q)}, \sigma_{k+1}^{(q-1)}.$$
		
		Then the\emph{ weight} of $P$,
		$$ w(P):= \left(\prod_{i=0}^{k-1}(-1)\langle \tau_i, \sigma_{i+1}\rangle \langle \tau_{i+1}, \sigma_{i+1}\rangle\right)\langle \tau_k, \sigma_{k+1}\rangle.$$
	\end{defin}
	
	Let $\V$ be a gradient vector field on $K$. A simplotope which does not appear in $\mathcal{V}$ is termed as \emph{$\V$-critical}, or simply \emph{critical}. We denote the set of all $q$-dimensional $\V$-critical simplotopes as $\Crit^{\mathcal{V}}_q(K)$.
	
	For each $q \geq 0$, the $q^{th}$ \emph{simplotopal Thom--Smale chain group} of $K$ with respect to $\mathcal{V}$, $C^{\mathcal{V}}_q(K)$, is the free module generated by $\Crit^{\V}_q(K)$ over $\Z$. We define a linear map $\partial_{q+1}^{\mathcal{V}}: C^{\mathcal{V}}_{q+1}(K) \rightarrow C^{\mathcal{V}}_q(K)$ as follows. Let $\tau \in \Crit_{q+1}^{\mathcal{V}}(K)$. Then,
	
	$$ \partial_{q+1}^{\mathcal{V}}(\tau):= \sum_{\sigma \in \Crit^{\mathcal{V}}_q(K)}\left(\sum_{P \in \Gamma(\tau, \sigma)}w(P)\right)\sigma.$$

	\begin{prop}
		$(C_{\#}^{\V}, \partial^{\V}_{\#})$ is a chain complex.
	\end{prop}
	Thus, we obtain the chain complex $(C^{\mathcal{V}}_{\#}(K), \partial_{\#}^{\mathcal{V}})$, which is referred to as the \emph{simplotopal Thom--Smale chain complex} on $K$. For each $q \ge 0$, the $q^{th}$ Thom--Smale homology group of $K$ with respect to $\V$, $$H_q^{\mathcal{V}}(K):= \frac{\Ker(\partial_q^{\mathcal{V}})}{\operatorname{Im}(\partial_{q+1}^{\mathcal{V}})}.$$
	The following result states a version of Forman's immensely powerful result (see \cite{Forman1, Forman2}), in the simplotopal setup, which follows from \autoref{koz}.
	\begin{thm}\label{forman}
		The simplotopal Thom--Smale homology group of $K$ (with respect to a gradient vector field $\V$) is isomorphic to the simplotopal homology group of $K$, that is, $H^{\V}_{\#}(K) \cong H_{\#}(K)$.
	\end{thm}
	
	\section{Proof of the main theorem}\label{thm}
	In this section, we prove our main result (\autoref{main}) in two parts. Let $X$ be a simplicial complex with subcomplexes $A_1, \dots ,A_k$ such that $X=\bigcup_{i=1}^{k}A_i$. Let $\N$ be the nerve complex of $X$. We recall that for each $\alpha \in \N$, $A_{\alpha}= \bigcap_{i \in \alpha}A_i$. Suppose $\oW$ is a given gradient vector field on $\Aa$, where $\Aa$ represents a copy of $A_{\alpha}$, for each $\alpha \in \N$ as introduced in Section~\ref{intro}. In the first part, we construct a simplotopal complex $\widetilde{X}$ from $X$ and show that the the simplotopal homology of $\widetilde{X}$ is isomorphic to the simplicial homology of $X$. In the second part, we construct a new chain complex $\Lc$ on $X$ and use the gradient vector fields $\oW$ on $\Aa$ for each $\alpha \in\N$ to  show that the homology of $\Lc$ is isomorphic to the simplotopal homology of $\widetilde{X}$.
	
	\subsection{Construction of a simplotopal complex $\widetilde{X}$}\label{constr}
	For each $\alpha \in \mathcal{N}(X)$, we define $ \widetilde{A}_{\alpha}:= \{(\sigma, \alpha) \mid \sigma \in A_{\alpha}\}.$
	Now, we define a simplotopal complex $\widetilde{X}$ as,
	$$ \widetilde{X}:= \bigcup_{\alpha \in \mathcal{N}(X)}\widetilde{A}_{\alpha}.$$
	We can verify that $\wX$ is indeed a simplotopal complex. We note here that $\wA \cap \widetilde{A}_\beta = \emptyset$ for $\alpha \ne \beta$, $\alpha, \beta \in \N$. Next, we construct a gradient vector field $\mathcal{V}$ on $\widetilde{X}$.
	
	\subsection{Construction of a gradient vector field $\mathcal{V}$ on $\widetilde{X}$}
	
	We partition $\wX$ into some pairwise disjoint sets and define a collection of pairs on each of these sets. For each $\sigma \in X$, first let us define $\Lambda_{\sigma}:= \{i \in [k]\mid \sigma \in A_i\}$. Thus, $\Lambda_{\sigma} \in \mathcal{N}(X)$ for each $\sigma \in X$. Next, we define, for each $\sigma \in X$, $U_{\sigma}:= \{(\sigma, \alpha) \mid \alpha \subseteq \Lambda_{\sigma}\}$. We can observe that $U_{\sigma}\cap U_{\tau} = \emptyset$ for $\sigma \ne \tau$, $\sigma, \tau \in X$. We know that each $(\sigma, \alpha) \in U_\sigma$ is also an element of $\widetilde{A}_\alpha$, which follows from the definition of $\widetilde{A}_\alpha$. Now, we further note that if $(\sigma, \alpha) \in \wA$ for some $\alpha \in \N$, then $\sigma \in A_i$ for each $i \in \alpha$, which implies that $\alpha \subseteq \Lambda_\sigma$. Consequently, $(\sigma, \alpha) \in U_\sigma$. Therefore, it follows that,
	$$ \bigcup_{\sigma \in X}U_{\sigma} = \bigcup_{\alpha \in \mathcal{N}(X)}\widetilde{A}_{\alpha} = \wX.$$
	
	We define a collection of pairs $\overline{\mathcal{V}}_{\sigma}$ on $U_{\sigma}$ for each $\sigma \in X$ as follows. Suppose, $\Lambda_{\sigma}= \{i_0, \dots , i_r\}$, where $1 \leq i_0 < i_1 < \dots <i_r \leq k$. Let $(\sigma, \alpha) \in U_{\sigma}$. If $i_0 \notin \alpha$, then $\left((\sigma, \alpha), (\sigma, \alpha \sqcup \{i_0\})\right) \in \overline{\V}_{\sigma}$. Hence, the unpaired simplotope with respect to $\V_\sigma$ is precisely $(\sigma, \{i_0\})$. \autoref{gvf} illustrates $\overline{\V}_{\sigma}$ on one such segment $U_{\sigma}$, for $\sigma=\{a,b\}$ and $\Lambda_{\sigma}=\{1,2,3\}$. 
	\begin{figure}[h]
		\centering
		\begin{tikzpicture}[
			vertex/.style={draw, circle, fill=black, inner sep=0.05pt, minimum size=0.7mm},
			baseline]
			
			\fill[gray!60] (6,0) -- (3,3) -- (3,4) -- (6,1) -- cycle;
			\fill[gray!30] (3,3) -- (3,4) -- (0,1) -- (0,0) -- cycle;
			\fill[gray!10] (0,0) -- (0,1) -- (6,1) -- (6,0)-- cycle;
			\draw[pattern=dots] (0,1) -- (3,4) -- (6,1) -- cycle;
			\draw[pattern=dots] (0,1) -- (0,0) -- (6,0) -- (6,1) -- cycle;
			\node[vertex, label=left:{$(\{b\},\{1\})$}] (v1) at (0,0) {};
			\node [vertex, label=right:{$(\{b\},\{2\})$}] (v2) at (6,0) {};
			\node [vertex, label={[fill=white, inner sep=1pt, yshift=-2pt]below :{\small$(\{b\},\{3\})$}}] (v3) at (3,3) {};
			\node [vertex, label=left:{$(\{a\},\{1\})$}] (v4) at (0,1) {};
			\node [vertex, label=above:{$(\{a\},\{3\})$}] (v5) at (3,4) {};
			\node [vertex, label=right:{$(\{a\},\{2\})$}] (v6) at (6,1) {};
			\node[vertex] (v7) at (3,3.5) {};
			\node (v8) at (2.2,2.7){};
			\node[vertex] (v9) at (4.5,2.1){};
			\node (v10) at (3.5, 1.5){};
			\node[vertex] (v11) at (6,0.5){};
			\node (v12) at (4.5, 0.5){};
			
			\draw (v1) -- (v2);
			\draw[dotted] (v1) -- (v3) -- (v2);
			\draw (v1) -- (v4);
			\draw (v2) -- (v6);
			\draw[dotted] (v3) -- (v5);
			\draw (v4) -- (v6);
			\draw (v4) -- (v5) -- (v6);
			\draw[->,thick] (v7) -- (v8)node[midway, above, fill=white, inner sep=1pt] {$p$};
			\draw[->,thick] (v9) -- (v10) node[midway, above, fill=white, inner sep=1pt] {$q$};
			\draw[->,thick] (v11) -- (v12) node[midway, above, fill=white, inner sep=1pt] {$r$};
			
		\end{tikzpicture}
		\caption{Illustration of $\overline{\V}_\sigma$ on $U_\sigma$, where $\sigma=\{a,b\}$, $\Lambda_\sigma=\{1,2,3\}$. The arrow ($p$) represents the pairing between the edge $(\{a,b\},\{3\})$ and the rectangle $(\{a,b\},\{1,3\})$, the arrow ($r$) represents the pairing between the edge $(\{a,b\},\{2\})$ and the rectangle $(\{a,b\},\{1,2\})$ while the arrow ($q$) represents the pairing between the rectangle $(\{a,b\},\{2,3\})$ and the prism $(\{a,b\},\{1,2,3\})$.}\label{gvf}
	\end{figure}
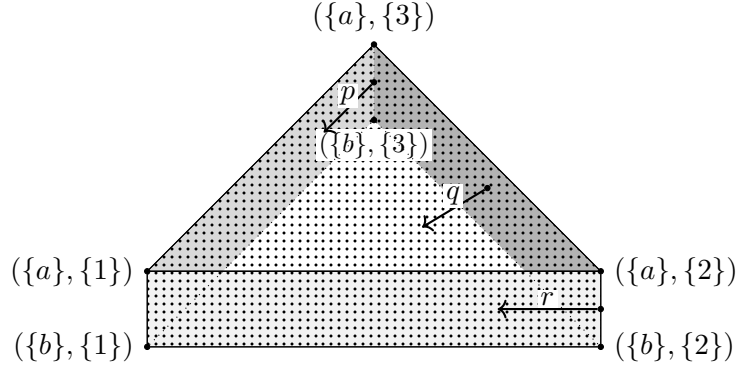
	
	Now, $\V:= \bigcup_{\sigma \in X} \overline{\V}_{\sigma}$. The well-definedness of $\V$ follows from the fact that $U_{\sigma}\cap U_{\tau} = \emptyset$ for $\sigma \ne \tau$, $\sigma, \tau \in X$. For each $\sigma \in X$, let $i_{\sigma}= \min\{i \mid i \in \Lambda_\sigma\}$. Therefore the simplotopes in $\widetilde{X}$ which are unpaired with respect to $\V$, are given by $\{(\sigma, \{i_{\sigma}\}) \mid \sigma \in X\}$.
	
	The next observation follows from the definition of $\Lambda_\sigma$ for $\sigma \in X$.
	\begin{obs}\label{V}
		Let $\sigma \subseteq \tau \in X$. Then $\Lambda_{\tau} \subseteq \Lambda_{\sigma}$. Consequently, $i_{\sigma} \leq i_{\tau}$.
	\end{obs}
	
	The following proposition characterises the $\mathcal{V}$-trajectories in $\widetilde{X}$.
	
	\begin{prop}\label{vtraj}
		Let $P: \tau_0', \sigma_1', \tau_1', \dots , \sigma_k', \tau_k', \sigma_{k+1}'$ be a $\V$-trajectory in $\widetilde{X}$, where $\sigma_j'=(\sigma_j, \alpha_j)$, $\tau_j'=(\tau_j, \beta_j)$, $\sigma_j, \tau_j \in X$, $\alpha_j, \beta_j \in \N$. Then $P$ is of the following form.
		\begin{enumerate}[label=(\roman*)]
				\item If $\tau_0=\sigma_1$ and $\ao <_f \beta_0$, then $i_{\tau_0} \notin \beta_0$.\label{cond1}
				\item For each $j \in [k]$, $\tau_j=\sigma_j$, $i_{\sigma_j} \notin \alpha_j$ and $\beta_j = \alpha_j \sqcup \{i_{\sigma_j}\}$. \label{cond2}
				\item For each $j \in [k-1]$, $\sigma_{j+1} <_f \tau_j$ and $\alpha_{j+1}= \beta_j$.\label{cond3}
				\item Either $\sigma_{k+1}=\tau_k$ or $\sigma_{k+1} <_f \tau_k$.\label{cond4}
			\end{enumerate}
		\end{prop}
		\begin{proof}
			\begin{enumerate}[label=(\roman*)]
				\item Since, $\sigma_1'$ is a facet of $\tau_0'$, therefore either we have $\sigma_1 <_f \tau_0$, $\beta_0=\ao$ or we have  $\tau_0=\sigma_1$, $\ao <_f \beta_0$. In the case, $\tau_0=\sigma_1$ and $\ao <_f \beta_0$, we observe that if $i_{\tau_0} \in \beta_0$, then $\alpha_1=\beta_0 \setminus \{x\}$, $x \ne i_{\tau_0}$, which  follows from the fact that $(\tau_0', \sigma_1') \notin \V$. Thus, $i_{\tau_0}=i_{\sigma_1} \in \alpha_1$, which means $\sigma_1'$ is paired with a lower dimensional simplotope $(\sigma_1, \ao \setminus \{i_{\sigma_1}\})$. This is not possible as $\sigma_1'$ is paired with a higher dimensional simplotope. 
				
				\item Since $(\sigma_j', \tau_j') \in \V$ for each $j \in [k]$, therefore, assertion \ref{cond2} follows from the construction of $\V$. 
				
				\item  Now, for each $j \in [k-1]$, since $\sigma_{j+1}'$ is a facet of $\tau_j'$, therefore either we have $\sigma_{j+1}=\tau_j$, $\alpha
				_{j+1} <_f \beta_j$, or we have $\sigma_{j+1} <_f \tau_j$, $\alpha_{j+1}=\beta_j$. We first consider the case when $\sigma_{j+1}=\tau_j$ and $\alpha
				_{j+1} <_f \beta_j$. Now, $\beta_j=\alpha_j \sqcup \{i_{\sigma_j}\}$ (from \ref{cond2}). As a result, since, $\alpha_{j+1} \ne \alpha_j$, therefore $\alpha_{j+1}=\beta_j \setminus \{x\}$, $x \neq i_{\sigma_j}$. Since, $\sigma_j=\tau_j$ (also from \ref{cond2}), therefore, this means $i_{\sigma_{j+1}}=i_{\tau_j}=i_{\sigma_j} \in \alpha_{j+1}$. Consequently, from the construction of $\V$, $\sigma_j'$ is paired with a lower dimensional simplotope, $(\sigma_{j+1}, \alpha_{j+1} \setminus \{i_{\sigma_{j+1}}\})$, which is a contradiction. Therefore,  $\sigma_{j+1} <_f \tau_j$ and $\alpha_{j+1}=\beta_j$ for each $j \in [k-1]$. 
				\item  Assertion \ref{cond4} follows from the fact that $\sigma_{k+1}'$ is a a facet of $\tau_{k}'$.
			\end{enumerate}
		\end{proof}
		\begin{lem}
			$\V$ is a gradient vector field.
		\end{lem}
		
		\begin{proof}
			Let $P: \tau_0', \sigma_1', \tau_1', \dots , \sigma_k', \tau_k', \sigma_{k+1}'$ be a $\V$-trajectory.
			It follows from \autoref{vtraj}\ref{cond3} that,
			$$ \dim(\tau_0) \geq \dim(\sigma_1) > \dim(\sigma_2) > \dots > \dim(\sigma_r), $$
			for any $r > 1$. Again from \autoref{vtraj}\ref{cond2}, we have $\dim(\sigma_r)=\dim(\tau_r)$ for any $r >1$.
			Therefore $\tau_0 \neq \tau_r$ for each $r > 1$. Consequently, $\tau_0' \ne \tau_r'$, for each $r>1$. Thus, $P$ is not closed.
		\end{proof}
		\noindent Therefore, for each $q \geq 0$, the set of all $q$-dimensional $\V$-critical simplotopes in $\wX$ is given by,
		$$ \CrV= \{(\sigma, \{i_{\sigma}\}) \mid \sigma \in S_q(X)\},$$
		where $S_q(X)$ is the set of all oriented $q$-simplices in $X$, as defined in Section~\ref{prelim}.
		
		In the next subsection, we show that the simplotopal Thom--Smale complex of $\widetilde{X}$ with respect to $\mathcal{V}$ is isomorphic to the simplicial chain complex of $X$.
		
		\subsection{Isomorphism between the simplotopal homology of $\wX$ and the simplicial homology of $X$}
		\begin{prop}\label{iso1}
			The following isomorphism holds.
			$$ (C_{\#}(X), \partial_{\#}) \cong (C_{\#}^{\V}(\wX), \partial^{\V}_{\#}).$$
			Hence, $H_{\#}(X) \cong H^{\V}_{\#}(\wX)$.
		\end{prop}
		
		\begin{proof}
			First we construct an isomorphism $h_{\#}:C_{\#}^{\V}(\wX) \rightarrow C_{\#}(X) $ and then we show that the following diagram commutes, that is, $h_{\#}$ is a chain map.
			
			\[\begin{tikzcd}
				\cdots & {C_{q+1}(X)} && {C_{q}(X)} & \cdots \\
				\\
				\cdots & {C_{q+1}^{\V}(\wX)} && {C_{q}^\V(\wX)} & \cdots
				\arrow[from=1-1, to=1-2]
				\arrow["{\partial_{q+1}}", from=1-2, to=1-4]
				\arrow["{h_{q+1}}"', from=3-2, to=1-2]
				\arrow[from=1-4, to=1-5]
				\arrow["{h_{q}}"', from=3-4, to=1-4]
				\arrow[from=3-1, to=3-2]
				\arrow["{\partial^\V_{{q+1}}}"', from=3-2, to=3-4]
				\arrow[from=3-4, to=3-5]
				\arrow[draw=none, from=1-2, to=3-4, "{\text{\Huge$\circlearrowleft$}}" description]
			\end{tikzcd}\]
			For each $q \ge 0$, let us define $h_q: C_{q}^{\V}(\wX) \rightarrow C_{q}(X)$. It suffices to define $h_q$ on the set of generators. Let $(\sigma, \{i_{\sigma}\}) \in \CrV$. Then we define, $h_q((\sigma, \{i_{\sigma}\})):= \sigma$. It follows from the definition that $h_q$ is an isomorphism for each $q \geq 0$.
			
			Next, we show that $h_{q} \circ \partial^\V_{{q+1}}= \partial_{{q+1}} \circ h_{q+1}$. We note that it is sufficient to prove that this holds on $\Crit_{q+1}^{\V}(\wX)$.
			
			Let $\tau'=(\tau, \{i_{\tau}\}) \in \Crit_{q+1}^{\V}(\wX)$. Then,
			\begin{equation*}
				\begin{aligned}
					\partial_{{q+1}}( h_{q+1}(\tau')) &= \partial_{q+1}(\tau) \\
					&= \sum_{\substack{\sigma: \sigma <_f \tau}} \langle \tau, \sigma\rangle \sigma.
				\end{aligned}
			\end{equation*}
			Now we calculate $\partial^\V_{{q+1}}(\tau')$.
			\begin{equation*}
				\begin{aligned}
					\partial^\V_{q+1}(\tau')= \sum_{(\sigma, \{i_{\sigma}\}) \in \CrV}\left(\sum_{P \in \Gamma(\tau', (\sigma, \{i_{\sigma}\}))}w(P)\right) (\sigma,\{ i_{\sigma}\}).
				\end{aligned}
			\end{equation*}
			Next, we show that there are precisely two kinds of trajectories in $\Gamma(\tau', (\sigma, \{i_{\sigma}\}))$, depending on the facets of $\tau'$. Let $(\sigma', \{i_{\tau}\})$ be a facet of $\tau'$, that is, $\sigma'<_f \tau$ (we note that this is the only kind of facet $\tau'$ can have). For each facet $(\sigma', \{i_{\tau}\})$ of $\tau'$, we now determine the type of trajectories of the form $(\tau'=)~(\tau, \{i_\tau\}), (\sigma', \{i_\tau\}), \dots, (\sigma, \{i_{\sigma}\})$, where $(\sigma, \{i_{\sigma}\})\in \CrV$. 
			
			From \autoref{V}, we have the following two cases.
			
			\textbf{Case I:} Let $i_{\sigma'}= i_{\tau}$.
			
			In this case, $(\sigma',\{i_{\tau}\}) \in \CrV$ from the construction of $\V$ and hence, we conclude that the only possible trajectory of the form $(\tau, \{i_\tau\}), (\sigma', \{i_\tau\}), \dots , (\sigma, \{i_{\sigma}\})$, where $(\sigma, \{i_{\sigma}\}) \in \CrV$ is $P: (\tau, \{i_{\tau}\}), (\sigma',\{i_{\tau}\})$. Thus, here the terminal (critical) simplotope of $P$ is of the form $(\sigma, \{i_{\sigma}\})$, where $\sigma <_f \tau$ and $i_\sigma = i_\tau$. Now, $$w(P)= \langle (\tau, \{i_{\tau}\}), (\sigma', \{i_{\tau}\}) \rangle= \langle \tau, \sigma' \rangle,$$ which follows from \autoref{inc} (Section~\ref{prelim}).
			
			\textbf{Case II:} Let $i_{\sigma'} < i_{\tau}$.
			
			Let $P$ be a trajectory of the form, $$P:((\tau, \{i_{\tau}\})=)\tau_0, ((\sigma', \{i_{\tau}\})= )\sigma_1, \tau_1, \sigma_2, \tau_2,  \dots , (\sigma, \{i_{\sigma}\}),$$ where $(\sigma, \{i_{\sigma}\}) \in \CrV$. We now show that for a fixed facet $(\sigma', \{i_{\tau}\})$ of $(\tau, \{i_{\tau}\})$, there exists a unique trajectory of this form. Since $i_{\sigma'}< i_\tau$, therefore from the construction of $\V$, it follows that $\tau_1=(\sigma', \{i_{\sigma'}, i_{\tau}\})$. Now, since $\sigma_2$ is a facet of $\tau_1$, therefore  either $\sigma_2=(\sigma'',\{i_{\sigma'}, i_{\tau}\})$, where $\sigma'' <_f \sigma'$ or $\sigma_2=(\sigma', \{i_{\sigma'}\})$. Suppose, $\sigma_2=(\sigma'',\{i_{\sigma'}, i_{\tau}\})$, where $\sigma'' <_f \sigma'$. In this case, we infer from \autoref{vtraj} (assertions \ref{cond3} and \ref{cond4}) that $\dim(\sigma) \le \dim(\sigma'')$. Therefore, it follows that, $$ \dim(\sigma) \leq \dim(\sigma'') < \dim(\sigma') < \dim(\tau),$$
			since $\sigma''<_f \sigma' <_f \tau$. This implies, $\dim(\sigma)< (\dim(\tau)-1)= q$, which is not possible (since $(\sigma, \{i_{\sigma}\}) \in \CrW$). So, $\sigma_2$ must be $(\sigma', \{i_{\sigma'}\})$ and hence critical. It follows that such a trajectory $P$ is unique and moreover, the terminal (critical) simplotope of $P$ is of the form $(\sigma, \{i_{\sigma}\})$, where $\sigma <_f \tau$ and $i_\sigma < i_\tau$.
			Now,
			\begin{equation*}
				\begin{aligned}
					w(P)&= -\langle \tau_0, \sigma_1\rangle \langle \tau_1, \sigma_1\rangle \langle \tau_1, \sigma_2 \rangle\\ &= -\langle (\tau, \{i_{\tau}\}), (\sigma', \{i_{\tau}\})  \rangle \langle (\sigma', \{i_{\sigma'}, i_{\tau}\}), (\sigma', \{i_{\tau}\})\rangle\langle (\sigma', \{i_{\sigma'}, i_{\tau}\}), (\sigma', \{i_{\sigma'}\})\rangle \\
					&=- \langle \tau, \sigma'\rangle(-1)^{\dim(\sigma') +0}  (-1)^{\dim(\sigma') +1} \\ &= \langle \tau, \sigma'\rangle.
				\end{aligned}
			\end{equation*}
			Therefore,
			\begin{equation*}
				\begin{aligned}
					\partial^\V_{q+1}((\tau, \{i_{\tau}\}))&= \sum_{\substack{(\sigma, \{i_{\sigma}\}) \in \CrV\\ \sigma <_f \tau \text{ and }\\i_{\sigma}=i_{\tau}}}\langle \tau, \sigma\rangle (\sigma, \{i_{\sigma}\}) + \sum_{\substack{(\sigma, \{i_{\sigma}\})\in \CrV \\ \sigma <_f \tau \text{ and }\\  i_{\sigma}<i_{\tau}}}\langle \tau, \sigma\rangle (\sigma, \{i_{\sigma}\}) \\
					&= \sum_{\substack{(\sigma, \{i_{\sigma}\}): \sigma <_f \tau}}\langle \tau, \sigma\rangle (\sigma, \{i_{\sigma}\})\\
					\implies h_q(\partial^\V_{q+1}((\tau, \{i_{\tau}\})))&= \sum_{\substack{(\sigma, \{i_{\sigma}\}): \sigma <_f \tau}}\langle \tau, \sigma\rangle h_q((\sigma, \{i_{\sigma}\}))\\
					&=	\sum_{\substack{\sigma: \sigma <_f \tau}}\langle \tau, \sigma\rangle\sigma\\
					&= \partial_{{q+1}}(h_{q+1}((\tau, \{i_{\tau}\}))).
				\end{aligned}
			\end{equation*}
			This proves that $h_{\#}$ is a chain isomorphism and hence the result follows.
		\end{proof}
		
		\noindent We obtain the following proposition as a direct consequence of \autoref{forman} and the above result (\autoref{iso1}).
		
		\begin{prop}\label{first}
			The simplicial homology of $X$ is isomorphic to the simplotopal homology of $\wX$, that is,
			$$ H_{\#}(X) \cong H_{\#}(\wX).$$
		\end{prop}
		
		In the next subsection, we construct a new gradient vector field $\W$ on $\wX$ and show that the Thom--Smale chain complex of $\wX$ with respect to $\W$ is isomorphic to the chain complex $\Lc$. This will establish an isomorphism between $H_{\#}(\wX)$ and $H^{\Lq}_{\#}(X)$.
		
		\subsection{Construction of $\mathcal{W}$ on $\wX$}
		We recall that $\Wo_\alpha$ is a given gradient vector field on $\Aa$ for each $\alpha \in \N$. We also recall that we defined $\widetilde{A}_{\alpha}:= \{(\sigma, \alpha) \mid \sigma \in A_{\alpha}\}$, for each $\alpha \in \N$ in Subsection~\ref{constr}. Now, we define a collection of pairs on each $\wA$ as $$\Va := \bigcup_{q=1}^{\dim(A_{\alpha}) - 1}\{((\sigma, \alpha),(\tau, \alpha)) \mid (\sigma^{(q-1)}_\alpha, \tau^{(q)}_\alpha) \in \oW \}.$$ Next, we define a discrete vector field on $\widetilde{X}$ as, $\W:= \bigcup\limits_{\alpha \in \N}\Va$. The well-definedness of $\W$ follows from the fact that $\wA \cap \widetilde{A}_{\beta} = \emptyset$ for $\alpha \ne \beta$, $\alpha$, $\beta \in \N$.
		
		Suppose, $\alpha=\{1,2\}$ and $A_\alpha=\{\{a\}, \{b\}, \{c\}, \{a,b\}, \{b,c\}, \{a,b,c\}\}$. 
		
		Let $\oW= \{(\{b\}_{\{1,2\}}, \{a,b\}_{\{1,2\}}), (\{c\}_{\{1,2\}}, \{a,c\}_{\{1,2\}}), (\{b,c\}_{\{1,2\}}, \{a,b,c\}_{\{1,2\}})\}$. Then \autoref{gvf2} illustrates $\Va$ on $\widetilde{A}_{\alpha}$.	
		\begin{figure}[h]
			\centering
			\begin{tikzpicture}[
				vertex/.style={draw, circle, fill=black, inner sep=0.05pt, minimum size=0.7mm},
				baseline]
				\fill[gray!60] (6,0) -- (3,3) -- (3,4) -- (6,1) -- cycle;
				\fill[gray!30] (3,3) -- (3,4) -- (0,1) -- (0,0) -- cycle;
				\fill[gray!10] (0,0) -- (0,1) -- (6,1) -- (6,0)-- cycle;
				\draw[pattern=dots] (0,1) -- (3,4) -- (6,1) -- cycle;
				\draw[pattern=dots] (0,1) -- (0,0) -- (6,0) -- (6,1) -- cycle;
				\node[vertex, label=left:{$(\{a\},\{1\})$}] (v1) at (0,0) {};
				\node [vertex, label=right:{$(\{b\},\{1\})$}] (v2) at (6,0) {};
				\node [vertex, label={[fill=white, inner sep=1pt,yshift=-2pt]below :{\small$(\{c\},\{1\})$}}] (v3) at (3,3) {};
				\node [vertex, label=left:{$(\{a\},\{2\})$}] (v4) at (0,1) {};
				\node [vertex, label=above:{$(\{c\},\{2\})$}] (v5) at (3,4) {};
				\node [vertex, label=right:{$(\{b\},\{2\})$}] (v6) at (6,1) {};
				\node[vertex] (v7) at (3,3.5) {};
				\node (v8) at (2.2,2.7){};
				\node[vertex] (v9) at (4.5,2.1){};
				\node (v10) at (3.5, 1.5){};
				\node[vertex] (v11) at (6,0.5){};
				\node (v12) at (4.5, 0.5){};
				
				\draw (v1) -- (v2);
				\draw[dotted] (v1) -- (v3) -- (v2);
				\draw (v1) -- (v4);
				\draw (v2) -- (v6);
				\draw[dotted] (v3) -- (v5);
				\draw (v4) -- (v6);
				\draw (v4) -- (v5) -- (v6);
				\draw[->,thick] (v7) -- (v8)node[midway, above, fill=white, inner sep=1pt]{$p$};
				\draw[->,thick] (v9) -- (v10)node[midway, above, fill=white, inner sep=1pt]{$q$};
				\draw[->,thick] (v11) -- (v12)node[midway, above, fill=white, inner sep=1pt]{$r$};
				
			\end{tikzpicture}
			\caption{Illustration of $\Va$ on $\widetilde{A}_\alpha$, where $\alpha=\{1,2\}$. The arrow ($p$) represent the pairing between the edge  $(\{c\},\{1,2\})$ and the rectangle $(\{a,c\},\{1,2\})$, the arrow ($q$) represents the pairing between the rectangle $(\{b,c\},\{1,2\})$ and the prism $(\{a,b,c\},\{1,2\})$ while the arrow($r$) represents the pairing between the edge $(\{b\},\{1,2\})$ and the rectangle $(\{a,b\},\{1,2\})$. }\label{gvf2}
		\end{figure}
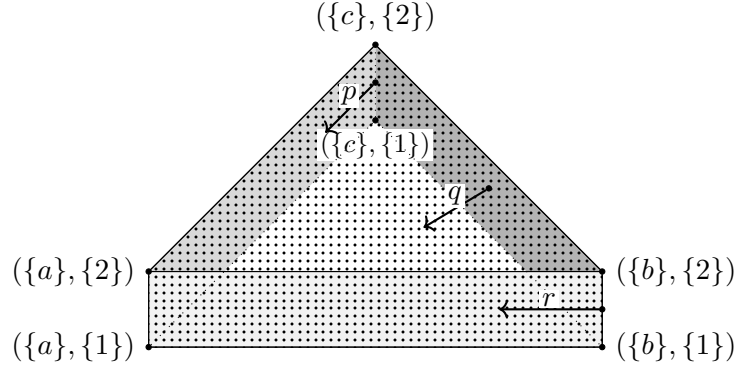
		\subsubsection{Types of $\W$ trajectories in $\wX$}
			\begin{prop}\label{traj}
				Let
				$$ P: \tau_0', \sigma_1', \tau_1', \dots , \sigma_k', \tau_k', \sigma_{k+1}'$$ be a $\W$-trajectory, where $\tau_i'=(\tau_i, \beta_i)$, for each $i \in \{0, \dots ,k\}$ , $\sigma_i'=(\sigma_i, \alpha_i)$ for each $i \in [k+1]$. Then, $P$ is either of the following types.
				
				\begin{enumerate}[label=(\alph*)]
					\item $(\sigma_j', \tau_j') \in \V_{\beta_0}$ for each $j \in [k]$, $\tau_0', \sigma_{k+1}' \in \widetilde{A}_{\beta_0}$. This trajectory is thus of the following form,
					
					$$ (\tau_0^{(p)}, \beta_0), (\sigma_1^{(p-1)}, \beta_0), (\tau_1^{(p)}, \beta_0), \dots, (\sigma^{(p-1)}_{k+1}, \beta_0),$$
					where $((\sigma_j)_{\beta_0},(\tau_j)_{\beta_0}) \in \overline{\W}_{\beta_0}$ for each $j \in [k]$, $((\sigma_j)_{\beta_0},(\tau_{j-1})_{\beta_0}) \notin \overline{\W}_{\beta_0}$, $(\sigma_j)_{\beta_0} <_f(\tau_{j-1})_{\beta_0}$	for each $j \in [k+1]$.	
					In other words, this trajectory is entirely in $\widetilde{A}_{\beta_0}$.
					
					We call this type of trajectory a pure trajectory in $\widetilde{A}_{\bz}$.	\label{type1}	
					
					\item There exists an increasing sequence $0 < i_1 < i_2 < \dots < i_t \leq k+1$, such that
					
					\begin{equation*}
						\begin{aligned}
							& \tau_j' \in \widetilde{A}_{\beta_0} \text{for each } 0 \leq j < i_1,\\
							&\sigma_j' \in \widetilde{A}_{\alpha_{1}} \text{ i.e., } (\sigma_j', \tau_j') \in \V_{\alpha_{1}}  \text{ for each } i_1 \leq j < i_2,\\
							&\vdots\\
							& \sigma_j' \in \widetilde{A}_{\alpha_{t}} \text{ for each } i_t\leq j \leq k+1,\\
						\end{aligned}
					\end{equation*}
					where $\alpha_t <_f \alpha_{t-1} <_f \dots <_f \alpha_1 <_f \beta_0$ and $1 \le t \le \dim(\beta_0)$. This trajectory is thus of the form,
					$$ \begin{aligned}
						&\underbrace{(\tau_0^{(p)}, \beta_0), (\sigma_1^{(p-1)}, \beta_0), \dots , (\tau_{i_1-1}^{(p)}, \beta_0)}_{\widetilde{A}_{\beta_0}}, \underbrace{(\sigma_{i_1}^{(p)}, \alpha_{1}), (\tau_{i_1}^{(p+1)}, \alpha_{1}), \dots , (\tau_{i_2-1}^{(p+1)}, \alpha_{1})}_{\widetilde{A}_{\alpha_{1}}}, \dots,\\ & \underbrace{(\sigma_{i_t}^{(p+t-1)}, \alpha_{t}), (\tau_{i_t}^{(p+t)}, \alpha_{t}), \dots (\sigma_{k+1}^{(p+t-1)}, \alpha_{t})}_{\widetilde{A}_{\alpha_{t}}}.
					\end{aligned}$$
					Intuitively speaking, this trajectory travels through $\widetilde{A}_{\beta_0}$ for some time following $\V_{\bz}$ and then enters $\widetilde{A}_{\alpha_{1}}$, travels through $\widetilde{A}_{\ao}$ for some time following $\V_{\ao}$, before entering into $\widetilde{A}_{\alpha_{2}}$ and so on. More precisely, the following conditions hold.
					\begin{enumerate}[label=(\roman*)]
						\item $(\sigma_j)_{\beta_0} <_f (\tau_{j-1})_{\beta_0}$ for each $j \in [i_1-1]$,
						\item  $(\sigma_j)_{\alpha_r}<_f(\tau_{j-1})_{\alpha_r})$ for each $i_r+1 \le j < i_{r+1}$ where $r \in [t-1]$,
						\item $(\sigma_j)_{\alpha_t} <_f (\tau_{j-1})_{\alpha_t})$ for each $i_t+1 \le j \le (k+1)$,   
						\item $((\sigma_j)_{\beta_0}, (\tau_j)_{\beta_0}) \in \overline{\W}_{\beta_0}$, for each $j \in [i_1-1]$, and $((\sigma_j)_{\beta_0}, (\tau_{j-1})_{\beta_0})\notin \Wo_{\bz}$ for each $j \in [i_1-1]$.  
						\item $((\sigma_j)_{\alpha_{r}}, (\tau_j)_{\alpha_{r}}) \in \overline{\W}_{\alpha_{i_r}}$ for each $i_r \le j < i_{r+1}$, and $((\sigma_j)_{\alpha_r}, (\tau_{j-1})_{\alpha_r})\notin \Wo_{\alpha_r}$ for each $i_r+1 \le j < i_{r+1}$, where $r \in [t-1]$.   
						\item $((\sigma_j)_{\alpha_{t}}, (\tau_j)_{\alpha_{t}}) \in \overline{\W}_{\alpha_{t}}$ for each $i_t \le j < (k+1)$ and $((\sigma_j)_{\alpha_t}, (\tau_{j-1})_{\alpha_t})\notin \Wo_{\alpha_t}$ for each $i_t+1 \le j \le k+1$,
						\item $\sigma_{i_r}= \tau_{i_r-1}$ for each $r \in [t]$, and
						\item $\alpha_t <_f \alpha_{t-1} <_f \dots <_f \alpha_1 <_f \beta_0$, where $1 \le t \le \dim(\beta_0)$.
					\end{enumerate}
					Further, $\dim(\sigma_{k+1})=\dim(\tau_0) +t-1$.
					
					We call this type of trajectory a mixed trajectory.\label{type2}
				\end{enumerate}	
			\end{prop}	
			\begin{proof}
				If each simplex of $P$ is in $\widetilde{A}_{\beta_0}$, then from the construction of $\W$, we must have $\sigma_j'=(\sigma_j^{(p-1)}, \beta_0)$ and $\tau_j'=(\tau_j^{(p)}, \beta_0)$, where $((\sigma_j)_{\bz}, (\tau_j)_{\bz}) \in \overline{\W}_{\bz}$ for each $j \in [k]$. Also, since $\sigma_{j}'$ is a facet of $\tau_{j-1}'$ for each $j \in [k+1]$, therefore, $\sigma_j'=(\sigma_j, \beta_0)$ and $\tau_{j-1}'=(\tau_{j-1}, \beta_0)$, where $(\sigma_{j})_{\bz} <_f (\tau_{j-1})_{\bz}$ for each $j \in [k+1]$. Again, from the construction of $\W$, since $(\sigma_{j}',\tau_{j-1}')\notin \W$ for each $j \in [k+1]$, therefore $((\sigma_{j})_{\bz},(\tau_{j-1})_{\bz}) \notin \overline{\W}_{\bz}$, for each $j \in [k+1]$. Hence, in this case, $P$ is of the first type, more precisely, of the form stated in \ref{type1}. 
				
				Otherwise, there exists $i_1 >1$ such that $\tau_{i_1-1}'= (\tau_{i_1-1}^{(p)}, \bz) \in \widetilde{A}_{\beta_0}$ and $\sigma_{i_1}' \in \widetilde{A}_{\alpha_{1}}$ for some $\alpha_1 \in \N$, $\alpha_{1} \ne \beta_0$. Since $\sigma_{i_1}'$ is a facet of $\tau_{i_1 -1}'$ and $\alpha_1 \ne \beta_0$, therefore $\sigma_{i_1}= \tau_{i_1 -1}$, $\alpha_{1} <_f \beta_0$. Now, $\sigma_j'$ is a facet of $\tau_{j-1}'$ for each $\sigma_j', \tau_{j-1}' \in \widetilde{A}_{\ao}$. So it follows that $\sigma_j'=(\sigma_j, \ao)$, $\tau_j'=(\tau_j, \ao)$, where $(\sigma_j)_{\ao} <_f (\tau_{j-1})_{\ao}$ for each $\sigma_j', \tau_{j-1}' \in \widetilde{A}_{\ao}$. Also from the construction of $\W$, we observe that $((\sigma_j)_{\ao}, (\tau_j)_{\ao}) \in \Wo_{\ao}$, for each $\sigma_j', \tau_j' \in \widetilde{A}_{\ao}$. Since, $(\sigma_j', \tau_{j-1}') \notin \W$ for each $\sigma_j', \tau_{j-1}' \in \widetilde{A}_{\ao}$, therefore again from the construction of $\W$, it follows that $((\sigma_j)_{\ao}, (\tau_{j-1})_{\ao}) \notin \Wo_{\ao}$ for each $\sigma_j', \tau_{j-1}' \in \widetilde{A}_{\ao}$. Now, if $\sigma_j' \in \widetilde{A}_{\alpha_{1}}$ for each $i_1 \le j \le (k+1)$, then $t=1$ and we obtain a trajectory of the second type. 
				
				Otherwise, again there exists $i_2 > i_1$, such that $\tau_{i_2-1}'=(\tau_{i_2-1}^{(p+1)}, \ao) \in \widetilde{A}_{\ao}$ and $\sigma_{i_2}' \in \widetilde{A}_{\alpha_2}$ for some $\alpha_2 \in \N$ where $\alpha_2 \ne \alpha_1$. Since $\sigma_{i_2}'$ is a facet of $\tau_{i_2 -1}'$ and $\alpha_2 \ne \ao$, therefore $\sigma_{i_2}= \tau_{i_2 -1}$ and $\alpha_{2} <_f \alpha_1$. As before, $(\sigma_j)_{\alpha_2} <_f (\tau_{j-1})_{\alpha_2}$ for each $\sigma_j', \tau_{j-1}' \in \widetilde{A}_{\alpha_2}$. Again from the construction of $\W$, $((\sigma_j)^{(p+1)}_{\alpha_2}, (\tau_j)^{(p+2)}_{\alpha_2}) \in \Wo_{\alpha_2}$, for each $\sigma_j', \tau_j' \in \widetilde{A}_{\alpha_2}$ and $((\sigma_j)_{\alpha_2}, (\tau_{j-1})_{\alpha_2}) \notin \Wo_{\alpha_2}$ for each $\sigma_j', \tau_{j-1}' \in \widetilde{A}_{\alpha_2}$. If $\sigma_j' \in \widetilde{A}_{\alpha_2}$ for $i_2 \le j \le (k+1)$, then $t=2$ and we obtain a trajectory of the second type, otherwise we continue in the same way. Hence we obtain a trajectory of the form stated in \ref{type2}. Consequently, we observe that $\dim(\sigma_{k+1})=\dim(\tau_0) +t-1$.
			\end{proof}
		
		The next lemma uses the properties of these trajectories to establish that $\W$ is indeed a gradient vector field.
		\begin{lem}
			$\W$ is a gradient vector field on $\wX$.
		\end{lem}
		\begin{proof}
			Any closed $\W$-trajectory is either pure or mixed. From \autoref{traj}\ref{type1}, we observe that any closed pure $\W$-trajectory in $\widetilde{A}_\alpha$, for some $\alpha \in \mathcal{N}(X)$, 
			corresponds to a closed $\oW$-trajectory in $\Aa$.
			Thus, a closed pure $\W$-trajectory cannot exist in $\wX$. 
			
			Now, let $P: \tau_0', \sigma_1', \tau_1', \dots , \sigma_k', \tau_k', \sigma_{k+1}'$ be a mixed $\W$-trajectory, where $\tau_i'=(\tau_i, \beta_i)$, for each $i \in \{0, \dots ,k\}$, $\sigma_i'=(\sigma_i, \alpha_i)$ for each $i \in [k+1]$. Therefore, from \autoref{traj}\ref{type2} there must exist $\sigma_r'=(\sigma_r, \alpha_r)$ such that $\sigma_r' \in \widetilde{A}_{\alpha_r}$, where $\alpha_r <_f \beta_0$, for some $r \ge 1$ and moreover, $\dim(\tau_j)> \dim(\tau_0)$ for any $j \ge r$. Therefore, $\tau_0' \ne \tau_j'$ for each $j \ge r$. So, there cannot exist a closed $\W$-trajectory $P$ which is mixed.
			Hence $\W$ is a gradient vector field on $\wX$.
		\end{proof}
		Therefore, for each $q \ge 0$, the $\W$-critical $q$-dimensional simplotopes are given by,
		$$ \CrW=\bigcup_{i=0}^q \bigcup_{\alpha^{(q-i)} \in \N}\{(\sigma, \alpha) \mid \sigma_\alpha \in \Crit_i^{\Wo_{\alpha}}(\Aa)\}.$$
		
		Let $\tau'=(\tau^{(i)}, \beta^{(q+1-i)}) \in \CrWq$ for a fixed $i \in \{0, \dots ,(q+1)\}$. We note from \autoref{traj} that the terminal simplotope of a trajectory initiating from $\tau'$ and ending at some $\sigma' \in \CrW$ must be of the form $(\sigma^{(i+t-1)}, \alpha^{(q+1-i-t)}) \in \CrW$ where $0 \le t \le q+1-i$. So, for each $t \in \{0, \dots , (q+1-i)\}$, let $(\sigma_t, \alpha_t)=(\sigma^{(i+t-1)}, \alpha^{(q+1-i-t)}) \in \CrW$. Then, we can write the $(q+1)^{th}$ boundary map of the simplotopal Thom--Smale chain complex of $\wX$ with respect to $\W$ as,
		$$ \pW_{q+1}(\tau')=
		\sum_{t=0}^{q+1-i} \sum_{(\sigma_t, \alpha_t) \in \CrW}\left(\sum_{P \in \Gamma(\tau', (\sigma_t, \alpha_t))} w(P)\right)(\sigma_t, \alpha_t),
		$$
		where $\Gamma(\tau', (\sigma_t, \alpha_t))$ represents the set of all $\W$-trajectories with initial simplotope $\tau'$ and terminal simplotope $(\sigma_t,\alpha_t)$, as introduced in Section~\ref{dmt}.
		\subsection{Isomorphism between simplotopal homology of $\wX$ and homology groups of $\Lc$}
		In this subsection, we first show that the simplotopal Thom--Smale chain complex of $\wX$ with respect to $\W$, $(C^{\W}_{\#}(\wX), \pW_{\#})$ is isomorphic to $(\Lq_{\#}(X), \pL_{\#})$, which will establish an isomorphism between the simplotopal homology of $\wX$ and the homology of $\Lc$. Consequentially, it will follow from \autoref{first} that the simplicial homology groups of $X$ are isomorphic to the respective homology groups of $\Lc$.
		
		We recall from Section~\ref{intro}, that for each $q \ge 0$,
		$$ L_q(X):= \bigcup_{i=0}^q \bigcup_{\alpha^{(q-i)} \in \N} \Crit_i^{\overline{\W}_{\alpha}}(\bar{A}_{\alpha}),$$
		and $\Lq_q(X):= \Z\langle L_q(X)\rangle$.
		Now, we further recall that the $q^{th}$ simplotopal Thom--Smale chain group of $\wX$ is given by,
		$$ \CW_q(\wX)= \mathbb{Z}\langle \CrW\rangle.$$
		
		Next, we prove some results which will play a crucial role in establishing an isomorphism between $(C^{\W}_{\#}(\wX), \pW_{\#})$ and $\Lc$.
		
		\begin{prop}\label{chiso}
			For each $q \ge 0$, $\CW_q(\wX) \cong \Lq_q(X)$. Further, for each $q \ge 0$, this isomorphism is given by,
			$f_q: \CW_q(\wX) \rightarrow \Lq_q(X)$, where for each $(\sigma^{(i)}, \alpha ^{(q-i)}) \in \CrW $, $0 \le i \le q$,
			$$ (\sigma^{(i)}, \alpha ^{(q-i)}) \longmapsto \sigma_\alpha^{(i
				)}.$$
		\end{prop}
		\begin{proof}
			Let us construct $f: \CW_q(\wX) \rightarrow \Lq_q(X)$, where for each $(\sigma^{(i)}, \alpha ^{(q-i)}) \in \CrW $, $0 \le i \le q$,
			$$ (\sigma^{(i)}, \alpha ^{(q-i)}) \longmapsto \sigma_\alpha,$$
			It follows from the characterisation of the $\W$-critical simplotopes that $\sigma_\alpha \in \Crit_i^{\oW}(\Aa)$. Hence, from the definition of $L_q$, it follows that $\sigma_\alpha \in L_q$. Therefore, this map is well-defined. Now we define another map $g:  \Lq_q(X) \rightarrow  \CW_q(\wX)$, as follows. Let $\sigma_\alpha^{(i)} \in L_q $ where $0 \le i \le q$, that is, $\sigma_\alpha \in \Crit_i^{\oW}(\Aa)$, where $\alpha^{(q-i)} \in \N$ (from the definition of $L_q$). Then,
			$$ \sigma_\alpha \longmapsto(\sigma^{(i)}, \alpha ^{(q-i)}).$$
			Again, from the characterisation of the $\W$-critical simplotopes, it follows that $(\sigma, \alpha) \in \CrW$. Hence, this map is well-defined. We define these maps on the set of generators and extend them linearly to their domains. It follows from the construction of the maps, that $f \circ g =\id$ and $g \circ f = \id$ on the set of generators and hence, on the entire domains. Therefore, $\CW_q(\wX) \cong \Lq_q(X)$.
		\end{proof}
		Now, we show that there exists a one-one correspondence between the elements of $\Gba$ and $\Tf$, that is, $\Tba$ (from \autoref{chiso}) for each $(\tau,\beta) \in \CrWq$, $(\sigma, \alpha) \in \CrW$, $q \ge 0$.
		\begin{prop}\label{bij}
			Let $(\tau,\beta) \in \CrWq, (\sigma, \alpha) \in \CrW$, $q \ge 0$. Then there exists a bijection between $\Gba$ and $\Tba$.
		\end{prop}
		\begin{proof}
			First, we construct a map, $\psi: \Gba \rightarrow \Tba$ and then we construct $\rho: \Tba \rightarrow \Gba$. Finally, we show that $\rho \circ \psi = \id$ and $\psi \circ \rho = \id$.
			
			For the construction of $\psi$, let $P \in \Gba$. Therefore, from \autoref{traj}, we define $\psi$ for the following two cases.\\
			
			\noindent \textbf{Case I:} $P$ is a pure trajectory in $\widetilde{A}_{\beta}$. Let $(\tau^{(i)}, \beta) \in \CrWq$, $(\sigma^{(i-1)},\alpha) \in \CrW$ where $1\le i \le (q+1)$. Then, from \autoref{traj}\ref{type1}, $P$ is of the following form.	
			$$ P: ((\tau,\beta)=)~(\tau_0^{(i)},\beta), (\sigma_1^{(i-1)},\beta), (\tau_1^{(i)}, \beta), \dots, (\sigma_k^{(i-1)},\beta), (\tau_k^{(i)},\beta), (\sigma_{k+1}^{(i-1)}, \beta) ~(=(\sigma,\alpha)). $$
			
			Again, from \autoref{traj}\ref{type1}, we know that such a trajectory corresponds to a unique $\Wo_{\beta}$-trajectory. So, we define $\psi(P)$ as this unique $\Wo_{\beta}$-trajectory, that is,
			$$ \psi(P): (\tau_\beta=)~(\tau_0)^{(i)}_{\beta}, (\sigma_1)^{(i-1)}_{\beta}, (\tau_1)^{(i)}_{\beta}, \dots, (\sigma_k)^{(i-1)}_{\beta}, (\tau_k)^{(i)}_{\beta}, (\sigma_{k+1})^{(i-1)}_{\beta} ~(=\sigma_\alpha).$$
			
			Thus, $\psi$ is well-defined and $\psi(P)$ is a generalised trajectory of the first kind (as defined in \autoref{gt}, Section~\ref{intro}), that is, $\psi(P)\in \Tba$.\\
			
			\noindent \textbf{Case II:} $P$ is a mixed trajectory. Let $(\tau^{(i)},\beta^{(q+1-i)}) \in \CrWq$, $(\sigma^{(i+t-1)},\alpha^{(q+1-i-t)}) \in \CrW$ (the dimensions of $\sigma$, $\alpha$ follow from \autoref{traj}\ref{type2}), where $0 \le i < (q+1)$ and $1 \leq t \leq (q+1-i)$.
			Then from \autoref{traj} \ref{type2}, $P$ is of the form,
			$$ \begin{aligned}
				P:~ &((\tau, \beta)=)~(\tau_{1,0}^{(i)}, \alpha_0), (\tau_{0,1}^{(i-1)}, \alpha_0), (\tau_{1,1}^{(i)}, \alpha_0), \dots , (\tau_{0,j_1}^{(i-1)}, \alpha_0),(\tau_{1,j_1}^{(i)}, \alpha_0),(\tau_{1,j_1}^{(i)}, \alpha_1), \\ & (\tau_{2,j_1+1}^{(i+1)}, \alpha_1), (\tau_{1,j_1+1}^{(i)}, \alpha_1),
				(\tau_{2,j_1+2}^{(i+1)}, \alpha_1), \dots , (\tau_{t-1,j_t-1}^{(i+t-2)}, \alpha_{t-1}), (\tau_{t,j_t}^{(i+t-1)}, \alpha_{t-1}), (\tau_{t,j_t}^{(i+t-1)}, \alpha_t),	\\ &(\tau_{t+1,j_t+1}^{(i+t)},\alpha_t), (\tau_{t,j_t+1}^{(i+t-1)}, \alpha_t), \dots, (\tau_{t+1,j_t+r}^{(i+t)},\alpha_t), (\tau_{t,j_t+r}^{(i+t-1)}, \alpha_t)~(=(\sigma,\alpha)).
			\end{aligned} $$
			where $\alpha=\alpha_t^{(q+1-i-t)} <_f \alpha_{t-1} <_f \dots <_f\alpha_0^{(q+1-i)}=\beta$.
			In this case, we define $\psi(P)$ as,
			$$ \begin{aligned}
				\psi(P):~ (\tau_\beta= )~&(\tau_{1,0})^{(i)}_{\alpha_0}, (\tau_{0,1})^{(i-1)}_{\alpha_0}, (\tau_{1,1})^{(i)}_{\alpha_0}, \dots , (\tau_{0,j_1})^{(i-1)}_{\alpha_0},(\tau_{1,j_1})^{(i)}_{\alpha_0},(\tau_{1,j_1})^{(i)}_{\alpha_1}, (\tau_{2,j_1+1})^{(i+1)}_{\alpha_1}, \\ &(\tau_{1,j_1+1})^{(i)}_{\alpha_1}, (\tau_{2,j_1+2})^{(i+1)}_{\alpha_1}, \dots , (\tau_{t-1,j_t-1})^{(i+t-2)}_{\alpha_{t-1}}, (\tau_{t,j_t})^{(i+t-1)}_{\alpha_{t-1}}, (\tau_{t,j_t})^{(i+t-1)}_{\alpha_t},\\ &(\tau_{t+1,j_t+1})^{(i+t)}_{\alpha_t}, (\tau_{t,j_t+1})^{(i+t-1)}_{\alpha_t}, \dots, (\tau_{t+1,j_t+r})^{(i+t)}_{\alpha_t}, (\tau_{t,j_t+r})^{(i+t-1)}_{\alpha_t}~(=\sigma_\alpha).
			\end{aligned}  $$
			(Informally speaking, each simplotope $(\sigma, \alpha)$ of $P$ is mapped to the copy of $\sigma$ in $\Aa$ in $\psi(P)$.)
			
			It follows from the properties of mixed trajectories (\autoref{traj}\ref{type2}) that $\psi$ is well-defined and $\psi(P) \in \Tba$, more precisely, $\psi(P)$ is a generalised trajectory of the second kind, as introduced in Section~\ref{intro} (\autoref{gt}).
			
			Next, for the construction of $\rho$, let $P \in \Tba$. Thus, we now define $\rho$ for the following two cases, according to the two types of generalised trajectories as introduced in Section~\ref{intro}.\\
			
			\noindent \textbf{Case I:} Let $\tau_\beta^{(i)} \in L_{q+1}(X), \sigma_\alpha^{(i-1)} \in L_q(X)$ where $1 \le i \le (q+1)$ and $P$ be a generalised trajectory of the first kind. Thus, $P$ is of the form,
			$$ P: (\tau_\beta=)~(\tau_0)^{(i)}_{\beta}, (\sigma_1)^{(i-1)}_{\beta}, (\tau_1)^{(i)}_{\beta}, \dots, (\sigma_k)^{(i-1)}_{\beta}, (\tau_k)^{(i)}_{\beta}, (\sigma_{k+1})^{(i-1)}_{\beta} ~(=\sigma_\alpha).$$
			We define $\rho(P)$ as,
			$$ \rho(P): ((\tau,\beta)=)~(\tau_0^{(i)},\beta), (\sigma_1^{(i-1)},\beta), (\tau_1^{(i)}, \beta), \dots, (\sigma_k^{(i-1)},\beta), (\tau_k^{(i)},\beta), (\sigma_{k+1}^{(i-1)}, \beta) ~(=(\sigma,\alpha)). $$
			It follows from the construction that $\rho(P)$ is well-defined and $\rho(P) \in \Gba$, more precisely, $\rho(P)$ is a pure trajectory in $\wA$ (from \autoref{traj}\ref{type1}). \\
			
			\noindent \textbf{Case II:}  Let $\tau_\beta^{(i)} \in L_{q+1}$, $\sigma_\alpha^{(i+t-1)} \in L_q$ where $0 \le i < (q+1)$ and $1 \leq t \leq (q+1-i)$, and $P$ be of the form,
			
			$$ \begin{aligned}
				P:~ &(\tau_{1,0})^{(i)}_{\alpha_0}, (\tau_{0,1})^{(i-1)}_{\alpha_0}, (\tau_{1,1})^{(i)}_{\alpha_0}, \dots , (\tau_{0,j_1})^{(i-1)}_{\alpha_0},(\tau_{1,j_1})^{(i)}_{\alpha_0},(\tau_{1,j_1})^{(i)}_{\alpha_1}, (\tau_{2,j_1+1})^{(i+1)}_{\alpha_1}, \\ &(\tau_{1,j_1+1})^{(i)}_{\alpha_1}, (\tau_{2,j_1+2})^{(i+1)}_{\alpha_1}, \dots , (\tau_{t-1,j_t-1})^{(i+t-2)}_{\alpha_{t-1}}, (\tau_{t,j_t})^{(i+t-1)}_{\alpha_{t-1}}, (\tau_{t,j_t})^{(i+t-1)}_{\alpha_t},\\ &(\tau_{t+1,j_t+1})^{(i+t)}_{\alpha_t}, (\tau_{t,j_t+1})^{(i+t-1)}_{\alpha_t}, \dots, (\tau_{t+1,j_t+r})^{(i+t)}_{\alpha_t}, (\tau_{t,j_t+r})^{(i+t-1)}_{\alpha_t}.
			\end{aligned}  $$
			where $\alpha=\alpha_t^{(q+1-i-t)} <_f \alpha_{t-1} <_f \dots <_f\alpha_0^{(q+1-i)}=\beta$.
			In this case, we define $\rho(P)$ as,
			$$ \begin{aligned}
				\rho(P):~ &((\tau, \beta)=)~(\tau_{1,0}^{(i)}, \alpha_0), (\tau_{0,1}^{(i-1)}, \alpha_0), (\tau_{1,1}^{(i)}, \alpha_0), \dots , (\tau_{0,j_1}^{(i-1)}, \alpha_0),(\tau_{1,j_1}^{(i)}, \alpha_0),(\tau_{1,j_1}^{(i)}, \alpha_1), \\ & (\tau_{2,j_1+1}^{(i+1)}, \alpha_1), (\tau_{1,j_1+1}^{(i)}, \alpha_1),
				(\tau_{2,j_1+2}^{(i+1)}, \alpha_1), \dots , (\tau_{t-1,j_t-1}^{(i+t-2)}, \alpha_{t-1}), (\tau_{t,j_t}^{(i+t-1)}, \alpha_{t-1}), (\tau_{t,j_t}^{(i+t-1)}, \alpha_t),	\\ &(\tau_{t+1,j_t+1}^{(i+t)},\alpha_t), (\tau_{t,j_t+1}^{(i+t-1)}, \alpha_t), \dots, (\tau_{t+1,j_t+r}^{(i+t)},\alpha_t), (\tau_{t,j_t+r}^{(i+t-1)}, \alpha_t)~(=(\sigma,\alpha)).
			\end{aligned} $$

			It follows from \autoref{traj}\ref{type2} that $\rho$ is well-defined and $\rho(P) \in \Gba$, more precisely, $\rho(P)$ is a mixed trajectory.
			
			It can be observed from the construction of the maps $\psi$ and $\rho$ that $\psi \circ \rho= \id$ and $\rho \circ \psi = \id$. Therefore, it follows that $\psi: \Gba \rightarrow \Tba$ is a bijection with its inverse as $\rho: \Tba \rightarrow \Gba$.
		\end{proof}
		The following observation describes the isomorphism stated in the above proposition. This follows naturally from the proof of the above proposition (\autoref{bij}).
		
		\begin{obs}\label{psi}
			Let $(\tau, \beta) \in \CrWq$, $(\sigma, \alpha) \in \CrW$, and $P \in \Gba$. Then, the isomorphism $\psi: \Gba \rightarrow \Tba$ is defined for the following two cases,
			
			\begin{enumerate}[label=(\roman*)]
				\item  $P$ is a pure trajectory in $\widetilde{A}_{\beta}$. Suppose, $(\tau^{(i)}, \beta) \in \CrWq$, $(\sigma^{(i-1)},\alpha) \in \CrW$, where $1\le i \le (q+1)$ and $P$ is of the form,
				$$ P: ((\tau,\beta)=)~(\tau_0^{(i)},\beta), (\sigma_1^{(i-1)},\beta), (\tau_1^{(i)}, \beta), \dots, (\sigma_k^{(i-1)},\beta), (\tau_k^{(i)},\beta), (\sigma_{k+1}^{(i-1)}, \beta) ~(=(\sigma,\alpha)). $$
				In this case, $\psi(P)$ is defined as,
				$$ \psi(P): (\tau_\beta=)~(\tau_0)^{(i)}_{\beta}, (\sigma_1)^{(i-1)}_{\beta}, (\tau_1)^{(i)}_{\beta}, \dots, (\sigma_k)^{(i-1)}_{\beta}, (\tau_k)^{(i)}_{\beta}, (\sigma_{k+1})^{(i-1)}_{\beta} ~(=\sigma_\alpha).$$ \label{psi1}
				\item $P$ is a mixed trajectory. Suppose, $(\tau^{(i)},\beta^{(q+1-i)}) \in \CrWq$, $(\sigma^{(i+t-1)},\alpha^{(q+1-i-t)}) \in \CrW$ where $0 \le i < (q+1)$ and $1 \leq t \leq (q+1-i)$ and $P$ is of the form,
				$$ \begin{aligned}
					P:~  &((\tau, \beta)=)~(\tau_{1,0}^{(i)}, \alpha_0), (\tau_{0,1}^{(i-1)}, \alpha_0), (\tau_{1,1}^{(i)}, \alpha_0), \dots , (\tau_{0,j_1}^{(i-1)}, \alpha_0),(\tau_{1,j_1}^{(i)}, \alpha_0),(\tau_{1,j_1}^{(i)}, \alpha_1), \\ & (\tau_{2,j_1+1}^{(i+1)}, \alpha_1), (\tau_{1,j_1+1}^{(i)}, \alpha_1),
					(\tau_{2,j_1+2}^{(i+1)}, \alpha_1), \dots , (\tau_{t-1,j_t-1}^{(i+t-2)}, \alpha_{t-1}), (\tau_{t,j_t}^{(i+t-1)}, \alpha_{t-1}), (\tau_{t,j_t}^{(i+t-1)}, \alpha_t),	\\ &(\tau_{t+1,j_t+1}^{(i+t)},\alpha_t), (\tau_{t,j_t+1}^{(i+t-1)}, \alpha_t), \dots, (\tau_{t+1,j_t+r}^{(i+t)},\alpha_t), (\tau_{t,j_t+r}^{(i+t-1)}, \alpha_t)~(=(\sigma,\alpha)).
				\end{aligned} $$
				where $\alpha=\alpha_t^{(q+1-i-t)} <_f \alpha_{t-1} <_f \dots <_f\alpha_0^{(q+1-i)}=\beta$.
				In this case, $\psi(P)$ is defined as,
				$$ \begin{aligned}
					\psi(P):~ (\tau_\beta= )~&(\tau_{1,0})^{(i)}_{\alpha_0}, (\tau_{0,1})^{(i-1)}_{\alpha_0}, (\tau_{1,1})^{(i)}_{\alpha_0}, \dots , (\tau_{0,j_1})^{(i-1)}_{\alpha_0},(\tau_{1,j_1})^{(i)}_{\alpha_0},(\tau_{1,j_1})^{(i)}_{\alpha_1}, (\tau_{2,j_1+1})^{(i+1)}_{\alpha_1}, \\ &(\tau_{1,j_1+1})^{(i)}_{\alpha_1}, (\tau_{2,j_1+2})^{(i+1)}_{\alpha_1}, \dots , (\tau_{t-1,j_t-1})^{(i+t-2)}_{\alpha_{t-1}}, (\tau_{t,j_t})^{(i+t-1)}_{\alpha_{t-1}}, (\tau_{t,j_t})^{(i+t-1)}_{\alpha_t},\\ &(\tau_{t+1,j_t+1})^{(i+t)}_{\alpha_t}, (\tau_{t,j_t+1})^{(i+t-1)}_{\alpha_t}, \dots, (\tau_{t+1,j_t+r})^{(i+t)}_{\alpha_t}, (\tau_{t,j_t+r})^{(i+t-1)}_{\alpha_t}~( =\sigma_\alpha).
				\end{aligned}  $$ \label{psi2a}
		\end{enumerate}
	\end{obs}
	\begin{prop}\label{weight}
		Let $(\tau, \beta) \in \CrWq$, $(\sigma, \alpha) \in \CrW$ and $P \in \Gba$. Then, $w(P)=w_G(\psi(P))$, where $\psi$ is the bijection between $\Gba$ and $\Tba$ (as asserted in \autoref{bij}).
	\end{prop}
	\begin{proof}
		Let $P \in \Gba$. Then, from \autoref{traj}, we have the following two cases.\\
		
		\noindent \textbf{Case I:} $P$ is a pure trajectory in $\widetilde{A}_{\beta}$. Let $(\tau^{(i)}, \beta) \in \CrWq$, $1\le i \le (q+1)$. Then, from \autoref{traj}\ref{type1}, $P$ is of the following form.
		$$ P: ((\tau,\beta)=)~(\tau_0^{(i)},\beta), (\sigma_1^{(i-1)},\beta), (\tau_1^{(i)}, \beta), \dots, (\sigma_k^{(i-1)},\beta), (\tau_k^{(i)},\beta), (\sigma_{k+1}^{(i-1)}, \beta) ~(=(\sigma,\alpha)). $$
		and from Observation~\ref{psi}\ref{psi1}, $\psi(P)$ is given by,
		$$ \psi(P): (\tau_\beta=)~(\tau_0)^{(i)}_{\beta}, (\sigma_1)^{(i-1)}_{\beta}, (\tau_1)^{(i)}_{\beta}, \dots, (\sigma_k)^{(i-1)}_{\beta}, (\tau_k)^{(i)}_{\beta}, (\sigma_{k+1})^{(i-1)}_{\beta} ~(=\sigma_\alpha).$$
		Then we recall from Section~\ref{dmt}, that
		\begin{equation*}
			\begin{aligned}
				w(P)&= \left(\prod_{i=0}^{k-1}(-1)\langle (\tau_i, \beta), (\sigma_{i+1}, \beta)\rangle \langle (\tau_{i+1}, \beta), (\sigma_{i+1}, \beta) \rangle\right) \langle (\tau_k, \beta), (\sigma_{k+1}, \beta)\rangle\\
				&= \left(\prod_{i=0}^{k-1}(-1)\langle \tau_i, \sigma_{i+1}\rangle \langle \tau_{i+1}, \sigma_{i+1} \rangle\right) \langle \tau_k, \sigma_{k+1}\rangle, \hspace{2em} \text{ (follows from \autoref{inc} in Section~\ref{prelim})} \\ &
				= w_G(\psi(P)), \hspace{2em} \text{ (follows from \autoref{gt}).}
			\end{aligned}
		\end{equation*}
		
		\noindent \textbf{Case II:} $P$ is a mixed trajectory. Let $(\tau^{(i)},\beta^{(q+1-i)}) \in \CrWq$, $(\sigma^{(i+t-1)},\alpha^{(q+1-i-t)}) \in \CrW$ (the dimensions of $(\sigma, \alpha)$ follow from \autoref{traj}\ref{type2}) where $0 \le i < (q+1)$ and $1 \leq t \leq (q+1-i)$.
		Then from \autoref{traj}\ref{type2}, $P$ is of the form,
		$$ \begin{aligned}
			P:~ &((\tau, \beta)=)~(\tau_{1,0}^{(i)}, \alpha_0), (\tau_{0,1}^{(i-1)}, \alpha_0), (\tau_{1,1}^{(i)}, \alpha_0), \dots , (\tau_{0,j_1}^{(i-1)}, \alpha_0),(\tau_{1,j_1}^{(i)}, \alpha_0),(\tau_{1,j_1}^{(i)}, \alpha_1), \\ & (\tau_{2,j_1+1}^{(i+1)}, \alpha_1), (\tau_{1,j_1+1}^{(i)}, \alpha_1),
			(\tau_{2,j_1+2}^{(i+1)}, \alpha_1), \dots , (\tau_{t-1,j_t-1}^{(i+t-2)}, \alpha_{t-1}), (\tau_{t,j_t}^{(i+t-1)}, \alpha_{t-1}), (\tau_{t,j_t}^{(i+t-1)}, \alpha_t),	\\ &(\tau_{t+1,j_t+1}^{(i+t)},\alpha_t), (\tau_{t,j_t+1}^{(i+t-1)}, \alpha_t), \dots, (\tau_{t+1,j_t+r}^{(i+t)},\alpha_t), (\tau_{t,j_t+r}^{(i+t-1)}, \alpha_t)~(=(\sigma,\alpha)).
		\end{aligned} $$
		where $\alpha=\alpha_t^{(q+1-i-t)} <_f \alpha_{t-1} <_f \dots <_f\alpha_0^{(q+1-i)}=\beta$.
		Thus, from Observation~\ref{psi}\ref{psi2a}$, \psi(P)$ is given by,
		$$ \begin{aligned}
			\psi(P):~ &(\tau_{1,0})^{(i)}_{\alpha_0}, (\tau_{0,1})^{(i-1)}_{\alpha_0}, (\tau_{1,1})^{(i)}_{\alpha_0}, \dots , (\tau_{0,j_1})^{(i-1)}_{\alpha_0},(\tau_{1,j_1})^{(i)}_{\alpha_0},(\tau_{1,j_1})^{(i)}_{\alpha_1}, (\tau_{2,j_1+1})^{(i+1)}_{\alpha_1}, \\ &(\tau_{1,j_1+1})^{(i)}_{\alpha_1}, (\tau_{2j_1+2})^{(i+1)}_{\alpha_1}, \dots , (\tau_{t-1,j_t-1})^{(i+t-2)}_{\alpha_{t-1}}, (\tau_{t,j_t})^{(i+t-1)}_{\alpha_{t-1}}, (\tau_{t,j_t})^{(i+t-1)}_{\alpha_t},\\ &(\tau_{t+1,j_t+1})^{(i+t)}_{\alpha_t}, (\tau_{t,j_t+1})^{(i+t-1)}_{\alpha_t}, \dots, (\tau_{t+1,j_t+r})^{(i+t)}_{\alpha_t}, (\tau_{t,j_t+r})^{(i+t-1)}_{\alpha_t}.
		\end{aligned}  $$
		Now, the weight of $P$ is,
		$$\begin{aligned}
			w(P)= & \left(\prod_{k=0}^{j_1-1}(-1)\langle (\tau_{1,k}, \alpha_0), (\tau_{0, k+1}, \alpha_0)\rangle \langle (\tau_{1,k+1}, \alpha_0), (\tau_{0,k+1}, \alpha_0) \rangle \right) \langle(\tau_{1,j_1}, \alpha_0), (\tau_{1,j_1}, \alpha_1)\rangle \\ & \prod_{l=1}^{t-1}\left(\left( \prod_{k=j_l}^{j_{l+1}-2}(-1) \langle (\tau_{l+1,k+1},\alpha_l), (\tau_{l,k}, \alpha_l) \rangle \langle (\tau_{l+1,k+1}, \alpha_l), (\tau_{l,k+1}, \alpha_l) \rangle\right) \right. \\&\qquad\left.  (-1)\langle (\tau_{l+1,j_{l+1}}, \alpha_l), (\tau_{l,j_{l+1}-1}, \alpha_l) \rangle \langle (\tau_{l+1, j_{l+1}}, \alpha_l), (\tau_{l+1,j_{l+1}}, \alpha_{l+1}) \rangle \right) \\&\left( \prod_{k=j_t}^{j_t+r-1}(-1) \langle (\tau_{t+1,k+1}, \alpha_t), (\tau_{t,k}, \alpha_t) \rangle \langle (\tau_{t+1,k+1}, \alpha_t), (\tau_{t,k+1}, \alpha_t) \rangle\right) \\
		\end{aligned}$$
		$$\begin{aligned}
			=& \prod_{l=0}^{t-1}\langle(\tau_{l+1,j_{l+1}}, \alpha_{l}), (\tau_{l+1,j_{l+1}}, \alpha_{l+1})\rangle \left(\prod_{k=0}^{j_1-1}(-1) \langle (\tau_{1,k}, \alpha_0), (\tau_{0, k+1}, \alpha_0)\rangle \langle (\tau_{1,k+1}, \alpha_0), (\tau_{0,k+1}, \alpha_0) \rangle \right) \\ &\prod_{l=1}^{t-1}\left(\left( \prod_{k=j_l}^{j_{l+1}-2}(-1)\langle (\tau_{l+1,k+1},\alpha_l), (\tau_{l,k}, \alpha_l) \rangle \langle (\tau_{l+1,k+1}, \alpha_l), (\tau_{l,k+1}, \alpha_l) \rangle\right)\right. \\&\qquad\left.(-1)\langle (\tau_{l+1,j_{l+1}}, \alpha_l), (\tau_{l,j_{l+1}-1}, \alpha_l) \rangle\right) \left( \prod_{k=j_t}^{j_t+r-1}(-1) \langle (\tau_{t+1,k+1}, \alpha_t), (\tau_{t,k}, \alpha_t) \rangle \langle (\tau_{t+1,k+1}, \alpha_t), (\tau_{t,k+1}, \alpha_t) \rangle\right) \\
			&= \prod_{l=0}^{t-1}(-1)^{i+l}\langle \alpha_{l}, \alpha_{l+1}\rangle \left(\prod_{k=0}^{j_1-1}(-1) \langle \tau_{1,k}, \tau_{0, k+1}\rangle \langle \tau_{1,k+1},\tau_{0,k+1} \rangle \right)\\ &\prod_{l=1}^{t-1}\left(\left( \prod_{k=j_l}^{j_{l+1}-2}(-1) \langle \tau_{l+1,k+1}, \tau_{l,k} \rangle \langle \tau_{l+1,k+1}, \tau_{l,k+1} \rangle\right) (-1)\langle \tau_{l+1,j_{l+1}}, \tau_{l,j_{l+1}-1} \rangle \right)\\ &\left( \prod_{k=j_t}^{j_t+r-1}(-1) \langle \tau_{t+1,k+1}, \tau_{t,k} \rangle \langle \tau_{t+1,k+1}, \tau_{t,k+1} \rangle\right), \hspace{0.5em} \text{ (follows from \autoref{inc}),}\\
			&= (-1)^{\phi(i,t)} \prod_{l=0}^{t-1} \langle \alpha_{l-1}, \alpha_{l}\rangle \left(\prod_{k=0}^{j_1-1}(-1) \langle \tau_{1,k}, \tau_{0,k+1} \rangle \langle \tau_{1,k+1}, \tau_{0,k+1} \rangle \right)\\ &\prod_{l=1}^{t-1}\left(\left( \prod_{k=j_l}^{j_{l+1}-2}(-1) \langle \tau_{l+1,k+1}, \tau_{l,k} \rangle \langle \tau_{l+1,k+1}, \tau_{l,k+1} \rangle\right) (-1)\langle \tau_{l+1,j_{l+1}}, \tau_{l,j_{l+1}-1} \rangle \right)\\ &\left( \prod_{k=j_t}^{j_t+r-1}(-1) \langle \tau_{t+1,k+1}, \tau_{t,k} \rangle \langle \tau_{t+1,k+1}, \tau_{t,k+1} \rangle\right), \hspace{1em} \text{ where } \phi(i,t)= it + \frac{t(t-1)}{2}, \\&
			=w_G(\psi(P)) \hspace{2em} \text{ (follows from \autoref{gt}).}
		\end{aligned}$$
	
		Hence $w(P)= w_G(\psi(P))$ for each $P \in \Gba$.
	\end{proof}
	\begin{prop}\label{iso}
		The following diagram commutes.
		\[\begin{tikzcd}
			\cdots & {\CW_{q+1}(\wX)} && {\CW_{q}(\wX)} & \cdots \\
			\\
			\cdots & {\Lq_{q+1}(X)} && {\Lq_{q}(X)} & \cdots
			\arrow[from=1-1, to=1-2]
			\arrow["{\pW_{q+1}}", from=1-2, to=1-4]
			\arrow["{f_{q+1}}"', from=1-2, to=3-2]
			\arrow[from=1-4, to=1-5]
			\arrow["{f_{q}}"', from=1-4, to=3-4]
			\arrow[from=3-1, to=3-2]
			\arrow["{\pL_{q+1}}"', from=3-2, to=3-4]
			\arrow[from=3-4, to=3-5]
			\arrow[draw=none, from=1-2, to=3-4, "{\text{\Huge$\circlearrowleft$}}" description]
		\end{tikzcd}\]
		Here, $f_q$ is the isomorphism between $\CW_q(\wX)$ and $\Lq_q(X)$ (as asserted in \autoref{chiso}) for each $q \ge 0$. Hence, $(\Lq_{\#}, \pL_{\#})$ is a chain complex and $(\CW_{\#}, \pW_{\#}) \cong (\Lq_{\#}, \pL_{\#})$. Consequentially, $\h^{\W}_{\#}(\wX) \cong H^{\Lq}_{\#}(X)$.
	\end{prop}
	\begin{proof}
		We only need to show that $f_q \circ \pW_{q+1} = \pL_{q+1} \circ f_{q+1}$ for each $q \ge 0$. It suffices to show that this holds on the set of generators.
		Let $\tau'= (\tau^{(i)}, \beta) \in \CrWq$ for a fixed $i \in \{0, \dots ,(q+1)\}$. For each $t \in \{0, \dots ,(q+1-i)\}$, let $(\sigma_t, \alpha_t)= (\sigma^{(i+t-1)}, \alpha^{(q+1-i-t)}) \in \CrW$. Therefore,
		\begin{equation*}
			\begin{aligned}
				f_q(\pW_{q+1}((\tau, \beta))) &= f_q\left(\sum_{t=0}^{q+1-i} \sum_{(\sigma_t, \alpha_t) \in \CrW}\left(\sum_{P \in \Gamma(\tau', (\sigma_t, \alpha_t))}w(P)\right) (\sigma_t, \alpha_t)\right) \\
				&= \sum_{t=0}^{q+1-i} \sum_{(\sigma_t, \alpha_t) \in \CrW}\left(\sum_{P \in \Gamma(\tau', (\sigma_t, \alpha_t))}w(P)\right) (\sigma_t)_{\alpha_t},\text{ (follows from \autoref{chiso}),} \\
				&= \sum_{t=0}^{q+1-i} \sum_{(\sigma_t, \alpha_t) \in \CrW}\left(\sum_{\psi(P) \in \GT(\tau_\beta, (\sigma_t)_{\alpha_t})}w(P)\right) (\sigma_t)_{\alpha_t}, \text{ (follows from \autoref{bij})} \\
				&= \sum_{t=0}^{q+1-i} \sum_{(\sigma_t, \alpha_t) \in \CrW}\left(\sum_{\psi(P) \in \GT(\tau_\beta, (\sigma_t)_{\alpha_t})}w_G(\psi(P))\right) (\sigma_t)_{\alpha_t}, \\ & \text{ (follows from \autoref{weight})} \\
				&= \sum_{t=0}^{q+1-i} \sum_{(\sigma_t)_{\alpha_t} \in L_{q+1}(X)}\left(\sum_{\psi(P) \in \GT(\tau_\beta, (\sigma_t)_{\alpha_t})}w_G(\psi(P))\right) (\sigma_t)_{\alpha_t}, \\ & \text{(follows from \autoref{chiso})}\\
				&= \pL_{q+1}(\tau_{\beta}) \\
				&= \pL_{q+1}(f_{q+1}((\tau, \beta))).
			\end{aligned}
		\end{equation*}
		Thus, $(\Lq_{\#}, \pL_{\#})$ is a chain complex and $f_{\#}$ is a chain isomorphism between $(\CW_{\#}, \pW_{\#})$ and $(\Lq_{\#}, \pL_{\#})$. Hence, the result follows.

	\end{proof}
	Therefore, we obtain the following isomorphisms,
	$$ H_{\#}(X) \cong \h_{\#}(\wX) \cong \h^{\W}_{\#}(\wX) \cong H^{\Lq}_{\#}(X),$$
	where the first isomorphism  follows from \autoref{first}, the second isomorphism follows from \autoref{forman} while the third follows from \autoref{iso}. This establishes our main result (\autoref{main}).
	
	\section{An application of the combinatorial nerve theorem}\label{app}
	In this section, we prove a version of the usual homological nerve theorem using our combinatorial nerve theorem (\autoref{main}). Suppose, $A_1, \dots, A_k$ are subcomplexes of $X$ such that $\bigcup_{i=1}^{k}A_i=X$ and $A_{\alpha}$ is collapsible for each $\alpha \in \N$. Then, this theorem states that the homology groups of $X$ are isomorphic to the respective homology groups of $\N$. Before delving into this theorem, we first state a couple results regarding the properties of the generalised trajectories. The following observation follows immediately from the definition of generalised trajectories (\autoref{gt}, Section~\ref{intro}).
	\begin{obs}{\label{gt1}}
		Let $P$ be a generalised trajectory initiating from a simplex $\tau_{\beta}^{(i)} \in L_{q+1}$, where $\beta \in \N$. Then any terminal simplex of $P$ must be of the form $\sigma_{\alpha}^{(j)} \in L_q$, $j \ge (i-1)$, $\alpha \subseteq \beta$. Conversely, let $\si^{(j)} \in L_q$ be a terminal simplex of $P$. Then, the following holds.
		\begin{enumerate}[label=(\alph*)]
			\item $j=i-1$, if and only if $\alpha=\beta$, that is, $P$ is a generalised trajectory of the first kind,\label{obs1}
			\item $j=i$, if and only if $\alpha <_f \beta$.\label{obs2}
		\end{enumerate}
	\end{obs}
	
	\begin{lem}\label{0gt}
		Let $\sa^{(0)} \in L_{q+1}$, $\az^{(q+1)} \in \N$. Then, for each $\ao <_f \az$, there exists a unique generalised trajectory, $P$ from $\sa^{(0)}$ which ends at some $\sigma_{\ao}'^{(0)} \in L_q$. In other words, among the trajectories which begin at $\sa$, continue into $\bar{A}_{\ao}$ following $\Wo_{\ao}$ and terminate at a $\Wo_{\ao}$-critical simplex, there is only one (there might be other trajectories beginning from $\sa$ and continuing into $\Aa$ for $\alpha <_f \az$, $\alpha \ne \ao$, or trajectories going consecutively through $\bar{A}_{\ao}$, $\bar{A}_{\alpha_2}$, and so on and terminating at some $\Wo_{\alpha_t}$-critical simplex, $t >1$, where $\alpha_t <_f \dots <_f \alpha_2<_f \ao$). Further, $w_G(P)= \langle \az, \ao\rangle$.
	\end{lem}
	
	\begin{proof}
		Since, $\sa$ is a $0$-simplex, therefore the terminal simplex of any generalised trajectory beginning from $\sigma_{\az}$ cannot be of lower dimension. Therefore, from \autoref{gt1}, the only possible generalised trajectory emanating from $\sa$ is of the second kind (as listed in \autoref{gt}). 	Let $\ao <_f \az$. So, a typical generalised trajectory initiating from $\sa$ with terminal simplex as $\sigma_{\ao}'$, for some $\sigma_{\ao}' \in L_q$, is of the form,
		$$ P:~\sa^{(0)}, ((\sigma_0)_{\ao}=)~\sigma_{\alpha_1}^{(0)}, (\tau_0)^{(1)}_{\ao}, (\sigma_1)^{(0)}_{\ao}, \dots , (\sigma_k)^{(0)}_{\ao}, (\tau_k)^{(1)}_{\ao}, (\sigma_{k+1})^{(0)}_{\ao}~(= \s). $$
		
		Now we show that given $\ao <_f \alpha$, this trajectory is unique. We observe that each $(\sigma_i)_{\alpha_1}$ is paired uniquely with $(\tau_i)_{\alpha_1}$ with respect to $\Wo_{\ao}$. Further, for each $i$, the facet $(\sigma_{i+1})_{\alpha_1}$ of $(\tau_i)_{\alpha_1}$, following $(\tau_i)_{\alpha_1}$ in the sequence, is unique, $(\tau_i)_{\alpha_1}$ being an $1$-simplex,. Since the subcomplexes are finite and there exists at least one $\Wo_{\ao}$-critical $0$-simplex, therefore the trajectory ends uniquely at some $\sigma_{\ao}'^{(0)} \in L_q$.
		
		Next, we compute the weight of this generalised trajectory.
		\begin{equation*}
			\begin{aligned}
				w_G(P) &= (-1)^{0}\langle \alpha_0, \alpha_1 \rangle\left(\prod_{i=0}^{k}(-1)\langle \tau_i, \sigma_i \rangle \langle  \tau_i, \sigma_{i+1} \rangle \right), \text{ (from \autoref{gt})}\\ &= (-1)^{0}\langle \alpha_0, \alpha_1 \rangle \left(\prod_{i=0}^{k} -(-1)(1)\right), \text{ (since, $\tau_i$ is a $1$-simplex for each $i$.)} \\
				&=\langle \alpha_0, \alpha_1\rangle.
			\end{aligned}
		\end{equation*}
	\end{proof}
	Let us denote this trajectory emanating from $\sa^{(0)} \in L_{q+1}$ ending at $\sigma_{\ao}'^{(0)} \in L_q$ as $\Pa$. We note here that as a direct consequence of this observation, once we fix $\sa^{(0)} \in L_{q+1}$, and $\ao <_f \az$, we obtain a unique generalised trajectory $\Pa$ and hence a unique terminal simplex corresponding to $\Pa$. We denote this unique simplex corresponding to $\sa$ and $\ao <_f \az$ as $\mathbf{t}(\sa,\ao)$.
	
	Now, we state the theorem.
	\begin{thm}
		Let $X$ be a simplicial complex with subcomplexes $A_1, \dots ,A_k$ such that $X= \bigcup_{i=1}^k A_i$. Let $\N$ be the nerve complex of $X$. Further, suppose that $A_{\alpha}$ is collapsible for each $\alpha \in \N$. Then, $$H_{\#}(X) \cong H_{\#}(\N).$$
	\end{thm}
	
	\begin{proof}
		Since, $A_\alpha$ is collapsible for each $\alpha \in \N$, therefore each $A_\alpha$ can be assigned a gradient vector field such that only one $0$-simplex is critical in $A_\alpha$ with respect to that gradient vector field. For each $\alpha \in \N$, let $\Aa$ be a copy of $A_\alpha$ as defined in Section~\ref{intro} and $\Wo_\alpha$ be a gradient vector field on $\Aa$ such that the only $\Wo_\alpha$-critical simplex in $\Aa$ is a $0$-simplex of $\Aa$.
		
		We show that, $H_{\#}^{\Lq}(X) \cong H_{\#}(\N)$. Then, it will follow from our theorem (\autoref{main}) that $H_{\#}(X) \cong H_{\#}^{\Lq}(X) \cong H_{\#}(\N)$. We note that, since here there are no $\oW$-critical $i$-simplices for $i \ge 1$, therefore, for each $q \ge 0$, the generators of $\Lq_q(X)$ are given by,
		
		$$ L_q(X)= \bigcup_{\alpha^{(q)} \in \N}\Crit_0^{\Wo_\alpha}(\Aa).$$
		
		We will show that the chain complexes $(C_{\#}(\N), \partial_{\#})$ and $\Lc$ are isomorphic. First, we construct an isomorphism $f_q$ between $C_q(\N)$ and $\Lq_q(X)$, for each $q \ge 0$. Then we will show that this is also a chain map, that is, the following diagram commutes, which will establish that $f_{\#}$ is a chain isomorphism.
		
		\[\begin{tikzcd}
			\cdots & {\Lq_{q+1}(X)} && {\Lq_{q}(X)} & \cdots \\
			\\
			\cdots & {C_{q+1}(\N)} && {C_{q}(\N)} & \cdots
			\arrow[from=1-1, to=1-2]
			\arrow["{\pL_{q+1}}", from=1-2, to=1-4]
			\arrow["{f_{q+1}}"', from=1-2, to=3-2]
			\arrow[from=1-4, to=1-5]
			\arrow["{f_{q}}"', from=1-4, to=3-4]
			\arrow[from=3-1, to=3-2]
			\arrow["{\partial_{q+1}}"', from=3-2, to=3-4]
			\arrow[from=3-4, to=3-5]
			\arrow[draw=none, from=1-2, to=3-4, "{\text{\Huge$\circlearrowleft$}}" description]
		\end{tikzcd}\]
		
		We define $f_q: \Lq_q(X) \rightarrow C_q(\N)$ on the set of generators. For $\alpha \in S_q(\N)$, let $\sigma_\alpha^{(0)} \in L_q(X)$, that is, $\Crit^{\Wo_\alpha}(\Aa) = \{\sigma_\alpha^{(0)}\}$. Now we define $f_q$ as,
		$$ \sigma_\alpha \longmapsto \alpha.$$
		Let us now define $g_q: C_q(\N) \rightarrow \Lq_q(X)$ as follows. Let $\alpha \in S_q(\N)$. Then we define $g_q$ as,
		$$ \alpha \longmapsto \sigma_\alpha. $$
		Both $f_q$ and $g_q$ are defined on the generators whence they are linearly extended to their domains. The well-definedness of $g_q$ follows from the fact that each $\Aa$, being collapsible, has a unique $\Wo_\alpha$-critical simplex. Again, following the same line of reasoning, it can be verified that $f_q \circ g_q = \id$ and $g_q \circ f_q = \id$ on the set of generators and hence on their respective domains.
		
		Now, we prove, $f_q \circ \pL_{q+1} = \partial_{q+1} \circ f_{q+1}.$
		
		From \autoref{gt1}, we note that for any $\sa^{(0)} \in L_{q+1}(X)$, all the generalised trajectories beginning from $\sa$ terminate at some $\sigma'^{(j)}_{\alpha}\in L_q(X)$, where $\alpha \subseteq \alpha_0$. Since, there are no $\oW$-critical $j$-simplices for $j \ge 1$, therefore $j =0$ and hence by \autoref{gt1}\ref{obs2}, $\alpha <_f \az$. Now, we can use \autoref{0gt} to conclude that the only possible generalised trajectories from $\sa$ are given by $\{P_{\sa, \alpha} \mid \alpha <_f \az\}$, where $w_G(P_{\sa, \alpha})= \langle \az, \alpha \rangle$ and the terminal simplex of $P_{\sa, \alpha}$, $\mathbf{t}(\sa, \alpha)$ is the unique $\Wo_{\alpha}$-critical simplex in $\bar{A}_{\alpha}$ (since $\bar{A}_{\alpha}$ is collapsible), for each $\alpha <_f \ao$. For each $\sa \in L_{q+1}$ and $\alpha <_f \az$, let us rename $\mathbf{t}(\sa, \alpha)$ as $\sigma'_\alpha$.
	
		Now, let $\sa^{(0)} \in L_{q+1}(X)$, where $\alpha_0^{(q+1)} \in \N$. Therefore, from the property of the generalised trajectories initiating from $\sa$, we can write,
		\begin{equation*}
			\begin{aligned}
				f_q(\pL_{q+1}(\sa)) &= f_q\left(\sum_{\alpha: \alpha <_f \az} w_G(P_{\sa, \alpha})\sigma_\alpha'\right)\\
				&= \sum_{\alpha: \alpha <_f \az}\langle \az, \alpha \rangle f_q(\sigma'_\alpha), \text{ (since, $w_G(P_{\sa, \alpha})= \langle \alpha_0, \alpha\rangle$)}\\
				&= \sum_{\alpha: \alpha <_f \az}\langle \az, \alpha \rangle\alpha \\& = \partial_{q+1}(\az)= \partial_{q+1}(f_{q+1}(\sa)).
			\end{aligned}
		\end{equation*}
		Hence, the result follows.
	\end{proof}
	We have so far used the gradient vector fields on the subcomplexes $A_\alpha$ to construct our chain complex $\Lc$. However, up until now, we have not at all exploited the structure of the nerve complex. Now, we use a gradient vector field on $\N$ to further reduce this chain complex.
	\section{Further reduction of the chain complex $\Lc$}\label{red}
	We recall from Subsection~\ref{alg} that a suitable acyclic pairing can be assigned to any chain complex to reduce it to a chain complex with isomorphic homology groups. We will use this to reduce our chain complex, $\Lc$, further. As in Subsection~\ref{alg}, this chain complex, $\Lc$ can also be represented as a ranked poset $(L, \le)$. Here, we observe that, for $y \in L_{q+1}$, $x \in L_q$,  $\langle y, x\rangle_L=\sum_{P \in \GT(y,x)}w_G(P)$ (coefficient of $x$ in $\pL_{q+1}(y)$), and hence, $x \le y$ if and only if, $\sum_{P \in \GT(y,x)}w_G(P) \ne 0$. For each $\alpha^{(q)} \in \N$, let $L_q^{\alpha}$ denote the set $\{\si \mid \si^{(0)} \in L_q \}$.
	
	Let $\VN$ be a given gradient vector field on $\N$. We now construct a pairing $\p$ on $\Lc$ as follows. Let $(\alpha^{(q)}, \beta^{(q+1)}) \in \VN$. We order the simplices of $L_{q+1}^{\beta}$ as $\tb, \dots ,(\tau_k)_{\beta}$. We choose $(\tau_1)_{\beta}$. From \autoref{0gt}, we conclude that, for each $\alpha <_f \beta$, there exists a unique generalised trajectory $P_{\tb, \alpha}$ from $\tb$ which ends at $\mathbf{t}(\tb, \alpha) \in L_q^\alpha$. Also from \autoref{0gt}, we know that $w_G(P_{\tb, \alpha})= \langle \beta, \alpha\rangle = \pm 1$. So, we pair $\tb$ with $\mathbf{t}(\tb, \alpha)$, that is, $(\mathbf{t}(\tb, \alpha), \tb)\in \p$. Thus $\p= \{(\mathbf{t}(\tb, \alpha), \tb) \mid (\alpha, \beta) \in \VN\}$. Since, $L^{\alpha}_{q} \cup L^{\beta}_{q+1}$ are mutually disjoint sets for each $(\alpha^{(q)}, \beta^{(q+1)}) \in \VN$, therefore $\z$ is well-defined.
	
	The next proposition characterises the $\z$-paths in $\Lc$.
	
	\begin{prop}\label{path}
		A typical $\p$-path in $\Lc$ initiating from some $\tau_0' \in L_{q+1}$ of the form,
		$$ \tau_0', \sigma_1', \tau_1',  \dots, \tau_k', \sigma_{k+1}',$$
		where $\sigma_i'= (\sigma_i)_{\alpha_i} \in L_q$, for each $i \in [k+1]$, $\tau_i'=(\tau_i)_{\beta_i} \in L_{q+1}$ for each $i \in \{0, \dots, k\}$, satisfies the following properties,
		\begin{enumerate}[label=(\alph*)]
			\item $\dim(\sigma_i')=\dim(\tau_i')=0$ for each $i \in [k]$.\label{prop0}
			\item  $(\alpha_i^{(q)}, \beta_i^{(q+1)}) \in \VN$ for each $i \in [k]$, \label{prop1}  
			\item $\alpha_i \ne \alpha_{i+1} <_f \beta_i$ for each $i \in [k-1]$, \label{prop2}
			\item For $\tau_0'\in L_{q+1}$, either $\dim(\tau_0')=0$, $\alpha_1 <_f \bz$, or $\dim(\tau_0')=1$, $\bz=\ao$,\label{prop3}
		\end{enumerate}	
	\end{prop}
	\begin{proof}
		
		The first two conditions (\ref{prop0} and \ref{prop1}) follow from the construction of $\z$. For each $i \in [k-1]$, since $\sigma_{i+1}'=(\sigma_{i+1})^{(0)}_{\alpha_{i+1}} \le \tau_i'=(\tau_i)_{\beta_i}^{(0)}$, therefore, there must exist a generalised trajectory from $\tau_i'$ to $\sigma_{i+1}'$. Hence, using \autoref{gt1}\ref{obs2}, we infer that $\alpha_{i+1} <_f \beta_i$. We observe from the construction of $\z$ that each $(\alpha, \beta) \in \VN$ corresponds to a unique pair in $\z$. Therefore, if $\alpha_i=\alpha_{i+1}$ for some $i \in [k-1]$, then $\sigma_i'=\sigma_{i+1}'$, which is not possible in a $\z$-path. So $\alpha_i \ne \alpha_{i+1}$ for each $i \in [k-1]$. 
		
		Since $\sigma_1'=(\sigma_1)^{(0)}_{\alpha_1} \le \tau_0'$, therefore there exists a generalised trajectory from $\tau_0'$ to $\sigma_1'$. So, from \autoref{gt1}, we infer that $\dim(\sigma_1') \ge \dim(\tau_0')-1$, which means $\dim(\tau_0') \le 1$. Now, if $\dim(\tau_0')=0$, then from \autoref{gt1}\ref{obs2}, we have $\ao <_f \bz$. However if $\dim(\tau_0')=1$, then from \autoref{gt1}\ref{obs1}, we have $\bz=\ao$. 
	\end{proof}
	We obtain the following observation as a direct consequence of this proposition (\autoref{path}\ref{prop1} and \ref{prop2}).
	\begin{obs}\label{nerve}
		Given any $\z$-path $P$ in $\Lc$ of the form,  $$ P:~\tau_0', \sigma_1', \tau_1',  \dots, \tau_k', \sigma_{k+1}',$$
		where $\sigma_i'= (\sigma_i)_{\alpha_i} \in L_q$, for each $i \in [k+1]$, $\tau_i'=(\tau_i)_{\beta_i} \in L_{q+1}$ for each $i \in \{0, \dots, k\}$, there exists a corresponding \emph{nerve path} of the form,
		\begin{equation*}\label{nv}
			\alpha_1^{(q)}, \beta_1^{(q+1)}, \dots , \alpha_k^{(q)}, \beta_k^{(q+1)}, \tag{i}
		\end{equation*}
		where $(\alpha_i, \beta_i) \in \VN$, for each $i \in [k]$, $\alpha_i \ne \alpha_{i+1} <_f \beta_i$ for each $i \in [k-1]$.
	\end{obs}
	For any $\z$-path $P$, we denote the corresponding nerve path as $P_{\mathcal{N}}$. The weight of this nerve path (refer to (\ref{nv})),
	$$ w(P_{\mathcal{N}}):= -\left(\prod_{i=1}^{k-1}(-1) \langle \beta_i, \alpha_i\rangle \langle \beta_i, \alpha_{i+1}\rangle\right)\langle\beta_k, \alpha_k\rangle.$$
	As in Subsection~\ref{alg}, we denote the set of all $\p$-paths from $\tau \in L_{q+1}$ and ending at $\sigma \in L_q$ as $\mathcal{P}(\tau, \sigma)$. Now, we show that $\p$ is an acyclic pairing on $\Lc$.
	\begin{prop}\label{acyc}
		$\p$ is an acyclic pairing.
	\end{prop}
	
	\begin{proof}
		Let $P$ be a $\p$-path in $\Lc$, initiating from $\tau_0' \in L_{q+1}$ as follows,
		$$ P:~\tau_0', \sigma_1', \tau_1',  \dots, \tau_k', \sigma_{k+1}',$$
		Now, if $\tau_0'=\tau_r'$ for some $r > 1$, then we use \autoref{nerve} to obtain a closed $\VN$-trajectory in $\N$, corresponding to $P$, as follows.
		$$\beta_0^{(q+1)}, \alpha_1^{(q)}, \beta_1^{(q+1)}, \dots , \alpha_r^{(q)}, \beta_r^{(q+1)}~(=\beta_0),$$
		where, $(\alpha_i, \beta_i) \in \VN$ for each $i \in [r]$, which is a contradiction. Hence, $\p$ is acyclic.
	\end{proof}
	
	Now, for each $q \ge 0$, the collection of all elements of $\CrqZ$, is given by,
	\begin{enumerate}[label=(\alph*)]
		\item $\{\sigma^{(i)}_{\alpha} \mid i \ge 1, \alpha^{(q-i)} \in \N\}$,
		\item $ \{(\tau_j)^{(0)}_{\beta} \in L_q^{\beta} \mid (\alpha^{(q-1)}, \beta^{(q)}) \in \VN, j \ne 1\} \cup \{\sigma^{(0)}_{\alpha} \in L_q^{\alpha} \mid (\alpha^{(q)}, \beta^{(q+1)}) \in \VN, \sigma^{(0)}_{\alpha} \ne \mathbf{t}((\tau_1)_{\beta}, \alpha), (\tau_1)_{\beta} \in L_{q+1}^{\beta} \}$ (i.e., the unpaired simplices in $L_q^\alpha$ for each $\alpha$ which appears in $\VN$),
		\item $\{\sigma^{(0)}_\alpha \mid \alpha \in \Crit^{\VN}_q(\N)\}$.
	\end{enumerate}
	
	Now, for each $q \ge 0$, $\Lq^{\z}_q$ is the free module generated by $\CrqZ$ over $\Z$. Let us rename $\CrqZ$ as $\Lz_q$.
	\begin{prop}
		Let $\tau_0' \in \Lz_{q+1}$, $\sigma_{k+1}' \in \Lz_q$ and let $P \in \mathcal{P}(\tau_0', \sigma_{k+1}')$ be as follows,
		$$ P:~\tau_0', \sigma_1', \tau_1',  \dots, \tau_k', \sigma_{k+1}',$$
		where $\sigma_i'= (\sigma_i)_{\alpha_i} \in L_q$, for each $i \in [k+1]$, $\tau_i'=(\tau_i)_{\beta_i} \in L_{q+1}$ for each $i \in \{0, \dots, k\}$. Let $P_{\mathcal{N}}$ be the corresponding nerve path. Then,
		$$w_L(P)= \left(\sum_{P\in \mathcal{P}(\tau_0', \sigma_1')}w_G(P)\right)\left(\sum_{P\in \mathcal{P}(\tau_k', \sigma_{k+1}')}w_G(P)\right)w(P_{\mathcal{N}}).$$
	\end{prop}
	\begin{proof}
		Since $(\sigma_i', \tau_i') \in \z$ for each $i \in [k]$, therefore from the construction of $\z$, there exists a unique generalised trajectory $P_{\tau_i', \alpha_i}$ from $\tau_i'$ to $\sigma_i'$, and further, $(\alpha_i, \beta_i) \in \VN$. Also, we know from \autoref{path}\ref{prop0}, that for each $i \in [k]$, $\dim(\sigma_i')=\dim(\tau_i')=0$. So, from \autoref{0gt}, we have, $w_G(P_{\tau_i', \alpha_i})= \langle \beta_i, \alpha_i\rangle$ for each $i \in [k]$. Again, from \autoref{path}\ref{prop0}, we know that $\dim(\tau_{i}')= \dim(\sigma_{i+1}')=0$, and from \autoref{path}\ref{prop2}, we have $\alpha_{i+1} <_f \beta_i$, for each $i \in [k-1]$. Since, $\sigma_{i+1}' \le \tau_i'$ for each $i \in [k-1]$, therefore, there exists a generalised trajectory from $\tau_i'$ to $\sigma_{i+1}'$. Now, we use \autoref{0gt} to deduce that there exists a unique generalised trajectory $P_{\tau_{i}', \alpha_{i+1}}$ from $\tau_i'$ to $\sigma_{i+1}'$ and $w_G(P_{\tau_i', \alpha_{i+1}})= \langle \beta_i, \alpha_{i+1} \rangle$. Thus,
		
		\begin{equation*}
			\begin{aligned}
				w_L(P) &= \left(\prod_{i=0}^{k-1}(-1) \langle \tau_i', \sigma_{i+1}'\rangle_L\langle \tau_{i+1}', \sigma_{i+1}'\rangle_L \right)\langle \tau_{k}', \sigma_{k+1}'\rangle_L\\	
				&=\langle \tau_0', \sigma_1'\rangle_L \left(\prod_{i=1}^{k-1}(-1)\langle \tau_i', \sigma_i'\rangle_L \langle \tau_{i}', \sigma_{i+1}' \rangle_L\right)(-1)\langle\tau_k', \sigma_{k}'\rangle_L \langle \tau_k', \sigma_{k+1}'\rangle_L\\
				&= \left(\sum_{P\in \mathcal{P}(\tau_0', \sigma_1')}w_G(P)\right)\left(\prod_{i=1}^{k-1}(-1) \langle \beta_i, \alpha_i\rangle \langle \beta_i, \alpha_{i+1}\rangle\right)(-1)\langle\beta_k, \alpha_k\rangle\left(\sum_{P\in \mathcal{P}(\tau_k', \sigma_{k+1}')}w_G(P)\right)\\
				&=\left(\sum_{P\in \mathcal{P}(\tau_0', \sigma_1')}w_G(P)\right)\left(\sum_{P\in \mathcal{P}(\tau_k', \sigma_{k+1}')}w_G(P)\right)w(P_{\mathcal{N}}).
			\end{aligned}
		\end{equation*}
	\end{proof}
	
	Now, from Subsection~\ref{alg}, the boundary map $\partial^{\z}_{q+1}: \Lq^{\z}_{q+1}(X) \rightarrow \Lq^{\z}_q(X)$ is defined for each $q \ge 0$, as follows. For each $\tau \in L^{\z}_{q+1}(X)$,
	
	$$ \partial^{\z}_{q+1}(\tau):= \sum_{\sigma \in L^{\z}_q(X)}\left(\sum_{P \in \mathcal{P}(\tau, \sigma)}w_L(P)\right)\sigma. $$
	
	The following result is due to \autoref{koz}.
	
	\begin{thm}\label{rt}
		$(\Lq^{\z}_{\#}(X), \partial^{\z}_{\#})$ is a chain complex, and $H^{\Lq^{\z}}_{\#}(X) \cong H^{\Lq}_{\#}(X)$.
	\end{thm}
	
	\section{Computational aspects of the combinatorial nerve theorem}\label{algo}
	In this section, we provide an algorithm for the computation of homology groups of a simplicial complex using our theorems (\autoref{main} and \autoref{rt}) . Thereafter, we compute the homology groups of the $5 \times 5$ chessboard complex using our theorem, as an illustrative example.
	\subsection{An algorithm for computation of the homology groups of a simplicial complex using the combinatorial nerve theorem}
	Let $X$ be a simplicial complex and $A_1, A_2, \ldots A_k$ be subcomplexes of $X$ such that $X=\cup_{i=1}^{k}A_i$. We recall that, for each $\alpha \in \N$, $\Aa$ is a copy of $A_{\alpha}$. For the sake of convenience, in this subsection we denote an element $x_\alpha$ of $\Aa$ as $\ax$.
	
	We denote the directed Hasse diagram on a ranked poset ($P, \leq$) by $H_P$. Thus $H_P$ is a directed graph with vertex set $V(H_P)=P$, and for $x$ and $y$ in $V(H_P)$, there is a directed edge from $y$ to $x$ if and only if $x \leq y$ and $\operatorname{rank}(y)=\operatorname{rank}(x)+1$.
	
	For any $\ax \in \Aa$, we define $\operatorname{rank}$($\ax$)$ :=\dim(\alpha)+\dim(x)$.  
	Now, let $C:=\bigcup \limits_{\alpha \in \N} \Aa$ and let us consider the ranked poset ($C, \leq_1 $), where the partial order is given as follows.
	$$\ax \leq_1 \by \text{ if and only if }\alpha \subseteq \beta, x \subseteq y \text{ and }\operatorname{rank}(\by)=\operatorname{rank}(\ax)+1.$$
	
	\begin{algorithm}[H]\footnotesize
		\caption{Algorithm for computation of boundary matrices
		}\label{alg:boundary first}
		\begin{algorithmic}[1] 
			\Statex \textbf{Input:} \begin{itemize}
				\item simplicial complex $X$ without $\emptyset$,
				
				\item subcomplexes $\{A_{i}\}_{i=1}^{k}$ such that $X = \bigcup \limits_{i=1}^{k} A_{i}$, with $\Aa=\{[\alpha;x]\mid x \in A_{\alpha}\}$
	
			\end{itemize}
			\Statex \textbf{Output:} 
			for critical $\ax$ and $\by$ with $\operatorname{rank}(\by)=\operatorname{rank}(\ax)+1$
			\begin{itemize}
				\item 
				all generalised trajectories from $\by$ to $\ax$,
				\item  $a(\by,\ax)$ 
				\Comment{$a(\by, \ax)$ is the coefficient of $\ax$ in the boundary of $\by$
				}
			\end{itemize}

			\smallskip
		
			\State $C \gets \bigcup \limits_{\alpha \in \N} \Aa$
			
			\State for $\alpha \in \N,$ $\oW \gets \operatorname{GET}$(gradient vector field on $\Aa$)

		\State $N \gets \max\{\operatorname{rank}(\ax) \mid \ax \in C\}$

		\State $D \gets H_C$ with direction of edges reversed along each $\oW$

		\State $\ell_1 \gets $ a topological ordering  on $V(D)$
		\State for $i = 0, \dots ,N$, $C_{i} \gets$ set of all $\oW$-critical elements of $\operatorname{rank}(i)$
		in $C$ 
		
		\State for $i=1, \dots, N$, $\by \in C_i$, and $\ax \in C_{i-1}$, $a(\by,\ax) \gets 0$
		
	\end{algorithmic}
\end{algorithm}
\begin{algorithm}[H]\footnotesize
	\ContinuedFloat
	\caption{Algorithm for computation of boundary matrices
		(CONT.)}\label{alg:boundary 1st 2nd page}
	\begin{algorithmic}[1]
		\setcounter{ALG@line}{7}
		\For{$i=0, \dots ,(N-1)$, \textbf{and} each $[\alpha;x]\in C_i$}
		\State $m \gets 1$
		\State $r \gets 1$
		\State $j \gets 0$
		\State \texttt{CurrentSimplex} $\gets  x$
		\State $Q \gets $ an empty stack

		\State $\operatorname{PUSH} (Q,[\alpha;x])$  
		\State \texttt{Power} $\gets 0$

		\While{$j < \ell_1([\alpha;\texttt{CurrentSimplex}])$}
		\State $[\beta;y] \gets x \in V(D)$ such that  $\ell_1(x)=j$
		\If{there is an edge from $[\beta;y]$ to $[\alpha;\texttt{CurrentSimplex}]$}
		\If{$\operatorname{rank}([\alpha;\texttt{CurrentSimplex}])+\operatorname{rank}(\by)=2i+1$}
	
		\If{$\alpha \neq \beta$} 
		\State \texttt{AddPower} $\gets \dim(y)$
		\EndIf
	
		\State  $\operatorname{PUSH}(Q,[\beta;y])$

		\State \texttt{Power} $\gets $ \texttt{Power} + \texttt{AddPower}
	
		\State  $m \gets m \times$ incidence number between \texttt{CurrentSimplex} and $y$

		\State r $\gets  r \times$ incidence number between $\alpha$ and $\beta$
	
		\State \texttt{CurrentSimplex} $\gets y$
		
		\State $j \gets 0$
		\Else
		\State $j \gets j+1$
		\EndIf
		\Else
		\State $j \gets j+1$
		\EndIf
		
		\If{$j=\ell_1([\beta;\texttt{CurrentSimplex}])$}
		\State $z \gets$ \texttt{CurrentSimplex}
		\If{$[\beta;z] \in C_{i+1}$}
		\State $\operatorname{PRINT}(\text{reversed }  Q)$\Comment{this is a (maximal) generalised trajectory from $[\beta;z]$ to $[\alpha;x]$}
	
		\State $a([\beta;z],[\alpha;x]) \gets
		a([\beta;z],[\alpha;x])
		+ (-1)^{\left(\frac{|Q|}{2}-1\right)+\texttt{Power}}
		\times m \times r$

		\EndIf
		\State $j \gets j +1$
	
		\State $Q\gets\operatorname{POP}{(Q)}$
	
		\State $[\omega;\texttt{CurrentSimplex}] \gets \operatorname{\operatorname{TOP}(Q)}$
		
		\If {$\beta \neq \omega$} 
		\State \texttt{RemovePower} $\gets \dim(z)$
		\EndIf
		\State \texttt{Power} $\gets $ \texttt{Power} $-$ \texttt{RemovePower}
		\State $m \gets m \times$incidence number between $z$ and \texttt{CurrentSimplex}
		\State r $\gets $ r $\times$ incidence number between $\beta$ and $\omega$

		\EndIf

		\EndWhile
		\EndFor
		\State for $i=1, \dots ,N$, $\by \in C_i$, and $\ax \in C_{i-1}$, $\operatorname{PRINT}(a(\by,\ax)$)
	\end{algorithmic}
\end{algorithm}

We define $C':= \bigcup_{i=0}^NC_i$ (critical simplices obtained in Algorithm~\ref{alg:boundary first}). 
Next, we consider the ranked poset ($C', \leq_2$), where \[\ax \leq_2 \by \text{ if and only if }   a(\by, \ax) \neq 0.\]

\begin{algorithm}[H]\footnotesize
	\caption{Algorithm for computation of boundary matrices of the reduced chain complex
	}\label{alg reduction}
	\begin{algorithmic}[1] 
		\Statex \textbf{Input:} \begin{itemize}
			
			\item $C':= \bigcup_{i=0}^NC_i$ (from Algorithm~\ref{alg:boundary first})
			
			\item $a(\by,\ax)$, for $\ax, \by \in C'$ with $\operatorname{rank}(\by)=\operatorname{rank}(\ax)+1$ (the output of Algorithm~\ref{alg:boundary first}),
			\item topological order $\ell_1$ on $V(D)$ (in Algorithm~\ref{alg:boundary first}))
			\item gradient vector field $\VN$ on $\N$
			
		\end{itemize}
		\Statex \textbf{Output:} 
		for critical $\ax$ and $\by$ with $\operatorname{rank}(\by)=\operatorname{rank}(\ax)+1$
		\begin{itemize}
			\item all paths from $\by$ to $\ax$,
			\item $b(\by,\ax)$ 
			\Comment{$b(\by, \ax)$ is the coefficient of $\ax$ in the boundary of $\by$
			}
		\end{itemize}
		\State $N \gets \max\{\operatorname{rank}(\ax) \mid \ax \in C'\}$
		\State $\z \gets$ set of all $\{(\ax, \by) \mid \ax, \by \in C'\}$ such that
		\begin{enumerate}[label=(\roman*)]
			\item  $(\alpha, \beta) \in \VN$,
			\item  $\dim(y)=0$ and $\ell_1(\by) < \ell_1([\beta;y'])$, for each $[\beta;y'] \in C'$ with $\dim(y')=0$,
			\item  $\dim(x)=0$ such that $\ax \leq_2 \by$
		\end{enumerate}
		\State $D' \gets H_{C'}$ with direction of edges reversed along $\z$

		\State $\ell_2 \gets$ a topological ordering on $V(D')$
		
		\State for $i = 0, \dots ,N$, $C_{i}' \gets$ set of all elements of $\operatorname{rank}(i)$ in $C'$ that are unpaired (critical) with respect to $\z$
		
		\State for $i=1, \dots ,N$, $\by \in C_i'$, and $\ax \in C_{i-1}'$, $b(\by,\ax) \gets 0$
		
		\For{$i=0, \dots ,(N-1)$, and each $[\alpha;x]\in C_i'$}
		\State $m \gets 1$
		\State $j \gets 0$
		
		\State \texttt{CurrentSimplex} $\gets x$
		\State $Q \gets $ an empty stack

		\State $\operatorname{PUSH}(Q,[\alpha;x])$ 
		
		\While{$j < \ell_2([\alpha;\texttt{CurrentSimplex}])$}
		\State $[\beta;y] \gets x \in V(D')$ such that $\ell_2(x)=j$ 
		\If{there is an edge from $[\beta;y]$ to $[\alpha;\texttt{CurrentSimplex}]$}
		\If{$\operatorname{rank}([\alpha;\texttt{CurrentSimplex}])+\operatorname{rank}(\by)=2i+1$}
		
		\State  $\operatorname{PUSH}(Q, [\beta;y])$

		\State  $m \gets m \times a([\alpha;\texttt{CurrentSimplex}],[\beta;y])$
		\State \texttt{CurrentSimplex} $\gets y$
		
		\State $j \gets 0$
		\Else
		\State $j \gets j+1$
		\EndIf
		\Else
		\State $j \gets j+1$
		\EndIf
		
		\If{$j=\ell_2([\beta;\texttt{CurrentSimplex}])$}
		\State $z \gets \texttt{CurrentSimplex}$
		\If{$[\beta;z] \in C_{i+1}'$}
		\State $\operatorname{PRINT}$(reversed $Q$)\Comment{this is a (maximal) path from $[\beta;z]$ to $[\alpha;x]$}
		
		\State  $b(\by,\ax) \gets b(\by,\ax)+ (-1)^{\left(\frac{|Q|}{2}-1\right)}
		\times m$
		
		\EndIf
		\State $j \gets j +1$

		\State $Q \gets \operatorname{POP}{(Q)}$
	
		\State $[\omega;\texttt{CurrentSimplex}] \gets {\operatorname{TOP}(Q)}$

		\State $m \gets m \div a([\beta;z],[\omega;\texttt{CurrentSimplex}])$
		\EndIf

		\EndWhile
		\EndFor
		\State for $i=1, \dots ,N$, $\by \in C_i'$, and $\ax \in C_{i-1}'$, $\operatorname{PRINT}(b(\by,\ax))$ 
	\end{algorithmic}
\end{algorithm}

\subsection{Computation of the homology groups of the $5 \times 5$ chessboard complex using the combinatorial nerve theorem}\label{comp}

Now we compute the homology groups of the $5 \times 5$ chessboard complex (i.e., the matching complex of the complete bipartite graph $\K$) using our theorems (\autoref{main} and \autoref{rt}), as an illustrative example. We recall that the \emph{matching complex} of a graph $G$ is the collection of all matchings in $G$ which we denote as $\M(G)$. It is worth noting here that the homology groups of $\M(\K)$ have been computed previously using an ingenious but indirect technique, with the help of Bouc's long exact sequence as well as the homology groups of $\M(K_{10})$(see \cite{Shareshian2007}). However, here we offer a direct and algorithmic approach to compute these homology groups efficiently.

The edge set  of $\K$ is $[5] \times [5]$. Now in order to apply our theorem, first we choose the following subcomplexes of $\M(\K)$.
$$A_{i}=\{\sigma\mid \sigma \sqcup \{(1,i)\} \in \M(\K)\}, \text{ for each } i \in [5].$$
We observe that $\cup_{i=1}^5A_i= \M(\K)$. Next, we consider the nerve of $\M(\K)$, $\mathcal{N}(\M(\K))$. We observe that $A_\alpha \ne \emptyset$, for each $\alpha \subseteq [5]$, $\alpha\ne [5]$. Therefore, $\mathcal{N}(\M(\K))= \{\alpha \subseteq [5] \mid \alpha \ne [5]\}$.	Now, for each $\alpha \in \mathcal{N}(\M(\K))$, we construct a gradient vector field $\oW$ on $\Aa$, as follows.\\

\textbf{Case I:} Let $\dim(\alpha)=0$.\label{ex:case1}

Therefore, $\alpha= \{i\}$, $i \in [5]$. We pair $\si (\neq \emptyset)$ with $\si \sqcup \{(1,i)\}_{\alpha}$, (i.e., $(\si, \si \sqcup \{(1,i)\}_{\alpha})\in \oW$). We can verify that, for each $\alpha$, $\oW$ is well-defined and acyclic.  Hence, the collection of all $\oW$-critical simplices is given by, $\{\{(1,i)_\alpha\} \mid i \in [5], \alpha=\{i\}\}$. Therefore, $$\sum \limits_{\alpha^{(0)} \in \mathcal{N}(\M(\K))}|\Crit_0^{\oW}(\Aa)|=5.$$

\textbf{Case II:} Let $\dim(\alpha)=1$.\label{ex:case2}

We observe here that each $\Aa$ is a copy of $\M(K_{4,3})$(where the edge set of $K_{4,3}$ is $[4] \times [3]$) which differ by a labelling of vertices of $K_{4,3}$. So, it suffices to construct a gradient vector field $\vv$ on $\M(K_{4,3})$.

Now, we construct a gradient vector field $\vv$ on $\M(K_{4,3})$ sequentially in $10$ steps, as follows. Let $\U_i$ denote the collection of unpaired simplices after the $i^{th}$ step.\\
\textbf{Step I:}
\begin{align*}
	\text{Let }\vv^{(1,1)}:=\{(\sigma, \sigma \sqcup \{(1,1)\}) \mid  \sigma \in \M(\kk) \text{ and } \sigma \sqcup \{(1,1)\} \in \M(\kk)\}.
\end{align*}
We represent the pairings in $\vv^{(1,1)}$ in \autoref{matching 01 (11) K43} and \autoref{matching 12 (11) K43}. The vertices are numbered from top to bottom. We have labelled the vertices in the first pairing in \autoref{matching 01 (11) K43}  for a reader's convenience. However, for simplicity, we will not explicitly label the vertices of $K_{4,3}$ hereafter.

\begin{figure}[htbp]
	\begin{center}
		
		\begin{tikzpicture}[transform shape, scale = 0.5]
			\node[circle,fill,inner sep=3pt, label={[font=\huge]left:$4$}] (a) at (0,0) {};
			\node[circle,fill,inner sep=3pt, label={[font=\huge]left:$3$}] (b) at (0,1) {};
			\node[circle,fill,inner sep=3pt, label={[font=\huge]left:$2$}] (c) at (0,2) {};
			\node[circle,fill,inner sep=3pt,color=black, label={[font=\huge]left:$1$}] (d) at (0,3) {};

			\node[circle,fill,inner sep=3pt, label={[font=\huge]right:$3$}] (e) at (2,0) {};
			\node[circle,fill,inner sep=3pt, label={[font=\huge]right:$2$}] (f) at (2,1.5) {};
			\node[circle,fill,inner sep=3pt, label={[font=\huge]right:$1$}] (g) at (2,3) {};

			\draw[line width=2pt, color=black] (d) -- (g);
			\draw (c) -- (f);
		
			\node[circle,fill,inner sep=3pt] (a1) at (5,0) {};
			\node[circle,fill,inner sep=3pt] (b1) at (5,1) {};
			\node[circle,fill,inner sep=3pt] (c1) at (5,2) {};
			\node[circle,fill,inner sep=3pt] (d1) at (5,3) {};

			\node[circle,fill,inner sep=3pt] (e1) at (7,0) {};
			\node[circle,fill,inner sep=3pt] (f1) at (7,1.5) {};
			\node[circle,fill,inner sep=3pt] (g1) at (7,3) {};

			\draw[line width=2pt, color=black] (d1) -- (g1);
			\draw (c1) -- (e1);

			\node[circle,fill,inner sep=3pt] (a2) at (10,0) {};
			\node[circle,fill,inner sep=3pt] (b2) at (10,1) {};
			\node[circle,fill,inner sep=3pt] (c2) at (10,2) {};
			\node[circle,fill,inner sep=3pt] (d2) at (10,3) {};

			\node[circle,fill,inner sep=3pt ] (e2) at (12,0) {};
			\node[circle,fill,inner sep=3pt ] (f2) at (12,1.5) {};
			\node[circle,fill,inner sep=3pt,color=black] (g2) at (12,3) {};

			\draw[ line width=2pt, color=black] (d2) -- (g2);
			\draw (b2) -- (f2);

			\node[circle,fill,inner sep=3pt] (a3) at (15,0) {};
			\node[circle,fill,inner sep=3pt] (b3) at (15,1) {};
			\node[circle,fill,inner sep=3pt] (c3) at (15,2) {};
			\node[circle,fill,inner sep=3pt] (d3) at (15,3) {};

			\node[circle,fill,inner sep=3pt] (e3) at (17,0) {};
			\node[circle,fill,inner sep=3pt] (f3) at (17,1.5) {};
			\node[circle,fill,inner sep=3pt] (g3) at (17,3) {};

			\draw[ line width=2pt, color=black] (d3) -- (g3);
			\draw (b3) -- (e3);

			\node[circle,fill,inner sep=3pt] (a4) at (20,0) {};
			\node[circle,fill,inner sep=3pt] (b4) at (20,1) {};
			\node[circle,fill,inner sep=3pt] (c4) at (20,2) {};
			\node[circle,fill,inner sep=3pt] (d4) at (20,3) {};

			\node[circle,fill,inner sep=3pt] (e4) at (22,0) {};
			\node[circle,fill,inner sep=3pt] (f4) at (22,1.5) {};
			\node[circle,fill,inner sep=3pt] (g4) at (22,3) {};

			\draw[line width=2.2pt, color=black] (d4) -- (g4);
			\draw (a4) -- (f4);
			
			\node[circle,fill,inner sep=3pt ] (a5) at (25,0) {};
			\node[circle,fill,inner sep=3pt ] (b5) at (25,1) {};
			\node[circle,fill,inner sep=3pt ] (c5) at (25,2) {};
			\node[circle,fill,inner sep=3pt,color=black] (d5) at (25,3) {};

			\node[circle,fill,inner sep=3pt ] (e5) at (27,0) {};
			\node[circle,fill,inner sep=3pt ] (f5) at (27,1.5) {};
			\node[circle,fill,inner sep=3pt,color=black] (g5) at (27,3) {};

			\draw[ line width=2pt, color=black] (d5) -- (g5);
			\draw (a5) -- (e5);
	
			\node[circle,fill,inner sep=3pt, label={[font=\huge]left:$4$}  ] (a7) at (0,-6) {};
			\node[circle,fill,inner sep=3pt, label={[font=\huge]left:$3$} ] (b7) at (0,-5) {};
			\node[circle,fill,inner sep=3pt, label={[font=\huge]left:$2$}  ] (c7) at (0,-4) {};
			\node[circle,fill,inner sep=3pt, label={[font=\huge]left:$1$} ] (d7) at (0,-3) {};

			\node[circle,fill,inner sep=3pt, label={[font=\huge]right:$3$}] (e7) at (2,-6) {};
			\node[circle,fill,inner sep=3pt, label={[font=\huge]right:$2$} ] (f7) at (2,-4.5) {};
			\node[circle,fill,inner sep=3pt, label={[font=\huge]right:$1$}] (g7) at (2,-3) {};

			\draw (c7) -- (f7);
			\draw[ >->] (1,-2)--(1,-1);
	
			\node[circle,fill,inner sep=3pt ] (a8) at (5,-6) {};
			\node[circle,fill,inner sep=3pt ] (b8) at (5,-5) {};
			\node[circle,fill,inner sep=3pt ] (c8) at (5,-4) {};
			\node[circle,fill,inner sep=3pt ] (d8) at (5,-3) {};

			\node[circle,fill,inner sep=3pt ] (e8) at (7,-6) {};
			\node[circle,fill,inner sep=3pt ] (f8) at (7,-4.5) {};
			\node[circle,fill,inner sep=3pt ] (g8) at (7,-3) {};

			\draw (c8) -- (e8);
			\draw[ >->] (6,-2)--(6,-1);

			\node[circle,fill,inner sep=3pt ] (a9) at (10,-6) {};
			\node[circle,fill,inner sep=3pt ] (b9) at (10,-5) {};
			\node[circle,fill,inner sep=3pt ] (c9) at (10,-4) {};
			\node[circle,fill,inner sep=3pt ] (d9) at (10,-3) {};

			\node[circle,fill,inner sep=3pt ] (e9) at (12,-6) {};
			\node[circle,fill,inner sep=3pt ] (f9) at (12,-4.5) {};
			\node[circle,fill,inner sep=3pt ] (g9) at (12,-3) {};

			\draw (b9) -- (f9);
			
			\draw[ >->] (11,-2)--(11,-1);

			\node[circle,fill,inner sep=3pt ] (a10) at (15,-6) {};
			\node[circle,fill,inner sep=3pt ] (b10) at (15,-5) {};
			\node[circle,fill,inner sep=3pt ] (c10) at (15,-4) {};
			\node[circle,fill,inner sep=3pt ] (d10) at (15,-3) {};

			\node[circle,fill,inner sep=3pt ] (e10) at (17,-6) {};
			\node[circle,fill,inner sep=3pt ] (f10) at (17,-4.5) {};
			\node[circle,fill,inner sep=3pt ] (g10) at (17,-3) {};

			\draw (b10) -- (e10);
			
			\draw[ >->] (16,-2)--(16,-1);

			\node[circle,fill,inner sep=3pt ] (a11) at (20,-6) {};
			\node[circle,fill,inner sep=3pt ] (b11) at (20,-5) {};
			\node[circle,fill,inner sep=3pt ] (c11) at (20,-4) {};
			\node[circle,fill,inner sep=3pt ] (d11) at (20,-3) {};

			\node[circle,fill,inner sep=3pt ] (e11) at (22,-6) {};
			\node[circle,fill,inner sep=3pt ] (f11) at (22,-4.5) {};
			\node[circle,fill,inner sep=3pt ] (g11) at (22,-3) {};

			\draw (a11) -- (f11);
			
			\draw[ >->] (21,-2)--(21,-1);

			\node[circle,fill,inner sep=3pt ] (a12) at (25,-6) {};
			\node[circle,fill,inner sep=3pt ] (b12) at (25,-5) {};
			\node[circle,fill,inner sep=3pt ] (c12) at (25,-4) {};
			\node[circle,fill,inner sep=3pt ] (d12) at (25,-3) {};

			\node[circle,fill,inner sep=3pt ] (e12) at (27,-6) {};
			\node[circle,fill,inner sep=3pt ] (f12) at (27,-4.5) {};
			\node[circle,fill,inner sep=3pt ] (g12) at (27,-3) {};

			\draw (a12) -- (e12);
			
			\draw[ >->] (26,-2)--(26,-1);
		\end{tikzpicture}
	\end{center}
	\caption{Pairings between $0$-simplices and $1$-simplices of $\M(K_{4,3})$ with respect to $\vv^{(1,1)}$. The leftmost diagram represents the pair $(\{(2,2)\}, \{(1,1),(2,2)\})$.}\label{matching 01 (11) K43}
\end{figure}
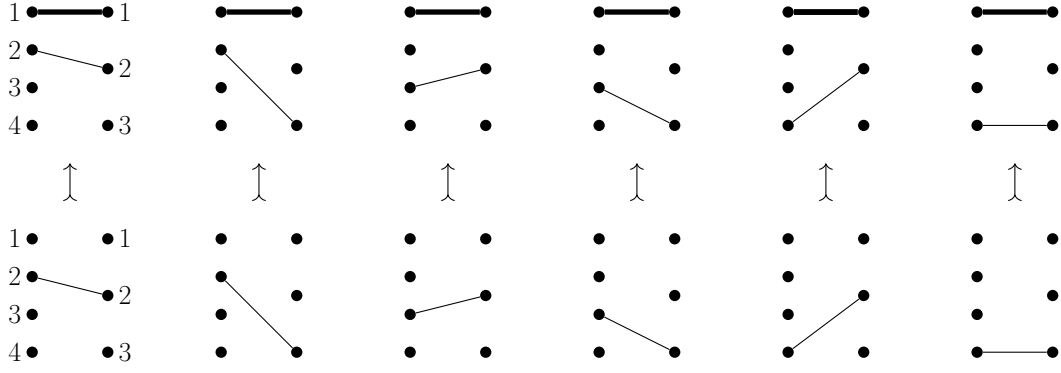

\begin{figure}[htbp]
	\begin{center}
		
		\begin{tikzpicture}[transform shape, scale = 0.5]
			\node[circle,fill,inner sep=3pt ] (a) at (0,0) {};
			\node[circle,fill,inner sep=3pt ] (b) at (0,1) {};
			\node[circle,fill,inner sep=3pt ] (c) at (0,2) {};
			\node[circle,fill,inner sep=3pt,color=black] (d) at (0,3) {};

			\node[circle,fill,inner sep=3pt ] (e) at (2,0) {};
			\node[circle,fill,inner sep=3pt ] (f) at (2,1.5) {};
			\node[circle,fill,inner sep=3pt,color=black] (g) at (2,3) {};

			\draw[ line width=2pt, color=black] (d) -- (g);
			\draw (c) -- (f);
			\draw (b) -- (e);

			\node[circle,fill,inner sep=3pt ] (a1) at (5,0) {};
			\node[circle,fill,inner sep=3pt ] (b1) at (5,1) {};
			\node[circle,fill,inner sep=3pt ] (c1) at (5,2) {};
			\node[circle,fill,inner sep=3pt,color=black] (d1) at (5,3) {};

			\node[circle,fill,inner sep=3pt ] (e1) at (7,0) {};
			\node[circle,fill,inner sep=3pt ] (f1) at (7,1.5) {};
			\node[circle,fill,inner sep=3pt,color=black] (g1) at (7,3) {};

			\draw[ line width=2pt, color=black] (d1) -- (g1);
			\draw (c1) -- (f1);
			\draw (a1) -- (e1);

			\node[circle,fill,inner sep=3pt ] (a2) at (10,0) {};
			\node[circle,fill,inner sep=3pt ] (b2) at (10,1) {};
			\node[circle,fill,inner sep=3pt ] (c2) at (10,2) {};
			\node[circle,fill,inner sep=3pt,color=black] (d2) at (10,3) {};

			\node[circle,fill,inner sep=3pt ] (e2) at (12,0) {};
			\node[circle,fill,inner sep=3pt ] (f2) at (12,1.5) {};
			\node[circle,fill,inner sep=3pt,color=black] (g2) at (12,3) {};

			\draw[ line width=2pt, color=black] (d2) -- (g2);
			\draw (b2) -- (f2);
			\draw (c2) -- (e2);

			\node[circle,fill,inner sep=3pt ] (a3) at (15,0) {};
			\node[circle,fill,inner sep=3pt ] (b3) at (15,1) {};
			\node[circle,fill,inner sep=3pt ] (c3) at (15,2) {};
			\node[circle,fill,inner sep=3pt,color=black] (d3) at (15,3) {};

			\node[circle,fill,inner sep=3pt ] (e3) at (17,0) {};
			\node[circle,fill,inner sep=3pt ] (f3) at (17,1.5) {};
			\node[circle,fill,inner sep=3pt,color=black] (g3) at (17,3) {};

			\draw[ line width=2pt, color=black] (d3) -- (g3);
			
			\draw (a3) -- (f3);
			\draw (c3) -- (e3);

			\node[circle,fill,inner sep=3pt ] (a4) at (20,0) {};
			\node[circle,fill,inner sep=3pt ] (b4) at (20,1) {};
			\node[circle,fill,inner sep=3pt ] (c4) at (20,2) {};
			\node[circle,fill,inner sep=3pt,color=black] (d4) at (20,3) {};

			\node[circle,fill,inner sep=3pt ] (e4) at (22,0) {};
			\node[circle,fill,inner sep=3pt ] (f4) at (22,1.5) {};
			\node[circle,fill,inner sep=3pt,color=black] (g4) at (22,3) {};

			\draw[ line width=2pt, color=black] (d4) -- (g4);
			
			\draw (b4) -- (f4);
			\draw (a4) -- (e4);
			
			\node[circle,fill,inner sep=3pt ] (a5) at (25,0) {};
			\node[circle,fill,inner sep=3pt ] (b5) at (25,1) {};
			\node[circle,fill,inner sep=3pt ] (c5) at (25,2) {};
			\node[circle,fill,inner sep=3pt,color=black] (d5) at (25,3) {};

			\node[circle,fill,inner sep=3pt ] (e5) at (27,0) {};
			\node[circle,fill,inner sep=3pt ] (f5) at (27,1.5) {};
			\node[circle,fill,inner sep=3pt,color=black] (g5) at (27,3) {};

			\draw[ line width=2pt, color=black] (d5) -- (g5);
			\draw (a5) -- (f5);
			\draw (b5) -- (e5);

			\node[circle,fill,inner sep=3pt ] (a7) at (0,-6) {};
			\node[circle,fill,inner sep=3pt ] (b7) at (0,-5) {};
			\node[circle,fill,inner sep=3pt ] (c7) at (0,-4) {};
			\node[circle,fill,inner sep=3pt ] (d7) at (0,-3) {};

			\node[circle,fill,inner sep=3pt ] (e7) at (2,-6) {};
			\node[circle,fill,inner sep=3pt ] (f7) at (2,-4.5) {};
			\node[circle,fill,inner sep=3pt ] (g7) at (2,-3) {};

			\draw (c7) -- (f7);
			\draw (b7) -- (e7);

			\draw[ >->] (1,-2)--(1,-1);

			\node[circle,fill,inner sep=3pt ] (a8) at (5,-6) {};
			\node[circle,fill,inner sep=3pt ] (b8) at (5,-5) {};
			\node[circle,fill,inner sep=3pt ] (c8) at (5,-4) {};
			\node[circle,fill,inner sep=3pt ] (d8) at (5,-3) {};

			\node[circle,fill,inner sep=3pt ] (e8) at (7,-6) {};
			\node[circle,fill,inner sep=3pt ] (f8) at (7,-4.5) {};
			\node[circle,fill,inner sep=3pt ] (g8) at (7,-3) {};

			\draw (c8) -- (f8);
			\draw (a8) -- (e8);
			
			\draw[ >->] (6,-2)--(6,-1);
			
			\node[circle,fill,inner sep=3pt ] (a9) at (10,-6) {};
			\node[circle,fill,inner sep=3pt ] (b9) at (10,-5) {};
			\node[circle,fill,inner sep=3pt ] (c9) at (10,-4) {};
			\node[circle,fill,inner sep=3pt ] (d9) at (10,-3) {};

			\node[circle,fill,inner sep=3pt ] (e9) at (12,-6) {};
			\node[circle,fill,inner sep=3pt ] (f9) at (12,-4.5) {};
			\node[circle,fill,inner sep=3pt ] (g9) at (12,-3) {};

			\draw (b9) -- (f9);
			\draw (c9) -- (e9);
			\draw[ >->] (11,-2)--(11,-1);

			\node[circle,fill,inner sep=3pt ] (a10) at (15,-6) {};
			\node[circle,fill,inner sep=3pt ] (b10) at (15,-5) {};
			\node[circle,fill,inner sep=3pt ] (c10) at (15,-4) {};
			\node[circle,fill,inner sep=3pt ] (d10) at (15,-3) {};

			\node[circle,fill,inner sep=3pt ] (e10) at (17,-6) {};
			\node[circle,fill,inner sep=3pt ] (f10) at (17,-4.5) {};
			\node[circle,fill,inner sep=3pt ] (g10) at (17,-3) {};

			\draw (a10) -- (f10);
			\draw (c10) -- (e10);
			\draw[ >->] (16,-2)--(16,-1);

			\node[circle,fill,inner sep=3pt ] (a11) at (20,-6) {};
			\node[circle,fill,inner sep=3pt ] (b11) at (20,-5) {};
			\node[circle,fill,inner sep=3pt ] (c11) at (20,-4) {};
			\node[circle,fill,inner sep=3pt ] (d11) at (20,-3) {};

			\node[circle,fill,inner sep=3pt ] (e11) at (22,-6) {};
			\node[circle,fill,inner sep=3pt ] (f11) at (22,-4.5) {};
			\node[circle,fill,inner sep=3pt ] (g11) at (22,-3) {};

			\draw (b11) -- (f11);
			\draw (a11) -- (e11);
			
			\draw[ >->] (21,-2)--(21,-1);

			\node[circle,fill,inner sep=3pt ] (a12) at (25,-6) {};
			\node[circle,fill,inner sep=3pt ] (b12) at (25,-5) {};
			\node[circle,fill,inner sep=3pt ] (c12) at (25,-4) {};
			\node[circle,fill,inner sep=3pt ] (d12) at (25,-3) {};

			\node[circle,fill,inner sep=3pt ] (e12) at (27,-6) {};
			\node[circle,fill,inner sep=3pt ] (f12) at (27,-4.5) {};
			\node[circle,fill,inner sep=3pt ] (g12) at (27,-3) {};

			\draw (a12) -- (f12);
			\draw (b12) -- (e12);
			
			\draw[ >->] (26,-2)--(26,-1);
		\end{tikzpicture}
	\end{center}
	\caption{Pairings between $1$-simplices and $2$-simplices of $\M(\kk)$ with respect to $\vv^{(1,1)}$.}\label{matching 12 (11) K43}
\end{figure}
Therefore, $\U_{1}=\{\sigma \mid \sigma \text{ does not appear in }\vv^{(1,1)}\}$.

\textbf{Step II:}
\begin{align*}
	\text{Let }\vv^{(2,1)}:=&\{(\sigma, \sigma \sqcup \{(2,1)\}) \mid \sigma \sqcup \{(2,1)\} \in \M(\kk), \text{ and both }\sigma, \sigma \sqcup \{(2,1)\} \in \U_{1}\}.
\end{align*}
\autoref{matching 01 (21) K43} and \autoref{matching 12 (21) K43} illustrate all such pairings.
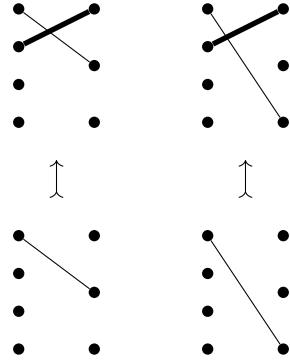
\begin{figure}[htbp]
	\begin{center}
		
		\begin{tikzpicture}[transform shape, scale = 0.5]
		
			\node[circle,fill,inner sep=3pt ] (a) at (0,0) {};
			\node[circle,fill,inner sep=3pt ] (b) at (0,1) {};
			\node[circle,fill,inner sep=3pt,color=black] (c) at (0,2) {};
			\node[circle,fill,inner sep=3pt ] (d) at (0,3) {};

			\node[circle,fill,inner sep=3pt ] (e) at (2,0) {};
			\node[circle,fill,inner sep=3pt ] (f) at (2,1.5) {};
			\node[circle,fill,inner sep=3pt,color=black] (g) at (2,3) {};

			\draw[ line width=2pt, color=black] (c) -- (g);
			\draw (d) -- (f);

			\node[circle,fill,inner sep=3pt ] (a1) at (5,0) {};
			\node[circle,fill,inner sep=3pt ] (b1) at (5,1) {};
			\node[circle,fill,inner sep=3pt,color=black] (c1) at (5,2) {};
			\node[circle,fill,inner sep=3pt ] (d1) at (5,3) {};

			\node[circle,fill,inner sep=3pt ] (e1) at (7,0) {};
			\node[circle,fill,inner sep=3pt ] (f1) at (7,1.5) {};
			\node[circle,fill,inner sep=3pt,color=black] (g1) at (7,3) {};
			
			\draw[ line width=2pt, color=black] (c1) -- (g1);
			\draw (d1) -- (e1);

			\node[circle,fill,inner sep=3pt ] (a3) at (0,-6) {};
			\node[circle,fill,inner sep=3pt ] (b3) at (0,-5) {};
			\node[circle,fill,inner sep=3pt ] (c3) at (0,-4) {};
			\node[circle,fill,inner sep=3pt ] (d3) at (0,-3) {};

			\node[circle,fill,inner sep=3pt ] (e3) at (2,-6) {};
			\node[circle,fill,inner sep=3pt ] (f3) at (2,-4.5) {};
			\node[circle,fill,inner sep=3pt ] (g3) at (2,-3) {};

			\draw (d3) -- (f3);
			
			\draw[ >->] (1,-2)--(1,-1);

			\node[circle,fill,inner sep=3pt ] (a4) at (5,-6) {};
			\node[circle,fill,inner sep=3pt ] (b4) at (5,-5) {};
			\node[circle,fill,inner sep=3pt ] (c4) at (5,-4) {};
			\node[circle,fill,inner sep=3pt ] (d4) at (5,-3) {};

			\node[circle,fill,inner sep=3pt ] (e4) at (7,-6) {};
			\node[circle,fill,inner sep=3pt ] (f4) at (7,-4.5) {};
			\node[circle,fill,inner sep=3pt ] (g4) at (7,-3) {};
			
			\draw (d4) -- (e4);
			\draw[ >->] (6,-2)--(6,-1);

		\end{tikzpicture}
	\end{center}
	\caption{Pairings between $0$-simplices and $1$-simplices of $\M(K_{4,3})$ with respect to $\vv^{(2,1)}$.}\label{matching 01 (21) K43}
\end{figure}
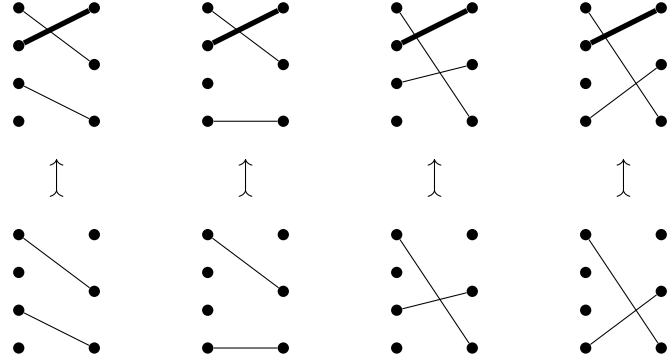
\begin{figure}[htbp]
	\begin{center}
		
		\begin{tikzpicture}[transform shape, scale = 0.5]
			\node[circle,fill,inner sep=3pt ] (a) at (0,0) {};
			\node[circle,fill,inner sep=3pt ] (b) at (0,1) {};
			\node[circle,fill,inner sep=3pt,color=black] (c) at (0,2) {};
			\node[circle,fill,inner sep=3pt ] (d) at (0,3) {};

			\node[circle,fill,inner sep=3pt ] (e) at (2,0) {};
			\node[circle,fill,inner sep=3pt ] (f) at (2,1.5) {};
			\node[circle,fill,inner sep=3pt,color=black] (g) at (2,3) {};

			\draw[ line width=2pt, color=black] (c) -- (g);
			\draw (d) -- (f);
			\draw (b) -- (e);

			\node[circle,fill,inner sep=3pt ] (a1) at (5,0) {};
			\node[circle,fill,inner sep=3pt ] (b1) at (5,1) {};
			\node[circle,fill,inner sep=3pt,color=black] (c1) at (5,2) {};
			\node[circle,fill,inner sep=3pt ] (d1) at (5,3) {};

			\node[circle,fill,inner sep=3pt ] (e1) at (7,0) {};
			\node[circle,fill,inner sep=3pt ] (f1) at (7,1.5) {};
			\node[circle,fill,inner sep=3pt,color=black] (g1) at (7,3) {};
			
			\draw[ line width=2pt, color=black] (c1) -- (g1);
			\draw (d1) -- (f1);
			\draw (a1) -- (e1);

			\node[circle,fill,inner sep=3pt ] (a2) at (10,0) {};
			\node[circle,fill,inner sep=3pt ] (b2) at (10,1) {};
			\node[circle,fill,inner sep=3pt,color=black] (c2) at (10,2) {};
			\node[circle,fill,inner sep=3pt ] (d2) at (10,3) {};

			\node[circle,fill,inner sep=3pt ] (e2) at (12,0) {};
			\node[circle,fill,inner sep=3pt ] (f2) at (12,1.5) {};
			\node[circle,fill,inner sep=3pt,color=black] (g2) at (12,3) {};
			
			\draw[ line width=2pt, color=black] (c2) -- (g2);
			\draw (b2) -- (f2);
			\draw (d2) -- (e2);

			\node[circle,fill,inner sep=3pt ] (a3) at (15,0) {};
			\node[circle,fill,inner sep=3pt ] (b3) at (15,1) {};
			\node[circle,fill,inner sep=3pt,color=black] (c3) at (15,2) {};
			\node[circle,fill,inner sep=3pt ] (d3) at (15,3) {};

			\node[circle,fill,inner sep=3pt ] (e3) at (17,0) {};
			\node[circle,fill,inner sep=3pt ] (f3) at (17,1.5) {};
			\node[circle,fill,inner sep=3pt,color=black] (g3) at (17,3) {};
			
			\draw[ line width=2pt, color=black] (c3) -- (g3);
			\draw (a3) -- (f3);
			\draw (d3) -- (e3);

			\node[circle,fill,inner sep=3pt ] (a5) at (0,-6) {};
			\node[circle,fill,inner sep=3pt ] (b5) at (0,-5) {};
			\node[circle,fill,inner sep=3pt ] (c5) at (0,-4) {};
			\node[circle,fill,inner sep=3pt ] (d5) at (0,-3) {};

			\node[circle,fill,inner sep=3pt ] (e5) at (2,-6) {};
			\node[circle,fill,inner sep=3pt ] (f5) at (2,-4.5) {};
			\node[circle,fill,inner sep=3pt ] (g5) at (2,-3) {};

			\draw (d5) -- (f5);
			\draw (b5) -- (e5);
			
			\draw[ >->] (1,-2)--(1,-1);

			\node[circle,fill,inner sep=3pt ] (a6) at (5,-6) {};
			\node[circle,fill,inner sep=3pt ] (b6) at (5,-5) {};
			\node[circle,fill,inner sep=3pt ] (c6) at (5,-4) {};
			\node[circle,fill,inner sep=3pt ] (d6) at (5,-3) {};

			\node[circle,fill,inner sep=3pt ] (e6) at (7,-6) {};
			\node[circle,fill,inner sep=3pt ] (f6) at (7,-4.5) {};
			\node[circle,fill,inner sep=3pt ] (g6) at (7,-3) {};
			
			\draw (d6) -- (f6);
			\draw (a6) -- (e6);
			
			\draw[ >->] (6,-2)--(6,-1);

			\node[circle,fill,inner sep=3pt ] (a7) at (10,-6) {};
			\node[circle,fill,inner sep=3pt ] (b7) at (10,-5) {};
			\node[circle,fill,inner sep=3pt ] (c7) at (10,-4) {};
			\node[circle,fill,inner sep=3pt ] (d7) at (10,-3) {};

			\node[circle,fill,inner sep=3pt ] (e7) at (12,-6) {};
			\node[circle,fill,inner sep=3pt ] (f7) at (12,-4.5) {};
			\node[circle,fill,inner sep=3pt ] (g7) at (12,-3) {};
			
			\draw (b7) -- (f7);
			\draw (d7) -- (e7);
			
			\draw[ >->] (11,-2)--(11,-1);

			\node[circle,fill,inner sep=3pt ] (a8) at (15,-6) {};
			\node[circle,fill,inner sep=3pt ] (b8) at (15,-5) {};
			\node[circle,fill,inner sep=3pt ] (c8) at (15,-4) {};
			\node[circle,fill,inner sep=3pt ] (d8) at (15,-3) {};

			\node[circle,fill,inner sep=3pt ] (e8) at (17,-6) {};
			\node[circle,fill,inner sep=3pt ] (f8) at (17,-4.5) {};
			\node[circle,fill,inner sep=3pt ] (g8) at (17,-3) {};
			
			\draw (a8) -- (f8);
			\draw (d8) -- (e8);
			\draw[ >->] (16,-2)--(16,-1);
		\end{tikzpicture}
	\end{center}
	\caption{All pairings among $1$-simplices and $2$-simplices of $\M(K_{4,3})$ with respect to $\vv^{(2,1)}$.}\label{matching 12 (21) K43}
\end{figure}

Therefore, $\U_{2}=\{\sigma \mid \sigma \text{ does not appear in  }\bigsqcup \limits_{i=1}^{2}\vv^{(i,1)}\}$.

\textbf{Step III:}
\begin{align*}
	\text {Let }\vv^{(3,1)}:=&\{(\sigma, \sigma \sqcup \{(3,1)\}) \mid \sigma \sqcup \{(3,1)\} \in \M(\kk) \text{ and both } \sigma, \sigma \sqcup (3,1) \in \U_{2}\}.
\end{align*}
\autoref{matching 12 (31) K43} illustrates all such pairings.
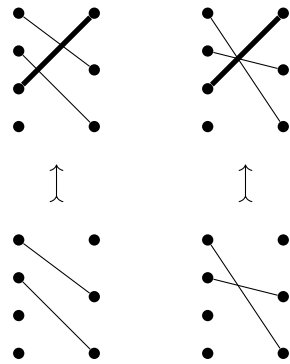
\begin{figure}[htbp]
	\begin{center}
		\begin{tikzpicture}[transform shape, scale = 0.5]
		
			\node[circle,fill,inner sep=3pt ] (a) at (0,0) {};
			\node[circle,fill,inner sep=3pt,color=black] (b) at (0,1) {};
			\node[circle,fill,inner sep=3pt ] (c) at (0,2) {};
			\node[circle,fill,inner sep=3pt ] (d) at (0,3) {};

			\node[circle,fill,inner sep=3pt ] (e) at (2,0) {};
			\node[circle,fill,inner sep=3pt ] (f) at (2,1.5) {};
			\node[circle,fill,inner sep=3pt,color=black] (g) at (2,3) {};

			\draw[ line width=2pt, color=black] (b) -- (g);
			\draw (d) -- (f);
			\draw (c) -- (e);

			\node[circle,fill,inner sep=3pt ] (a1) at (5,0) {};
			\node[circle,fill,inner sep=3pt,color=black] (b1) at (5,1) {};
			\node[circle,fill,inner sep=3pt ] (c1) at (5,2) {};
			\node[circle,fill,inner sep=3pt ] (d1) at (5,3) {};

			\node[circle,fill,inner sep=3pt ] (e1) at (7,0) {};
			\node[circle,fill,inner sep=3pt ] (f1) at (7,1.5) {};
			\node[circle,fill,inner sep=3pt,color=black] (g1) at (7,3) {};
			
			\draw[ line width=2pt, color=black] (b1) -- (g1);
			\draw (c1) -- (f1);
			\draw (d1) -- (e1);

			\node[circle,fill,inner sep=3pt ] (a3) at (0,-6) {};
			\node[circle,fill,inner sep=3pt ] (b3) at (0,-5) {};
			\node[circle,fill,inner sep=3pt ] (c3) at (0,-4) {};
			\node[circle,fill,inner sep=3pt ] (d3) at (0,-3) {};

			\node[circle,fill,inner sep=3pt ] (e3) at (2,-6) {};
			\node[circle,fill,inner sep=3pt ] (f3) at (2,-4.5) {};
			\node[circle,fill,inner sep=3pt ] (g3) at (2,-3) {};

			\draw (d3) -- (f3);
			\draw (c3) -- (e3);
			
			\draw[ >->] (1,-2)--(1,-1);

			\node[circle,fill,inner sep=3pt ] (a4) at (5,-6) {};
			\node[circle,fill,inner sep=3pt ] (b4) at (5,-5) {};
			\node[circle,fill,inner sep=3pt ] (c4) at (5,-4) {};
			\node[circle,fill,inner sep=3pt ] (d4) at (5,-3) {};

			\node[circle,fill,inner sep=3pt ] (e4) at (7,-6) {};
			\node[circle,fill,inner sep=3pt ] (f4) at (7,-4.5) {};
			\node[circle,fill,inner sep=3pt ] (g4) at (7,-3) {};
			
			\draw (c4) -- (f4);
			\draw (d4) -- (e4);
			
			\draw[ >->] (6,-2)--(6,-1);

		\end{tikzpicture}
	\end{center}
	\caption{All pairings between $1$-simplices and $2$-simplices of $\M(K_{4,3})$ with respect to $\vv^{(3,1)}$.}\label{matching 12 (31) K43}
\end{figure}

Therefore, $\U_{3}=\{\sigma \mid \sigma \text{ does not appear in }\bigsqcup \limits_{i=1}^{3}\vv^{(i,1)}\}$.

\textbf{Step IV:}
\begin{align*}
	\text {Let }\vv^{(1,2)}:=&\{(\sigma, \sigma \sqcup\{(1,2)\}) \mid \sigma \sqcup \{(1,2)\} \in \M(\kk),\text{ and both both } \sigma, \sigma \sqcup \{(1,2)\} \in \U_{3}\}.
\end{align*}
\autoref{matching 01 (12) K43} and \autoref{matching 12 (12) K43} illustrate these pairings.
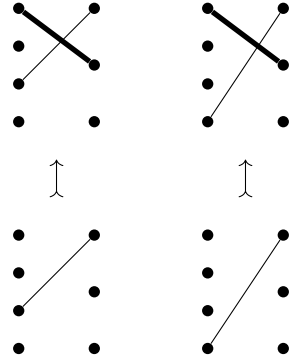
\begin{figure}[htbp]
	\begin{center}
		\begin{tikzpicture}[transform shape, scale = 0.5]
		
			\node[circle,fill,inner sep=3pt ] (a) at (0,0) {};
			\node[circle,fill,inner sep=3pt ] (b) at (0,1) {};
			\node[circle,fill,inner sep=3pt ] (c) at (0,2) {};
			\node[circle,fill,inner sep=3pt,color=black] (d) at (0,3) {};

			\node[circle,fill,inner sep=3pt ] (e) at (2,0) {};
			\node[circle,fill,inner sep=3pt,color=black] (f) at (2,1.5) {};
			\node[circle,fill,inner sep=3pt ] (g) at (2,3) {};

			\draw[ line width=2pt, color=black] (d) -- (f);
			\draw (b) -- (g);

			\node[circle,fill,inner sep=3pt ] (a1) at (5,0) {};
			\node[circle,fill,inner sep=3pt ] (b1) at (5,1) {};
			\node[circle,fill,inner sep=3pt ] (c1) at (5,2) {};
			\node[circle,fill,inner sep=3pt,color=black] (d1) at (5,3) {};

			\node[circle,fill,inner sep=3pt ] (e1) at (7,0) {};
			\node[circle,fill,inner sep=3pt,color=black] (f1) at (7,1.5) {};
			\node[circle,fill,inner sep=3pt ] (g1) at (7,3) {};
			
			\draw[ line width=2pt, color=black] (d1) -- (f1);
			\draw (a1) -- (g1);

			\node[circle,fill,inner sep=3pt ] (a3) at (0,-6) {};
			\node[circle,fill,inner sep=3pt ] (b3) at (0,-5) {};
			\node[circle,fill,inner sep=3pt ] (c3) at (0,-4) {};
			\node[circle,fill,inner sep=3pt ] (d3) at (0,-3) {};

			\node[circle,fill,inner sep=3pt ] (e3) at (2,-6) {};
			\node[circle,fill,inner sep=3pt ] (f3) at (2,-4.5) {};
			\node[circle,fill,inner sep=3pt ] (g3) at (2,-3) {};

			\draw (b3) -- (g3);
			
			\draw[ >->] (1,-2)--(1,-1);

			\node[circle,fill,inner sep=3pt ] (a4) at (5,-6) {};
			\node[circle,fill,inner sep=3pt ] (b4) at (5,-5) {};
			\node[circle,fill,inner sep=3pt ] (c4) at (5,-4) {};
			\node[circle,fill,inner sep=3pt ] (d4) at (5,-3) {};

			\node[circle,fill,inner sep=3pt ] (e4) at (7,-6) {};
			\node[circle,fill,inner sep=3pt ] (f4) at (7,-4.5) {};
			\node[circle,fill,inner sep=3pt ] (g4) at (7,-3) {};
			
			\draw (a4) -- (g4);
			
			\draw[ >->] (6,-2)--(6,-1);

		\end{tikzpicture}
	\end{center}
	\caption{Pairing between $0$-simplices and $1$-simplices of $\M(K_{4,3})$ with respect to $\vv^{(1,2)}$.}\label{matching 01 (12) K43}
\end{figure}
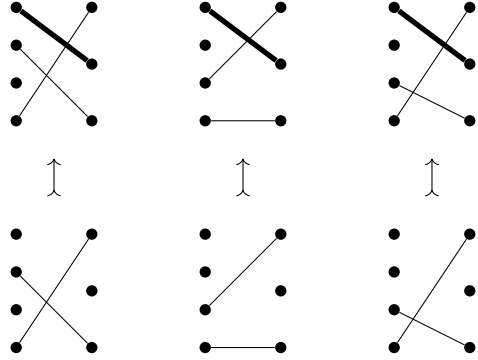
\begin{figure}[htbp]
	\begin{center}
		
		\begin{tikzpicture}[transform shape, scale = 0.5]
		
			\node[circle,fill,inner sep=3pt ] (a) at (0,0) {};
			\node[circle,fill,inner sep=3pt ] (b) at (0,1) {};
			\node[circle,fill,inner sep=3pt ] (c) at (0,2) {};
			\node[circle,fill,inner sep=3pt,color=black] (d) at (0,3) {};

			\node[circle,fill,inner sep=3pt ] (e) at (2,0) {};
			\node[circle,fill,inner sep=3pt,color=black] (f) at (2,1.5) {};
			\node[circle,fill,inner sep=3pt ] (g) at (2,3) {};

			\draw[ line width=2pt, color=black] (d) -- (f);
			\draw (c) -- (e);
			\draw (a) -- (g);

			\node[circle,fill,inner sep=3pt ] (a1) at (5,0) {};
			\node[circle,fill,inner sep=3pt ] (b1) at (5,1) {};
			\node[circle,fill,inner sep=3pt ] (c1) at (5,2) {};
			\node[circle,fill,inner sep=3pt,color=black] (d1) at (5,3) {};

			\node[circle,fill,inner sep=3pt ] (e1) at (7,0) {};
			\node[circle,fill,inner sep=3pt,color=black] (f1) at (7,1.5) {};
			\node[circle,fill,inner sep=3pt ] (g1) at (7,3) {};
			
			\draw[ line width=2pt, color=black] (d1) -- (f1);
			\draw (b1) -- (g1);
			\draw (a1) -- (e1);

			\node[circle,fill,inner sep=3pt ] (a2) at (10,0) {};
			\node[circle,fill,inner sep=3pt ] (b2) at (10,1) {};
			\node[circle,fill,inner sep=3pt ] (c2) at (10,2) {};
			\node[circle,fill,inner sep=3pt,color=black] (d2) at (10,3) {};

			\node[circle,fill,inner sep=3pt ] (e2) at (12,0) {};
			\node[circle,fill,inner sep=3pt,color=black] (f2) at (12,1.5) {};
			\node[circle,fill,inner sep=3pt ] (g2) at (12,3) {};
			
			\draw[ line width=2pt, color=black] (d2) -- (f2);
			\draw (b2) -- (e2);
			\draw (a2) -- (g2);

			\node[circle,fill,inner sep=3pt ] (a4) at (0,-6) {};
			\node[circle,fill,inner sep=3pt ] (b4) at (0,-5) {};
			\node[circle,fill,inner sep=3pt ] (c4) at (0,-4) {};
			\node[circle,fill,inner sep=3pt ] (d4) at (0,-3) {};

			\node[circle,fill,inner sep=3pt ] (e4) at (2,-6) {};
			\node[circle,fill,inner sep=3pt ] (f4) at (2,-4.5) {};
			\node[circle,fill,inner sep=3pt ] (g4) at (2,-3) {};

			\draw (c4) -- (e4);
			\draw (a4) -- (g4);
			
			\draw[ >->] (1,-2)--(1,-1);
	
			\node[circle,fill,inner sep=3pt ] (a5) at (5,-6) {};
			\node[circle,fill,inner sep=3pt ] (b5) at (5,-5) {};
			\node[circle,fill,inner sep=3pt ] (c5) at (5,-4) {};
			\node[circle,fill,inner sep=3pt ] (d5) at (5,-3) {};

			\node[circle,fill,inner sep=3pt ] (e5) at (7,-6) {};
			\node[circle,fill,inner sep=3pt ] (f5) at (7,-4.5) {};
			\node[circle,fill,inner sep=3pt ] (g5) at (7,-3) {};
			
			\draw (b5) -- (g5);
			\draw (a5) -- (e5);
			
			\draw[ >->] (6,-2)--(6,-1);

			\node[circle,fill,inner sep=3pt ] (a6) at (10,-6) {};
			\node[circle,fill,inner sep=3pt ] (b6) at (10,-5) {};
			\node[circle,fill,inner sep=3pt ] (c6) at (10,-4) {};
			\node[circle,fill,inner sep=3pt ] (d6) at (10,-3) {};

			\node[circle,fill,inner sep=3pt ] (e6) at (12,-6) {};
			\node[circle,fill,inner sep=3pt ] (f6) at (12,-4.5) {};
			\node[circle,fill,inner sep=3pt ] (g6) at (12,-3) {};
			
			\draw (b6) -- (e6);
			\draw (a6) -- (g6);

			\draw[ >->] (11,-2)--(11,-1);
			
		\end{tikzpicture}
	\end{center}
	\caption{Pairing between $1$-simplices and $2$-simplices of $\M(K_{4,3})$ with respect to $\vv^{(1,2)}$.}\label{matching 12 (12) K43}
\end{figure}

Therefore, $\U_{4}=\{\sigma \mid \sigma \text{ does not appear in }\bigsqcup \limits_{i=1}^{3}\vv^{(i,1)}\bigsqcup \vv^{(1,2)}\}$.

\textbf{Step V:}
\begin{align*}
	\text{Let }\vv^{(2,2)}:=&\{(\sigma, \sigma \sqcup \{(2,2)\}) \mid \sigma \sqcup \{(2,2)\} \in \M(\kk) \text{ and both }\sigma, \sigma \sqcup \{(2,2)\} \in \U_{4}\}.
\end{align*}
\autoref{matching 12 (22) K43} illustrates all such pairings.
\begin{figure}[htbp]
	\begin{center}
		
		\begin{tikzpicture}[transform shape, scale = 0.5]
		
			\node[circle,fill,inner sep=3pt ] (a) at (0,0) {};
			\node[circle,fill,inner sep=3pt ] (b) at (0,1) {};
			\node[circle,fill,inner sep=3pt ] (c) at (0,2) {};
			\node[circle,fill,inner sep=3pt ] (d) at (0,3) {};

			\node[circle,fill,inner sep=3pt ] (e) at (2,0) {};
			\node[circle,fill,inner sep=3pt ] (f) at (2,1.5) {};
			\node[circle,fill,inner sep=3pt ] (g) at (2,3) {};

			\draw (d) -- (e);
			\draw (a) -- (g);

			\node[circle,fill,inner sep=3pt ] (a2) at (5,0) {};
			\node[circle,fill,inner sep=3pt ] (b2) at (5,1) {};
			\node[circle,fill,inner sep=3pt,color=black] (c2) at (5,2) {};
			\node[circle,fill,inner sep=3pt ] (d2) at (5,3) {};

			\node[circle,fill,inner sep=3pt ] (e2) at (7,0) {};
			\node[circle,fill,inner sep=3pt,color=black] (f2) at (7,1.5) {};
			\node[circle,fill,inner sep=3pt ] (g2) at (7,3) {};

			\draw[ line width=2pt, color=black] (c2) -- (f2);
			\draw (d2) -- (e2);
			\draw (a2) -- (g2);
			
			\draw[ >->] (3,1.5)--(4,1.5);
			
		\end{tikzpicture}
	\end{center}
	\caption{Pairings between $1$-simplices and $2$-simplices of $\M(K_{4,3})$ with respect to $\vv^{(2,2)}$.}\label{matching 12 (22) K43}
\end{figure}
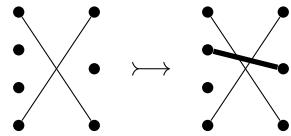

Therefore, $\U_{5}=\{\sigma \mid \sigma \text{ does not appear in }\bigsqcup \limits_{i=1}^{3}\vv^{(i,1)}\bigsqcup \limits_{i=1}^{2}\vv^{(i,2)}\}$.

\textbf{Step VI:}
\begin{align*}
	\text {Let }\vv^{(3,2)}:=&\{(\sigma, \sigma \sqcup \{(3,2)\}) \mid  \sigma \sqcup \{(3,2)\} \in \M(\kk) \text{ and both } \sigma, \sigma \sqcup \{(3,2)\} \in \U_{5}\}.
\end{align*}
\autoref{matching 01 (32) K43} and \autoref{matching 12 (32) K43} illustrate all such pairs.
\begin{figure}[htbp]
	\begin{center}
		
		\begin{tikzpicture}[transform shape, scale = 0.5]
		
			\node[circle,fill,inner sep=3pt ] (a) at (0,0) {};
			\node[circle,fill,inner sep=3pt ] (b) at (0,1) {};
			\node[circle,fill,inner sep=3pt ] (c) at (0,2) {};
			\node[circle,fill,inner sep=3pt ] (d) at (0,3) {};

			\node[circle,fill,inner sep=3pt ] (e) at (2,0) {};
			\node[circle,fill,inner sep=3pt ] (f) at (2,1.5) {};
			\node[circle,fill,inner sep=3pt ] (g) at (2,3) {};

			\draw (c) -- (g);

			\node[circle,fill,inner sep=3pt ] (a2) at (5,0) {};
			\node[circle,fill,inner sep=3pt,color=black] (b2) at (5,1) {};
			\node[circle,fill,inner sep=3pt ] (c2) at (5,2) {};
			\node[circle,fill,inner sep=3pt ] (d2) at (5,3) {};

			\node[circle,fill,inner sep=3pt ] (e2) at (7,0) {};
			\node[circle,fill,inner sep=3pt,color=black] (f2) at (7,1.5) {};
			\node[circle,fill,inner sep=3pt ] (g2) at (7,3) {};
			
			\draw[ line width=2pt, color=black] (b2) -- (f2);
			
			\draw (c2) -- (g2);
			
			\draw[ >->] (3,1.5)--(4,1.5);
			
		\end{tikzpicture}
	\end{center}
	\caption{Pairings between $0$-simplices and $1$-simplices of $\M(K_{4,3})$ with respect to $\vv^{(3,2)}$.}\label{matching 01 (32) K43}
\end{figure}
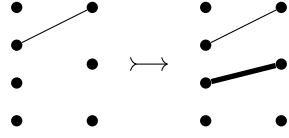
\begin{figure}[htbp]
	\begin{center}
		
		\begin{tikzpicture}[transform shape, scale = 0.5]
		
			\node[circle,fill,inner sep=3pt ] (a) at (0,0) {};
			\node[circle,fill,inner sep=3pt ] (b) at (0,1) {};
			\node[circle,fill,inner sep=3pt ] (c) at (0,2) {};
			\node[circle,fill,inner sep=3pt ] (d) at (0,3) {};

			\node[circle,fill,inner sep=3pt ] (e) at (2,0) {};
			\node[circle,fill,inner sep=3pt ] (f) at (2,1.5) {};
			\node[circle,fill,inner sep=3pt ] (g) at (2,3) {};

			\draw (c) -- (g);
			\draw (a) -- (e);

			\node[circle,fill,inner sep=3pt ] (a2) at (5,0) {};
			\node[circle,fill,inner sep=3pt,color=black] (b2) at (5,1) {};
			\node[circle,fill,inner sep=3pt ] (c2) at (5,2) {};
			\node[circle,fill,inner sep=3pt ] (d2) at (5,3) {};

			\node[circle,fill,inner sep=3pt ] (e2) at (7,0) {};
			\node[circle,fill,inner sep=3pt,color=black] (f2) at (7,1.5) {};
			\node[circle,fill,inner sep=3pt ] (g2) at (7,3) {};
			
			\draw[ line width=2pt, color=black] (b2) -- (f2);
			
			\draw (c2) -- (g2);
			\draw (a2) -- (e2);

			\draw[ >->] (3,1.5)--(4,1.5);

		\end{tikzpicture}
	\end{center}
	\caption{Pairings between $1$-simplices and $2$-simplices of $\M(K_{4,3})$ with respect to $\vv^{(3,2)}$.}\label{matching 12 (32) K43}
\end{figure}
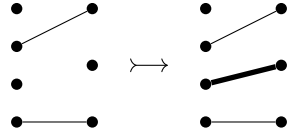

Therefore, $\U_{6}=\{\sigma \mid \sigma \text{ does not appear in }\bigsqcup \limits_{i=1}^{3}\vv^{(i,1)}\bigsqcup \limits_{i=1}^{3}\vv^{(i,2)}\}$.

\textbf{Step VII:}
\begin{align*}
	\text{Let }\vv^{(4,2)}:=&\{(\sigma, \sigma \sqcup \{(4,2)\}) \mid \sigma \sqcup \{(4,2)\} \in \M(\kk) \text{ and both } \sigma, \sigma \sqcup \{(4,2)\} \in \U_{6}\}.
\end{align*}
\autoref{matching 12 (42) K43} illustrates these pairs.

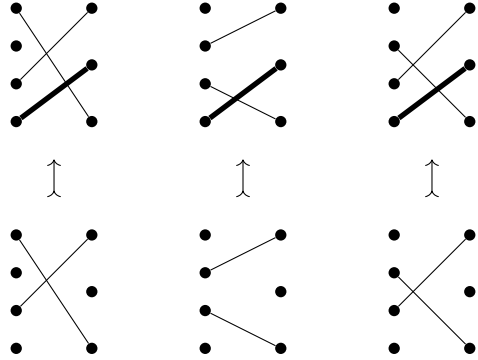
\begin{figure}[htbp]
	\begin{center}
		
		\begin{tikzpicture}[transform shape, scale = 0.5]
	
			\node[circle,fill,inner sep=3pt,color=black] (a) at (0,0) {};
			\node[circle,fill,inner sep=3pt ] (b) at (0,1) {};
			\node[circle,fill,inner sep=3pt ] (c) at (0,2) {};
			\node[circle,fill,inner sep=3pt ] (d) at (0,3) {};

			\node[circle,fill,inner sep=3pt ] (e) at (2,0) {};
			\node[circle,fill,inner sep=3pt,color=black] (f) at (2,1.5) {};
			\node[circle,fill,inner sep=3pt ] (g) at (2,3) {};

			\draw[ line width=2pt, color=black] (a) -- (f);
			\draw (d) -- (e);
			\draw (b) -- (g);

			\node[circle,fill,inner sep=3pt,color=black] (a1) at (5,0) {};
			\node[circle,fill,inner sep=3pt ] (b1) at (5,1) {};
			\node[circle,fill,inner sep=3pt ] (c1) at (5,2) {};
			\node[circle,fill,inner sep=3pt ] (d1) at (5,3) {};

			\node[circle,fill,inner sep=3pt ] (e1) at (7,0) {};
			\node[circle,fill,inner sep=3pt,color=black] (f1) at (7,1.5) {};
			\node[circle,fill,inner sep=3pt ] (g1) at (7,3) {};
			
			\draw[ line width=2pt, color=black] (a1) -- (f1);
			\draw (c1) -- (g1);
			\draw (b1) -- (e1);

			\node[circle,fill,inner sep=3pt,color=black] (a2) at (10,0) {};
			\node[circle,fill,inner sep=3pt ] (b2) at (10,1) {};
			\node[circle,fill,inner sep=3pt ] (c2) at (10,2) {};
			\node[circle,fill,inner sep=3pt ] (d2) at (10,3) {};

			\node[circle,fill,inner sep=3pt ] (e2) at (12,0) {};
			\node[circle,fill,inner sep=3pt,color=black] (f2) at (12,1.5) {};
			\node[circle,fill,inner sep=3pt ] (g2) at (12,3) {};
			
			\draw[ line width=2pt, color=black] (a2) -- (f2);
			\draw (c2) -- (e2);
			\draw (b2) -- (g2);

			\node[circle,fill,inner sep=3pt ] (a4) at (0,-6) {};
			\node[circle,fill,inner sep=3pt ] (b4) at (0,-5) {};
			\node[circle,fill,inner sep=3pt ] (c4) at (0,-4) {};
			\node[circle,fill,inner sep=3pt ] (d4) at (0,-3) {};

			\node[circle,fill,inner sep=3pt ] (e4) at (2,-6) {};
			\node[circle,fill,inner sep=3pt ] (f4) at (2,-4.5) {};
			\node[circle,fill,inner sep=3pt ] (g4) at (2,-3) {};

			\draw (d4) -- (e4);
			\draw (b4) -- (g4);
			
			\draw[ >->] (1,-2)--(1,-1);

			\node[circle,fill,inner sep=3pt ] (a5) at (5,-6) {};
			\node[circle,fill,inner sep=3pt ] (b5) at (5,-5) {};
			\node[circle,fill,inner sep=3pt ] (c5) at (5,-4) {};
			\node[circle,fill,inner sep=3pt ] (d5) at (5,-3) {};

			\node[circle,fill,inner sep=3pt ] (e5) at (7,-6) {};
			\node[circle,fill,inner sep=3pt ] (f5) at (7,-4.5) {};
			\node[circle,fill,inner sep=3pt ] (g5) at (7,-3) {};
			
			\draw (c5) -- (g5);
			\draw (b5) -- (e5);
			
			\draw[ >->] (6,-2)--(6,-1);

			\node[circle,fill,inner sep=3pt ] (a6) at (10,-6) {};
			\node[circle,fill,inner sep=3pt ] (b6) at (10,-5) {};
			\node[circle,fill,inner sep=3pt ] (c6) at (10,-4) {};
			\node[circle,fill,inner sep=3pt ] (d6) at (10,-3) {};

			\node[circle,fill,inner sep=3pt ] (e6) at (12,-6) {};
			\node[circle,fill,inner sep=3pt ] (f6) at (12,-4.5) {};
			\node[circle,fill,inner sep=3pt ] (g6) at (12,-3) {};
			
			\draw (c6) -- (e6);
			\draw (b6) -- (g6);

			\draw[ >->] (11,-2)--(11,-1);

		\end{tikzpicture}
	\end{center}
	\caption{Pairings between $1$-simplices and $2$-simplices of $\M(K_{4,3})$ with respect to $\vv^{(4,2)}$.}\label{matching 12 (42) K43}
\end{figure}

Therefore, $\U_{7}=\{\sigma \mid \sigma \text{ does not appear in }\bigsqcup \limits_{i=1}^{3}\vv^{(i,1)}\bigsqcup \limits_{i=1}^{4}\vv^{(i,2)}\}$.

\textbf{Step VIII:}
\begin{align*}
	\text{Let }\vv^{(1,3)}:=&\{(\sigma, \sigma \sqcup \{(1,3)\}) \mid \sigma \sqcup \{(1,3)\} \in \M(\kk) \text{ and both } \sigma, \sigma \sqcup \{(1,3)\} \in \U_{7}\}.
\end{align*}
\autoref{matching 12 (13) K43} illustrates this pair.
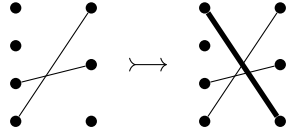
\begin{figure}[htbp]
	\begin{center}
		
		\begin{tikzpicture}[transform shape, scale = 0.5]
		
			\node[circle,fill,inner sep=3pt ] (a) at (0,0) {};
			\node[circle,fill,inner sep=3pt ] (b) at (0,1) {};
			\node[circle,fill,inner sep=3pt ] (c) at (0,2) {};
			\node[circle,fill,inner sep=3pt ] (d) at (0,3) {};

			\node[circle,fill,inner sep=3pt ] (e) at (2,0) {};
			\node[circle,fill,inner sep=3pt ] (f) at (2,1.5) {};
			\node[circle,fill,inner sep=3pt ] (g) at (2,3) {};

			\draw (b) -- (f);
			\draw (a) -- (g);

			\node[circle,fill,inner sep=3pt ] (a2) at (5,0) {};
			\node[circle,fill,inner sep=3pt ] (b2) at (5,1) {};
			\node[circle,fill,inner sep=3pt ] (c2) at (5,2) {};
			\node[circle,fill,inner sep=3pt,color=black] (d2) at (5,3) {};

			\node[circle,fill,inner sep=3pt ] (e2) at (7,0) {};
			\node[circle,fill,inner sep=3pt,color=black] (f2) at (7,1.5) {};
			\node[circle,fill,inner sep=3pt ] (g2) at (7,3) {};
			
			\draw[ line width=2pt, color=black] (d2) -- (e2);
			\draw (b2) -- (f2);
			\draw (a2) -- (g2);

			\draw[ >->] (3,1.5)--(4,1.5);

		\end{tikzpicture}
	\end{center}
	\caption{Pairing between $1$-simplices and $2$-simplices of $\M(K_{4,3})$ with respect to $\vv^{(1,3)}$.}\label{matching 12 (13) K43}
\end{figure}

Therefore, $\U_{8}=\{\sigma \mid \sigma \text{ does not appear in }\bigsqcup \limits_{i=1}^{3}\vv^{(i,1)}\bigsqcup \limits_{i=1}^{4}\vv^{(i,2)}\bigsqcup \vv^{(1,3)}\}$.

\textbf{Step IX:}
\begin{align*}
	\text{Let }\vv^{(2,3)}:=&\{(\sigma, \sigma \sqcup \{(2,3)\}) \mid \sigma \sqcup \{(2,3)\} \in \M(\kk) \text{ and both } \sigma, \sigma \sqcup \{(2,3)\} \in \U_{8}\}.
\end{align*}
\autoref{matching 12 (23) K43} illustrates this pair.
\begin{figure}[htbp]
	\begin{center}
		
		\begin{tikzpicture}[transform shape, scale = 0.5]
	
			\node[circle,fill,inner sep=3pt ] (a) at (0,0) {};
			\node[circle,fill,inner sep=3pt ] (b) at (0,1) {};
			\node[circle,fill,inner sep=3pt ] (c) at (0,2) {};
			\node[circle,fill,inner sep=3pt ] (d) at (0,3) {};

			\node[circle,fill,inner sep=3pt ] (e) at (2,0) {};
			\node[circle,fill,inner sep=3pt ] (f) at (2,1.5) {};
			\node[circle,fill,inner sep=3pt ] (g) at (2,3) {};

			\draw (c) -- (f);
			\draw (a) -- (g);

			\node[circle,fill,inner sep=3pt ] (a2) at (5,0) {};
			\node[circle,fill,inner sep=3pt,color=black] (b2) at (5,1) {};
			\node[circle,fill,inner sep=3pt ] (c2) at (5,2) {};
			\node[circle,fill,inner sep=3pt ] (d2) at (5,3) {};

			\node[circle,fill,inner sep=3pt,color=black] (e2) at (7,0) {};
			\node[circle,fill,inner sep=3pt ] (f2) at (7,1.5) {};
			\node[circle,fill,inner sep=3pt ] (g2) at (7,3) {};

			\draw[ line width=2pt, color=black] (b2) -- (e2);
			\draw (c2) -- (f2);
			\draw (a2) -- (g2);
			
			\draw[ >->] (3,1.5)--(4,1.5);

		\end{tikzpicture}
	\end{center}
	\caption{Pairing between $1$-simplices and $2$-simplices of $\M(K_{4,3})$ with respect to $\vv^{(2,3)}$.}\label{matching 12 (23) K43}
\end{figure}
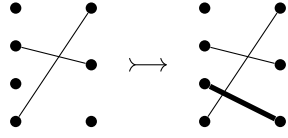

Therefore, $\U_{9}=\{\sigma \mid \sigma \text{ does not appear in }\bigsqcup \limits_{i=1}^{3}\vv^{(i,1)}\bigsqcup \limits_{i=1}^{4}\vv^{(i,2)}\bigsqcup \limits_{i=1}^{2} \vv^{(i,3)}\}$.

\textbf{Step X:}
\begin{align*}
	\text{Let }\vv^{(4,3)}:=&\{(\sigma, \sigma \sqcup \{(4,3)\}) \mid \sigma \sqcup \{(4,3)\} \in \M(\kk) \text{ and both } \sigma, \sigma \sqcup \{(4,3)\} \in \U_{9}\}.
\end{align*}
\autoref{matching 12 (43) K43} illustrates this pair.
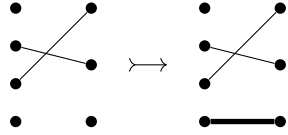
\begin{figure}[htbp]
	\begin{center}
		
		\begin{tikzpicture}[transform shape, scale = 0.5]
		
			\node[circle,fill,inner sep=3pt ] (a) at (0,0) {};
			\node[circle,fill,inner sep=3pt ] (b) at (0,1) {};
			\node[circle,fill,inner sep=3pt ] (c) at (0,2) {};
			\node[circle,fill,inner sep=3pt ] (d) at (0,3) {};

			\node[circle,fill,inner sep=3pt ] (e) at (2,0) {};
			\node[circle,fill,inner sep=3pt ] (f) at (2,1.5) {};
			\node[circle,fill,inner sep=3pt ] (g) at (2,3) {};

			\draw (c) -- (f);
			\draw (b) -- (g);

			\node[circle,fill,inner sep=3pt,color=black] (a2) at (5,0) {};
			\node[circle,fill,inner sep=3pt ] (b2) at (5,1) {};
			\node[circle,fill,inner sep=3pt ] (c2) at (5,2) {};
			\node[circle,fill,inner sep=3pt ] (d2) at (5,3) {};

			\node[circle,fill,inner sep=3pt,color=black] (e2) at (7,0) {};
			\node[circle,fill,inner sep=3pt ] (f2) at (7,1.5) {};
			\node[circle,fill,inner sep=3pt ] (g2) at (7,3) {};

			\draw[ line width=2pt, color=black] (a2) -- (e2);
			\draw (c2) -- (f2);
			\draw (b2) -- (g2);

			\draw[ >->] (3,1.5)--(4,1.5);
			
		\end{tikzpicture}
	\end{center}
	\caption{Pairing between $1$-simplices and $2$-simplices of $\M(K_{4,3})$ with respect to $\vv^{(4,3)}$.}\label{matching 12 (43) K43}
\end{figure}

Now, we define $\vv = \bigsqcup \limits_{i=1}^{3}\vv^{(i,1)} \bigsqcup \limits_{i=1}^{4}\vv^{(i,2)} \bigsqcup \limits_{\substack{i=1\\ i\neq 3}}^{4} \vv^{(i,3)}$. By construction, it follows that $\vv$ is well-defined. We can also verify that $\vv$ is acyclic. \autoref{k43 unpaired} illustrates the $\vv$-critical simplices of $\M(K_{4,3})$.
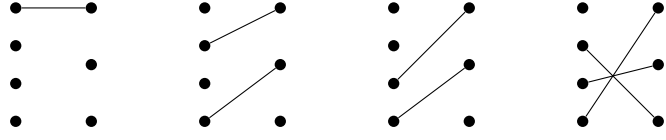
\begin{figure}[htbp]
	\begin{center}
		
		\begin{tikzpicture}[transform shape, scale = 0.5]
		
			\node[circle,fill,inner sep=3pt ] (a) at (0,0) {};
			\node[circle,fill,inner sep=3pt ] (b) at (0,1) {};
			\node[circle,fill,inner sep=3pt ] (c) at (0,2) {};
			\node[circle,fill,inner sep=3pt ] (d) at (0,3) {};

			\node[circle,fill,inner sep=3pt ] (e) at (2,0) {};
			\node[circle,fill,inner sep=3pt ] (f) at (2,1.5) {};
			\node[circle,fill,inner sep=3pt ] (g) at (2,3) {};

			\draw (d) -- (g);
	
			\node[circle,fill,inner sep=3pt ] (a1) at (5,0) {};
			\node[circle,fill,inner sep=3pt ] (b1) at (5,1) {};
			\node[circle,fill,inner sep=3pt ] (c1) at (5,2) {};
			\node[circle,fill,inner sep=3pt ] (d1) at (5,3) {};

			\node[circle,fill,inner sep=3pt ] (e1) at (7,0) {};
			\node[circle,fill,inner sep=3pt ] (f1) at (7,1.5) {};
			\node[circle,fill,inner sep=3pt ] (g1) at (7,3) {};
			
			\draw(c1) -- (g1);
		
			\draw (a1) -- (f1);

			\node[circle,fill,inner sep=3pt ] (a2) at (10,0) {};
			\node[circle,fill,inner sep=3pt ] (b2) at (10,1) {};
			\node[circle,fill,inner sep=3pt ] (c2) at (10,2) {};
			\node[circle,fill,inner sep=3pt ] (d2) at (10,3) {};

			\node[circle,fill,inner sep=3pt ] (e2) at (12,0) {};
			\node[circle,fill,inner sep=3pt ] (f2) at (12,1.5) {};
			\node[circle,fill,inner sep=3pt ] (g2) at (12,3) {};
			
			\draw (b2) -- (g2);
		
			\draw (a2) -- (f2);

			\node[circle,fill,inner sep=3pt ] (a3) at (15,0) {};
			\node[circle,fill,inner sep=3pt ] (b3) at (15,1) {};
			\node[circle,fill,inner sep=3pt ] (c3) at (15,2) {};
			\node[circle,fill,inner sep=3pt ] (d3) at (15,3) {};

			\node[circle,fill,inner sep=3pt ] (e3) at (17,0) {};
			\node[circle,fill,inner sep=3pt ] (f3) at (17,1.5) {};
			\node[circle,fill,inner sep=3pt ] (g3) at (17,3) {};
			
			\draw(c3) -- (e3);
			\draw (b3) -- (f3);
			\draw (a3) -- (g3);
		\end{tikzpicture}
	\end{center}
	\caption{All $\vv$-critical simplices in $\M(K_{4,3})$}\label{k43 unpaired}
\end{figure}

Therefore, for each $\alpha^{(1)} \in \mathcal{N}(\M(\K))$, the $\oW$-critical simplices \textbf{are} copies of the  $\vv$-critical simplices which differ by a labelling of the vertices. Since there are $10$ one-dimensional simplices in $\mathcal{N}(\M(\K))$, therefore,
\begin{align*}
	&\sum \limits_{\alpha^{(1)} \in \mathcal{N}(\M(\K))}|\Crit_0^{\oW}(\Aa)|=10,\\ & \sum \limits_{\alpha^{(1)} \in \mathcal{N}(\M(\K))}|\Crit_1^{\oW}(\Aa)|=20,\\ & \sum \limits_{\alpha^{(1)} \in \mathcal{N}(\M(\K))}|\Crit_2^{\oW}(\Aa)|=10.
\end{align*}

\textbf{Case III:} Let $\dim(\alpha)=2$.\label{ex:case3}

In this case, we observe that each $\Aa$ is a copy of $\M(\Kk)$ (where the edge set of $\Kk$ is $[4] \times [2]$), which differ by a labelling of vertices of $\Kk$. As in Case II, here also, it is sufficient to construct a gradient vector field on $\M(\Kk)$. We now construct such a gradient vector field $\vv'$ sequentially in $4$ steps. As before, the left and right partite sets denote $[4]$ and $[2]$ respectively, while the vertices are numbered from top to bottom in the figures.

Let $\U_{i}'$ denote the collection of all unpaired simplices after the $i^{th}$ step.

\textbf{Step I:}
\begin{align*}
	\text{Let }\vv^{(1,1)}:=&\{(\sigma, \sigma \sqcup \{(1,1)\}) \mid \sigma \sqcup \{(1,1)\} \in \M(K_{4,2})\}.
\end{align*}
\autoref{matching 01 (11) K42}  illustrates these pairs.
\begin{figure}[htbp]
	\begin{center}
		
		\begin{tikzpicture}[transform shape, scale = 0.5]
		
			\node[circle,fill,inner sep=3pt ] (a) at (0,0) {};
			\node[circle,fill,inner sep=3pt ] (b) at (0,1) {};
			\node[circle,fill,inner sep=3pt ] (c) at (0,2) {};
			\node[circle,fill,inner sep=3pt,color=black] (d) at (0,3) {};

			\node[circle,fill,inner sep=3pt ] (e) at (2,0.5) {};
			\node[circle,fill,inner sep=3pt,color=black] (f) at (2,2.5) {};
			
			\draw[ line width=2pt, color=black] (d) -- (f);
			\draw (c) -- (e);

			\node[circle,fill,inner sep=3pt ] (a1) at (5,0) {};
			\node[circle,fill,inner sep=3pt ] (b1) at (5,1) {};
			\node[circle,fill,inner sep=3pt ] (c1) at (5,2) {};
			\node[circle,fill,inner sep=3pt,color=black] (d1) at (5,3) {};

			\node[circle,fill,inner sep=3pt ] (e1) at (7,0.5) {};
			\node[circle,fill,inner sep=3pt,color=black] (f1) at (7,2.5) {};
			
			\draw[ line width=2pt, color=black] (d1) -- (f1);
			\draw (b1) -- (e1);

			\node[circle,fill,inner sep=3pt ] (a2) at (10,0) {};
			\node[circle,fill,inner sep=3pt ] (b2) at (10,1) {};
			\node[circle,fill,inner sep=3pt ] (c2) at (10,2) {};
			\node[circle,fill,inner sep=3pt,color=black] (d2) at (10,3) {};

			\node[circle,fill,inner sep=3pt ] (e2) at (12,0.5) {};
			\node[circle,fill,inner sep=3pt,color=black] (f2) at (12,2.5) {};
			
			\draw[ line width=2pt, color=black] (d2) -- (f2);
			\draw (a2) -- (e2);

			\node[circle,fill,inner sep=3pt ] (a4) at (0,-6) {};
			\node[circle,fill,inner sep=3pt ] (b4) at (0,-5) {};
			\node[circle,fill,inner sep=3pt ] (c4) at (0,-4) {};
			\node[circle,fill,inner sep=3pt ] (d4) at (0,-3) {};

			\node[circle,fill,inner sep=3pt ] (e4) at (2,-5.5) {};
			\node[circle,fill,inner sep=3pt ] (f4) at (2,-3.5) {};
			
			\draw (c4) -- (e4);
			
			\draw[ >->] (1,-2)--(1,-1);

			\node[circle,fill,inner sep=3pt ] (a5) at (5,-6) {};
			\node[circle,fill,inner sep=3pt ] (b5) at (5,-5) {};
			\node[circle,fill,inner sep=3pt ] (c5) at (5,-4) {};
			\node[circle,fill,inner sep=3pt ] (d5) at (5,-3) {};

			\node[circle,fill,inner sep=3pt ] (e5) at (7,-5.5) {};
			\node[circle,fill,inner sep=3pt ] (f5) at (7,-3.5) {};
			
			\draw (b5) -- (e5);
			
			\draw[ >->] (6,-2)--(6,-1);
			
			\node[circle,fill,inner sep=3pt ] (a6) at (10,-6) {};
			\node[circle,fill,inner sep=3pt ] (b6) at (10,-5) {};
			\node[circle,fill,inner sep=3pt ] (c6) at (10,-4) {};
			\node[circle,fill,inner sep=3pt ] (d6) at (10,-3) {};

			\node[circle,fill,inner sep=3pt ] (e6) at (12,-5.5) {};
			\node[circle,fill,inner sep=3pt ] (f6) at (12,-3.5) {};
			
			\draw (a6) -- (e6);

			\draw[ >->] (11,-2)--(11,-1);
			
		\end{tikzpicture}
	\end{center}
	\caption{Pairing between $0$-simplices and $1$-simplices of $\M(K_{4,2})$ with respect to $\vv^{(1,1)}$.}\label{matching 01 (11) K42}
\end{figure}
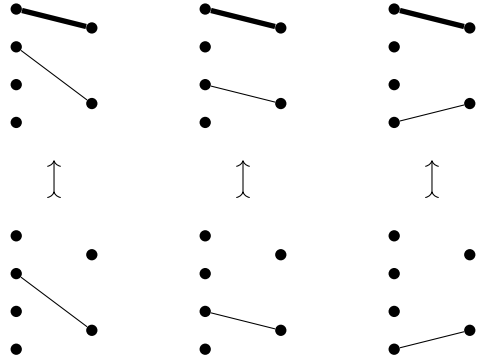

Therefore,
$\U_{1}'=\{\sigma \mid \sigma \text{ does not appear in }\vv^{(1,1)}\}$.

\textbf{Step II:}
\begin{align*}
	\text{Let }\vv^{(2,1)}:=&\{(\sigma, \sigma \sqcup \{(2,1)\}) \mid \sigma \sqcup \{(2,1)\} \in \M(K_{4,2}) \text{ and both } \sigma, \sigma \sqcup \{(2,1)\} \in \U_{1}'\}.
\end{align*}
\autoref{matching 01 (21) K42} illustrates this pair

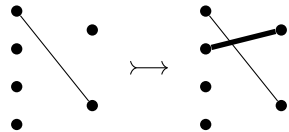
\begin{figure}[htbp]
	\begin{center}
		
		\begin{tikzpicture}[transform shape, scale = 0.5]
		
			\node[circle,fill,inner sep=3pt ] (a) at (0,0) {};
			\node[circle,fill,inner sep=3pt ] (b) at (0,1) {};
			\node[circle,fill,inner sep=3pt ] (c) at (0,2) {};
			\node[circle,fill,inner sep=3pt ] (d) at (0,3) {};

			\node[circle,fill,inner sep=3pt ] (e) at (2,0.5) {};
			\node[circle,fill,inner sep=3pt ] (f) at (2,2.5) {};
			
			\draw (d) -- (e);

			\node[circle,fill,inner sep=3pt ] (a2) at (5,0) {};
			\node[circle,fill,inner sep=3pt ] (b2) at (5,1) {};
			\node[circle,fill,inner sep=3pt,color=black] (c2) at (5,2) {};
			\node[circle,fill,inner sep=3pt ] (d2) at (5,3) {};

			\node[circle,fill,inner sep=3pt ] (e2) at (7,0.5) {};
			\node[circle,fill,inner sep=3pt,color=black] (f2) at (7,2.5) {};
			
			\draw (d2) -- (e2);
			\draw[ line width=2pt, color=black] (c2) -- (f2);
			
			\draw[ >->] (3,1.5)--(4,1.5);

		\end{tikzpicture}
	\end{center}
	\caption{Pairing between $0$-simplices and $1$-simplices of $\M(K_{4,2})$ with respect to $\vv^{(2,1)}$.}\label{matching 01 (21) K42}
\end{figure}
Therefore, $\U_{2}'=\{\sigma \mid \sigma \text{ does not appear }\bigsqcup \limits_{i=1}^{2} \vv^{(i,1)}\}$.

\textbf{Step III:}
\begin{align*}
	\text{Let }\vv^{(1,2)}:=&\{(\sigma, \sigma \sqcup \{(1,2)\}) \mid \sigma \sqcup \{(1,2)\} \in \M(K_{4,2}) \text{ and both } \sigma, \sigma \sqcup \{(1,2)\} \in \U_{2}'\}.
\end{align*}
\autoref{matching 01 (12) K42} illustrates these pairs.
\begin{figure}[htbp]
	\begin{center}
		
		\begin{tikzpicture}[transform shape, scale = 0.5]
			
			\node[circle,fill,inner sep=3pt ] (a) at (0,0) {};
			\node[circle,fill,inner sep=3pt ] (b) at (0,1) {};
			\node[circle,fill,inner sep=3pt ] (c) at (0,2) {};
			\node[circle,fill,inner sep=3pt,color=black] (d) at (0,3) {};

			\node[circle,fill,inner sep=3pt,color=black] (e) at (2,0.5) {};
			\node[circle,fill,inner sep=3pt ] (f) at (2,2.5) {};
			
			\draw[ line width=2pt, color=black] (d) -- (e);
			\draw (b) -- (f);

			\node[circle,fill,inner sep=3pt ] (a1) at (5,0) {};
			\node[circle,fill,inner sep=3pt ] (b1) at (5,1) {};
			\node[circle,fill,inner sep=3pt ] (c1) at (5,2) {};
			\node[circle,fill,inner sep=3pt,color=black] (d1) at (5,3) {};

			\node[circle,fill,inner sep=3pt,color=black] (e1) at (7,0.5) {};
			\node[circle,fill,inner sep=3pt ] (f1) at (7,2.5) {};
			
			\draw[ line width=2pt, color=black] (d1) -- (e1);
			\draw (a1) -- (f1);

			\node[circle,fill,inner sep=3pt ] (a3) at (0,-6) {};
			\node[circle,fill,inner sep=3pt ] (b3) at (0,-5) {};
			\node[circle,fill,inner sep=3pt ] (c3) at (0,-4) {};
			\node[circle,fill,inner sep=3pt ] (d3) at (0,-3) {};

			\node[circle,fill,inner sep=3pt ] (e3) at (2,-5.5) {};
			\node[circle,fill,inner sep=3pt ] (f3) at (2,-3.5) {};
			
			\draw (b3) -- (f3);
			
			\draw[ >->] (1,-2)--(1,-1);

			\node[circle,fill,inner sep=3pt ] (a4) at (5,-6) {};
			\node[circle,fill,inner sep=3pt ] (b4) at (5,-5) {};
			\node[circle,fill,inner sep=3pt ] (c4) at (5,-4) {};
			\node[circle,fill,inner sep=3pt ] (d4) at (5,-3) {};

			\node[circle,fill,inner sep=3pt ] (e4) at (7,-5.5) {};
			\node[circle,fill,inner sep=3pt ] (f4) at (7,-3.5) {};
			
			\draw (a4) -- (f4);

			\draw[ >->] (6,-2)--(6,-1);

		\end{tikzpicture}
	\end{center}
	\caption{Pairings between $0$-simplices and $1$-simplices of $\M(K_{4,2})$ with respect to $\vv^{(1,2)}$.}\label{matching 01 (12) K42}
\end{figure}
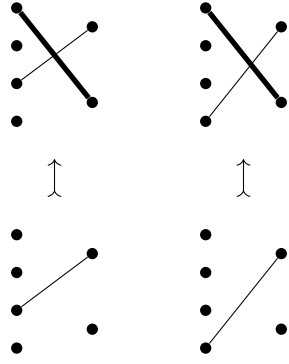

Therefore, $\U_{3}'=\{\sigma \mid \sigma \text{ does not appear in }\bigsqcup \limits_{i=1}^{2} \vv^{(i,1)} \bigsqcup \vv^{(1,2)}\}$.

\textbf{Step IV:}
\begin{align*}
	\text{Let }\vv^{(3,2)}:=&\{(\sigma, \sigma \sqcup \{(3,2)\}) \mid \sigma \sqcup \{(3,2)\} \in \M(K_{4,2}) \text{ and both } \sigma, \sigma \sqcup \{(3,2)\} \in \U_{3}'\}.
\end{align*}
\autoref{matching 01 (32) K42} illustrates this pair.
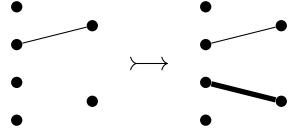
\begin{figure}[htbp]
	\begin{center}
		
		\begin{tikzpicture}[transform shape, scale = 0.5]
		
			\node[circle,fill,inner sep=3pt ] (a) at (0,0) {};
			\node[circle,fill,inner sep=3pt ] (b) at (0,1) {};
			\node[circle,fill,inner sep=3pt ] (c) at (0,2) {};
			\node[circle,fill,inner sep=3pt ] (d) at (0,3) {};

			\node[circle,fill,inner sep=3pt ] (e) at (2,0.5) {};
			\node[circle,fill,inner sep=3pt ] (f) at (2,2.5) {};
			
			\draw (c) -- (f);
		
			\node[circle,fill,inner sep=3pt ] (a2) at (5,0) {};
			\node[circle,fill,inner sep=3pt,color=black] (b2) at (5,1) {};
			\node[circle,fill,inner sep=3pt ] (c2) at (5,2) {};
			\node[circle,fill,inner sep=3pt ] (d2) at (5,3) {};

			\node[circle,fill,inner sep=3pt,color=black] (e2) at (7,0.5) {};
			\node[circle,fill,inner sep=3pt ] (f2) at (7,2.5) {};
			
			\draw[ line width=2pt, color=black] (b2) -- (e2);
			\draw (c2) -- (f2);

			\draw[ >->] (3,1.5)--(4,1.5);
			
		\end{tikzpicture}
	\end{center}
	\caption{Pairing between $0$-simplices and $1$-simplices of $\M(K_{4,2})$ with respect to $\vv^{(3,2)}$.}\label{matching 01 (32) K42}
\end{figure}

Thus, we now define $\vv':=\bigsqcup \limits_{i=1}^{2}\vv^{(i,1)} \bigsqcup \vv^{(1,2)} \bigsqcup \vv^{(3,2)}$. The well-definedness of $\vv'$ follows from the construction. We can also check that $\vv'$ is acyclic. \autoref{k42 unpaired} illustrates the $\vv'$ critical simplices of $\M(K_{4,2})$.

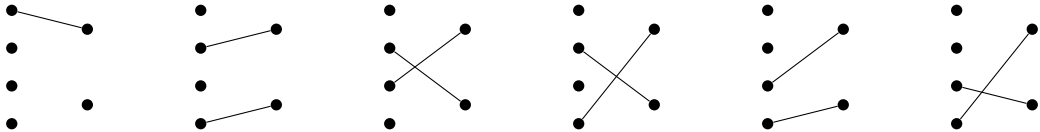
\begin{figure}[htbp]
	\begin{center}
		
		\begin{tikzpicture}[transform shape, scale = 0.5]
		
			\node[circle,fill,inner sep=3pt ] (a) at (0,0) {};
			\node[circle,fill,inner sep=3pt ] (b) at (0,1) {};
			\node[circle,fill,inner sep=3pt ] (c) at (0,2) {};
			\node[circle,fill,inner sep=3pt ] (d) at (0,3) {};

			\node[circle,fill,inner sep=3pt ] (e) at (2,0.5) {};
			\node[circle,fill,inner sep=3pt ] (f) at (2,2.5) {};
			
			\draw (d) -- (f);

			\node[circle,fill,inner sep=3pt ] (a1) at (5,0) {};
			\node[circle,fill,inner sep=3pt ] (b1) at (5,1) {};
			\node[circle,fill,inner sep=3pt ] (c1) at (5,2) {};
			\node[circle,fill,inner sep=3pt ] (d1) at (5,3) {};

			\node[circle,fill,inner sep=3pt ] (e1) at (7,0.5) {};
			\node[circle,fill,inner sep=3pt ] (f1) at (7,2.5) {};
			
			\draw (c1) -- (f1);
			\draw (a1) -- (e1);

			\node[circle,fill,inner sep=3pt ] (a2) at (10,0) {};
			\node[circle,fill,inner sep=3pt ] (b2) at (10,1) {};
			\node[circle,fill,inner sep=3pt ] (c2) at (10,2) {};
			\node[circle,fill,inner sep=3pt ] (d2) at (10,3) {};

			\node[circle,fill,inner sep=3pt ] (e2) at (12,0.5) {};
			\node[circle,fill,inner sep=3pt ] (f2) at (12,2.5) {};
			
			\draw (c2) -- (e2);
			\draw (b2) -- (f2);

			\node[circle,fill,inner sep=3pt ] (a3) at (15,0) {};
			\node[circle,fill,inner sep=3pt ] (b3) at (15,1) {};
			\node[circle,fill,inner sep=3pt ] (c3) at (15,2) {};
			\node[circle,fill,inner sep=3pt ] (d3) at (15,3) {};

			\node[circle,fill,inner sep=3pt ] (e3) at (17,0.5) {};
			\node[circle,fill,inner sep=3pt ] (f3) at (17,2.5) {};
			
			\draw (c3) -- (e3);
			\draw (a3) -- (f3);

			\node[circle,fill,inner sep=3pt ] (a4) at (20,0) {};
			\node[circle,fill,inner sep=3pt ] (b4) at (20,1) {};
			\node[circle,fill,inner sep=3pt ] (c4) at (20,2) {};
			\node[circle,fill,inner sep=3pt ] (d4) at (20,3) {};

			\node[circle,fill,inner sep=3pt ] (e4) at (22,0.5) {};
			\node[circle,fill,inner sep=3pt ] (f4) at (22,2.5) {};
			
			\draw (b4) -- (f4);
			\draw (a4) -- (e4);

			\node[circle,fill,inner sep=3pt ] (a5) at (25,0) {};
			\node[circle,fill,inner sep=3pt ] (b5) at (25,1) {};
			\node[circle,fill,inner sep=3pt ] (c5) at (25,2) {};
			\node[circle,fill,inner sep=3pt ] (d5) at (25,3) {};

			\node[circle,fill,inner sep=3pt ] (e5) at (27,0.5) {};
			\node[circle,fill,inner sep=3pt ] (f5) at (27,2.5) {};
			
			\draw (b5) -- (e5);
			\draw (a5) -- (f5);

		\end{tikzpicture}
	\end{center}
	\caption{All $\vv'$-critical simplices of $\M(K_{4,2})$.}\label{k42 unpaired}
\end{figure}

Thus, for each $\alpha^{(2)} \in \mathcal{N}(\M(\K))$, the $\oW$-critical simplices are but copies of these $\vv'$-critical simplices, which differ by a labelling of vertices.
There are precisely $10$ two-dimensional simplices in $\mathcal{N}(\M(\K))$. Therefore,
\begin{align*}
	&\sum \limits_{\alpha^{(2)} \in \mathcal{N}(\M(\K))}|\Crit_0^{\oW}(\Aa)|=10,\\ & \sum \limits_{\alpha^{(2)} \in \mathcal{N}(\M(\K))}|\Crit_1^{\oW}(\Aa)|=50.
\end{align*}

\textbf{Case IV:} Let $\dim(\alpha)=3$.

We observe that each $\Aa$ is a copy of $\M(K_{4,1})$, which differ by a labelling of vertices. Thus each $\Aa$ is a $0$-dimensional simplicial complex with precisely four $0$-simplices, all of which must be critical with respect to any gradient vector field. Since, there are five $3$-simplices in $\mathcal{N}(\M(\K))$, therefore,
$$ \sum \limits_{\alpha^{(3)} \in \mathcal{N}(\M(\K))}|\Crit_0^{\oW}(\Aa)|=20.$$

Now that we have assigned a gradient vector field $\oW$ on each $\Aa$, we can find $L_0$, $L_1$, $L_2$ and $L_3$. We list here the count of elements for each $L_i$.

\begin{align*}
	& |L_0|= \sum \limits_{\alpha^{(0)} \in \mathcal{N}(\M(\K))}|\Crit_0^{\oW}(\Aa)|= 5, \\
	& |L_1|= \sum \limits_{\alpha^{(1)} \in \mathcal{N}(\M(\K))}|\Crit_0^{\oW}(\Aa)|=10, \\
	& |L_2| = \sum \limits_{\alpha^{(1)} \in \mathcal{N}(\M(\K))}|\Crit_1^{\oW}(\Aa)| +\sum \limits_{\alpha^{(2)} \in \mathcal{N}(\M(\K))}|\Crit_0^{\oW}(\Aa)|=20 + 10 =30,\\
	&|L_3|= \sum \limits_{\alpha^{(1)} \in \mathcal{N}(\M(\K))}|\Crit_2^{\oW}(\Aa)| + \sum \limits_{\alpha^{(2)} \in \mathcal{N}(\M(\K))}|\Crit_1^{\oW}(\Aa)| \\ & + \sum \limits_{\alpha^{(3)} \in \mathcal{N}(\M(\K))}|\Crit_0^{\oW}(\Aa)| = 10+50+20 =80.
\end{align*}
We have computed the generalised trajectories for $\M(\K)$ through Python. We list two of them here as illustrative examples (see \autoref{Mgen} and \autoref{Mgen1}).

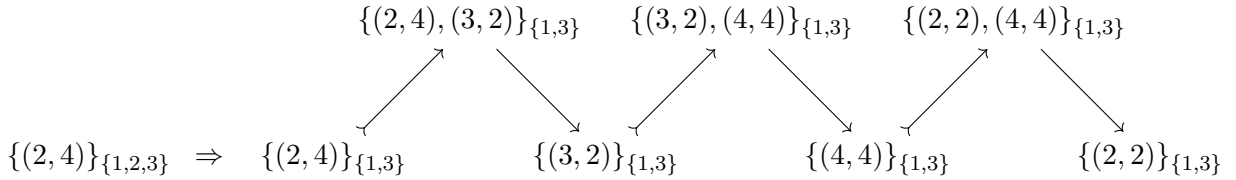
\begin{figure}[h]
	\centering
	\begin{tikzpicture}[
		vertex/.style={draw, circle, fill=black, inner sep=0.05pt, minimum size=0.7mm},
		baseline, scale=1.8]
		
		\node (a) at (0,0){$\{(2,4)\}_{\{1,2,3\}}$};
		\node (b) at (0.8,0) {$\hspace{0.8em}\Rightarrow$};
		\node (c) at (1.8,0){$\{(2,4)\}_{\{1,3\}}$};
		\node (d) at (2.8,1){$\{(2,4), (3,2)\}_{\{1,3\}}$};
		\node (e) at (3.8,0){$\{(3,2)\}_{\{1,3\}}$};
		\node (f) at (4.8,1){$\{(3,2), (4,4)\}_{\{1,3\}}$};
		\node (g) at (5.8,0){$\{(4,4)\}_{\{1,3\}}$};
		\node (h) at (6.8,1){$\{(2,2), (4,4)\}_{\{1,3\}}$};
		\node (i) at (7.8,0){$\{(2,2)\}_{\{1,3\}}$};
	
		\draw[>->] (c) -- (d);
		\draw[->] (d) -- (e);
		\draw[>->] (e) -- (f);
		\draw[->] (f) -- (g);
		\draw[>->] (g) -- (h);
		\draw[->] (h) -- (i);

	\end{tikzpicture}
	\caption{A generalised trajectory initiating from $\{(2,4)\}_{\{1,2,3\}} \in L_2$ and terminating at $\{(2,2)\}_{\{1,3\}}\in L_1$.}\label{Mgen}
\end{figure}
\begin{figure}[h]
	\centering
	\begin{tikzpicture}[
		vertex/.style={draw, circle, fill=black, inner sep=0.05pt, minimum size=0.7mm},
		baseline, scale=1.6]
		
		\node (a) at (0,0){$\{(4,4)\}_{\alpha_0}$};
		\node (b) at (0.8,0) {$\hspace{0.8em}\Rightarrow$};
		\node (c) at (1.8,0){$\{(4,4)\}_{\alpha_1}$};
		\node (d) at (1.8,1){$\{(2,5), (4,4)\}_{\alpha_1}$};
		\node (e) at (3.2,0){$\{(2,5)\}_{\alpha_1}$};
		\node (f) at (3.8,1){$\{(2,5), (3,4)\}_{\alpha_1}$};
		\node (g) at (4.8,1){$\hspace{0.8em}\Rightarrow$};
		\node (h) at (5.9,1){$\{(2,5), (3,4)\}_{\alpha_2}$};
		\node (i) at (5.9,2){$\{(2,5), (3,4), (4,3)\}_{\alpha_2}$};
		\node (j) at (7.7,1){$\{(2,5), (4,3)\}_{\alpha_2}$};
		\node (k) at (8.3,2){$\{(2,5), (4,3), (5,4)\}_{\alpha_2}$};
		\node (l) at (9.5,1) {$\{(4,3), (5,4)\}_{\alpha_2}$};
	
		\draw[>->] (c) -- (d);
		\draw[->] (d) -- (e);
		\draw[>->] (e) -- (f);
		\draw[>->] (h) -- (i);
		\draw[->] (i) -- (j);
		\draw[>->] (j) -- (k);
		\draw[->] (k) -- (l);

	\end{tikzpicture}
	\caption{A generalised trajectory initiating from $\{(4,4)\}_{\{1,2,3,5\}} \in L_3$ and terminating at $\{(4,3), (5,4)\}_{\{1,2\}}\in L_2$. Here, $\alpha_0= \{1,2,3,5\}$, $\alpha_1=\{1,2,3\}$, $\alpha_2=\{1,2\}$.}\label{Mgen1}
\end{figure}
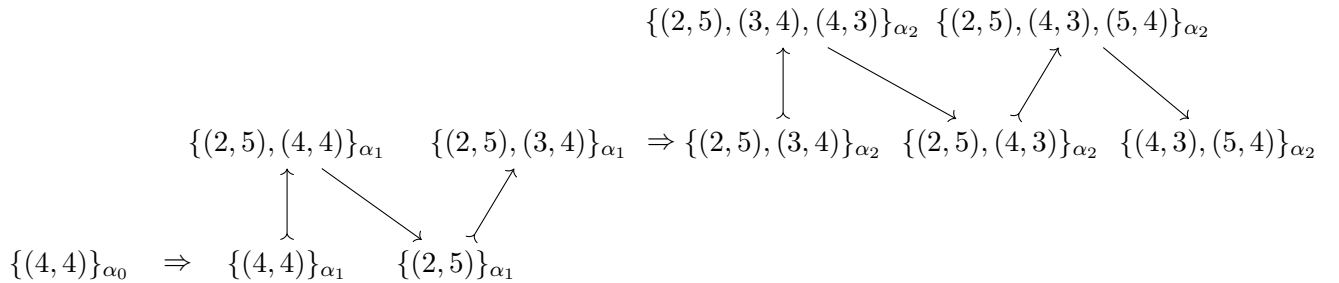

Next, we have computed the matrices of the boundary maps $\pL_1$, $\pL_2$ and $\pL_3$ using Python (Algorithm~\ref{alg:boundary first}). These are given as follows.

\begin{figure}[H]
	\centering
	\setlength{\arraycolsep}{4pt}
	\subfigure{$
		\pL_{1}=\left[
		\begin{array}{*{32}{c}}
			-1 & 1 & 0 & 0 & 0\\
			-1 & 0 & 1 & 0 & 0\\
			-1 & 0 & 0 & 1 & 0\\
			-1 & 0 & 0 & 0 & 1\\
			0 & -1 & 1 & 0 & 0\\
			0 & -1 & 0 & 1 & 0\\
			0 & -1 & 0 & 0 & 1\\
			0 & 0 &-1 & 1 & 0\\
			0 & 0 &-1 & 0 & 1\\
			0 & 0 & 0 &-1 & 1\\
			
		\end{array}
		\right]^{\mathsf T}$},
	\quad
	\subfigure{$
		\pL_{2}=\left[
		\begin{array}{*{32}{c}}
			0 & 0 & 0 & 0 & 0 & 0 & 0 & 0 & 0 & 0\\
			0 & 0 & 0 & 0 & 0 & 0 & 0 & 0 & 0 & 0\\
			1 & -1 &  0 &  0 & 1 & 0 & 0 & 0 & 0 & 0\\
			1 & 0 & -1 &  0 & 0 & 1 & 0 & 0 & 0 & 0\\
			1 & 0 & 0 & -1 & 0 & 0 & 1 & 0 & 0 & 0\\
			0 & 0 & 0 & 0 & 0 & 0 & 0 & 0 & 0 & 0\\
			0 & 0 & 0 & 0 & 0 & 0 & 0 & 0 & 0 & 0\\
			0 & 1 & -1 & 0 & 0 & 0 & 0 & 1 & 0 & 0\\
			0 & 1 & 0 & -1 & 0 & 0 & 0 & 0 & 1 & 0\\
			0 & 0 & 0 & 0 & 0 & 0 & 0 & 0 & 0 & 0\\
			0 & 0 & 0 & 0 & 0 & 0 & 0 & 0 & 0 & 0\\
			0 &  0 & 1 & -1 & 0 & 0 & 0 & 0 & 0 & 1\\
			0 & 0 & 0 & 0 & 0 & 0 & 0 & 0 & 0 & 0\\
			0 & 0 & 0 & 0 & 0 & 0 & 0 & 0 & 0 & 0\\
			0 & 0 & 0 & 0 & 0 & 0 & 0 & 0 & 0 & 0\\
			0 & 0 & 0 & 0 & 0 & 0 & 0 & 0 & 0 & 0\\
			0 & 0 &  0 & 0 & 1 & -1 & 0 & 1 & 0 & 0\\
			0 & 0 & 0 & 0 & 1 & 0 & -1 & 0 & 1 & 0\\
			0 & 0 & 0 & 0 & 0 & 0 & 0 & 0 & 0 & 0\\
			0 & 0 & 0 & 0 & 0 & 0 & 0 & 0 & 0 & 0\\
			0 & 0 & 0 & 0 & 0 & 1 & -1 & 0 & 0 & 1\\
			0 & 0 & 0 & 0 & 0 & 0 & 0 & 0 & 0 & 0\\
			0 & 0 & 0 & 0 & 0 & 0 & 0 & 0 & 0 & 0\\
			0 & 0 & 0 & 0 & 0 & 0 & 0 & 0 & 0 & 0\\
			0 & 0 & 0 & 0 & 0 & 0 & 0 & 0 & 0 & 0\\
			0 & 0 & 0 & 0 & 0 & 0 & 0 & 1 & -1 & 1\\
			0 & 0 & 0 & 0 & 0 & 0 & 0 & 0 & 0 & 0\\
			0 & 0 & 0 & 0 & 0 & 0 & 0 & 0 & 0 & 0\\
			0 & 0 & 0 & 0 & 0 & 0 & 0 & 0 & 0 & 0\\
			0 & 0 & 0 & 0 & 0 & 0 & 0 & 0 & 0 & 0\\
		\end{array}
		\right]^{\mathsf{T}}
		$}
\end{figure}\label{del1}

\begin{figure}[H]
	\centering
	\setlength{\arraycolsep}{5pt}
	
	\resizebox{0.95\textwidth}{!}{$
		\pL_{3}=\left[
		\begin{array}{*{32}{c}}
			0 & 0 & 0 & 0 & 0 & 0 & 0 & 0 & 0 & 0 & 0 & 0 & 0 & 0 & 0 & 0 & 0 & 0 & 0 & 0 & 0 & 0 & 0 & 0 & 0 & 0 & 0 & 0 & 0 & 0\\
			-1 & 0 & 0 & 0 & 0 & 1 & 0 & 0 & 0 & 0 & 0 & 0 & 0 & 0 & -1 & 0 & 0 & 0 & 0 & 0 & 0 & 0 & 0 & 0
			&  0 & 0 & 0 & 0 & 0 & 0\\
			1 & 0 & 0 & 0 & 0 & -1 & 0 & 0 & 0 & 0 & 0 & 0 & 0 & 0 & 1 & 0 & 0 & 0 & 0 & 0 & 0 & 0 & 0 & 0 &
			0 & 0 & 0 & 0 & 0 & 0\\
			0 & -1 &  0 & 0 & 0 & 0 & 1 & 0 & 0 & 0 & 0 & 0 & 0 & 0 & 0 & -1 & 0 & 0 & 0 & 0 & 0 & 0 & 0 & 0
			&  0 & 0 & 0 & 0 & 0 & 0\\
			0 & 1 & 0 & 0 & 0 & 0 & -1 & 0 & 0 & 0 & 0 & 0 & 0 & 0 & 0 & 1 & 0 & 0 & 0 & 0 & 0 & 0 & 0 & 0
			& 0 & 0 & 0 & 0 & 0 & 0\\
			0 & 0 & 0 & 0 & 0 & 0 & 0 & 0 & 0 & 0 & 0 & 0 & 0 & 0 & 0 & 0 & 0 & 0 & 0 & 0 & 0 & 0 & 0 & 0 & 0 & 0 & 0 & 0 & 0 & 0\\
			0 & -1 & -1 & 1 & 0 & 0 & 1 & -1 & 0 &  0 & -1 & 0 & 0 & 0 & 0 & -1 & 1 & 0 & 0 & 1 & 0 & 0 & 0 & 0
			& -1 & 0 & 0 &  0 & 0 & 0\\
			-1 & 0 & -1 & 1 & 0 & 1 & 0 & -1 & 0 & -1 & 0 & 0 & 0 & 0 &-1 & 0 & 1 & 0 & 1 & 0 & 0 & 0 & 0 & -1 &
			0 & 0 & 0 & 0 & 0 & 0\\
			-1 & -1 & -1 & 1 & 0 & 1 & 1 & -1 & 0 & -1 & -1 & 0 & 0 & 0 & -1 & -1 & 1 & 0 & 1 & 1 & 0 & 0 & 0 & -1 &
			-1 & 0 & 0 & 0 & 0 & 0\\
			0 & 0 & -1 & 1 & 0 & 0 & 0 & -1 & 0 & 0 & 0 & 0 & 0 & 0 & 0 & 0 & 1 & 0 &  0 & 0 & 0 & 0 & 0 & 0
			&  0 & 0 & 0 & 0 & 0 & 0\\
			0 & -1 & -1 & 0 & 1 & 0 & 1 & 0 & -1 & 0 & 0 & 0 & -1 & 0 & 0 & -1 & 0 & 1 & 0 & 0 & 0 & 1 & 0 & 0
			&  0 & 0 & -1 & 0 & 0 & 0\\
			0 & 0 & -1 & 0 & 1 & 0 & 0 & 0 & -1 & 0 & 0 & 0 & -1 & -1 & 0 & 0 & 0 & 1 & 0 & 0 & 0 & 1 & 1 & 0 &
			0 &  0 & -1 & -1 & 0 & 0\\
			-1 & -1 & -1 & 0 & 1 & 1 & 1 & 0 & -1 & 0 & 0 & 0 & 0 & 0 &-1 & -1 & 0 & 1 & 0 & 0 & 0 & 0 & 0 & 0
			&  0 & 0 & 0 & 0 & 0 & 0\\
			-1 & 0 & -1 & 0 & 1 & 1 & 0 & 0 & -1 & 0 & 0 & 0 & 0 & -1 & -1 & 0 & 0 & 1 & 0 & 0 & 0 & 0 & 1 & 0
			&  0 & 0 & 0 & -1 &  0 & 0\\
			-1 & 0 & 0 & 0 & 0 & 0 & 0 & 0 & 0 & -1 & 0 & 0 & 0 & 0 & 0 & 0 & 0 & 0 & 1 & 0 & 0 & 0 & 0 & 0
			&  0 & 0 & 0 & 0 & 0 & 0\\
			0 & -1 & 0 & 0 & 0 & 0 & 0 & 0 & 0 & 0 & -1 & 0 & 0 & 0 & 0 & 0 & 0 & 0 & 0 & 1 & 0 & 0 & 0 & 0
			&  0 & 0 & 0 & 0 & 0 & 0\\
			0 & 0 & 0 & 0 & 0 & 0 & 0 & 0 & 0 & 0 & 0 & 0 & 0 & 0 & 0 & 0 & 0 & 0 & 0 & 0 & 0 & 0 & 0 & 0 & 0 & 0 & 0 & 0 & 0 & 0\\
			1 & 0 & 0 & 0 & 0 & 0 & 0 & 0 & 0 & 1 & 0 & 0 & 0 & 0 & 0 & 0 & 0 & 0 & -1 & 0 & 0 & 0 & 0 & 0
			&  0 & 0 & 0 & 0 & 0 & 0\\
			0 & 1  &0 & 0 & 0 & 0 & 0 & 0 & 0 & 0 & 1 & 0 & 0 & 0 & 0 & 0 & 0 & 0 & 0 & -1 & 0 & 0 & 0 & 0
			&  0 & 0 & 0 & 0 & 0 & 0\\
			0 & 1 & 0 & -1 & 1 & 0 & 0 & 0 & 0 & 0 & 1 & -1 & 0 & -1 &  0 & 0 & 0 & 0 & 0 & -1 & 1 & 0 & 1 & 0
			&  0 & 0 & 0 & 0 & -1 & 0\\
			0 & 0 & 0 & -1 & 1 & 0 & 0 & 0 & 0 & 0 & 0 & -1 & 0 & 0 & 0 & 0 & 0 & 0 & 0 & 0 & 1 & 0 & 0 & 0 &
			0 & 0 & 0 & 0 & -1 & -1\\
			1 & 1 & 0 & -1 & 1 & 0 & 0 & 0 & 0 & 1 & 1 & -1 & -1 & -1 &  0 & 0 & 0 &  0 & -1 & -1 & 1 & 1 & 1 & 0
			&  0 & 0 & 0 & 0 & 0 & 0\\
			1 & 0 & 0 & -1 & 1 & 0 & 0 & 0 & 0 & 1 & 0 & -1 & -1 & 0 & 0 & 0 & 0 & 0 & -1 & 0 & 1 & 1 & 0 & 0 &
			0 & 0 & 0 & 0 & 0 & -1\\
			0 & 1 & 0 & 0 & 0 & 0 & 0 & 0 & 0 & 0 & 0 & 0 & 0 & -1 & 0 & 0 & 0 & 0 & 0 & 0 & 0 & 0 & 1 & 0
			&  0 & 0 & 0 & 0 & 0 & 0\\
			0 & 0 & 0 & 0 & 0 & 0 & 0 & 0 & 0 & 0 & 0 & 0 & 0 & 0 & 0 & 0 & 0 & 0 & 0 & 0 & 0 & 0 & 0 & 0 & 0 & 0 & 0 & 0 & 0 & 0\\
			-1 & 0 & 0 & 0 & 0 & 0 & 0 & 0 & 0 & 0 & 0 & 0 & 1 & 0 & 0 & 0 & 0 & 0 & 0 & 0 & 0 & -1 &  0 & 0 &
			0 & 0 & 0 & 0 & 0 & 0\\
			1 & 0 & 0 & 0 & 0 & 0 & 0 & 0 & 0 & 0 & 0 & 0 & -1 & 0 & 0 & 0 & 0 & 0 & 0 & 0 & 0 & 1 & 0 & 0
			&  0 & 0 & 0 & 0 & 0 & 0\\
			0 & -1 & 0 & 0 & 0 & 0 & 0 & 0 & 0 & 0 & 0 & 0 & 0 & 1 & 0 & 0 & 0 & 0 & 0 & 0 & 0 & 0 & -1 & 0
			&  0 & 0 & 0 & 0 & 0 & 0\\
			0 & 0 & 0 & 0 & 0 & 0 & 0 & 0 & 0 & 0 & 0 & 0 & 0 & 0 & 0 & 0 & 0 & 0 & 0 & 0 & 0 & 0 & 0 & 0 & 0 & 0 & 0 & 0 & 0 & 0\\
			0 & 0 & 0 & 0 & 0 & -1 & 0 & 0 & 0 & 1 & 0 & 0 & 0 & 0 & 0 & 0 & 0 & 0 & 0 & 0 & 0 & 0 & 0 & 1
			&  0 & 0 & 0 & 0 & 0 & 0\\
			0 & 0 & 0 & 0 & 0 & 0 & 1 & 0 & 0 & 0 & -1 & 0 & 0 & 0 & 0 & 0 & 0 & 0 & 0 & 0 & 0 & 0 & 0 & 0
			& -1 & 0 & 0 & 0 & 0 & 0\\
			0 & 0 & 0 & 0 & 0 & 0 & -1 & 0 & 0 & 0 & 1 & 0 & 0 & 0 & 0 & 0 & 0 & 0 & 0 & 0 & 0 & 0 & 0 & 0
			&  1 & 0 & 0 & 0 & 0 & 0\\
			0 & 0 & 0 & 0 & 0 & 0 & 0 & 0 & 0 & 0 & 0 & 0 & 0 & 0 & 0 & 0 & 0 & 0 & 0 & 0 & 0 & 0 & 0 & 0 & 0 & 0 & 0 & 0 & 0 & 0\\
			0 & 0 & 0 & 0 & 0 & 1 & 0 & 0 & 0 & -1 & 0 & 0 & 0 & 0 & 0 & 0 & 0 & 0 & 0 & 0 & 0 & 0 & 0 & -1 &
			0 & 0 & 0 & 0 & 0 & 0\\
			0 & 0 & 0 & 0 & 0 & 0 & 0 & -1 & 1 & 0 & 0 & -1 & 0 & 0 & 0 & 0 & 0 & 0 & 0 & 0 & 0 & 0 & 0 & 0 &
			0 & 1 & 0 & 0 & 0 & 0\\
			0 & 0 & 0 & 0 & 0 & 1 & 1 & -1 & 1 & -1 & -1 & -1 & 1 & 1 & 0 & 0 & 0 & 0 & 0 & 0 & 0 & 0 & 0 & -1 &
			-1 & 1 & 1 & 1 & -1 & -1\\
			0 & 0 & 0 & 0 & 0 & 1 & 0 & -1 & 1 & -1 & 0 & -1 & 1 & 0 & 0 & 0 & 0 & 0 & 0 & 0 & 0 & 0 & 0 & -1 &
			0 & 1 & 1 & 0 & -1 & 0\\
			0 & 0 & 0 & 0 & 0 & 0 & 1 & -1 & 1 & 0 & -1 & -1 & 0 & 1 & 0 & 0 & 0 & 0 & 0 & 0 & 0 & 0 & 0 & 0
			& -1 & 1 & 0 & 1 & 0 & -1\\
			0 & 0 & 0 & 0 & 0 & 0 & 1 & 0 & 0 & 0 & 0 & 0 & 0 & 1 & 0 & 0 & 0 & 0 & 0 & 0 & 0 & 0 & 0 & 0 & 0 & 0 & 0 & 1 & 0 & 0\\
			0 & 0 & 0 & 0 & 0 & 0 & -1 & 0 & 0 & 0 & 0 & 0 & 0 & -1 &  0 & 0 & 0 & 0 & 0 & 0 & 0 & 0 & 0 & 0 &
			0 & 0 & 0 & -1 & 0 & 0\\
			0 & 0 & 0 & 0 & 0 & 0 & 0 & 0 & 0 & 0 & 0 & 0 & 0 & 0 & 0 & 0 & 0 & 0 & 0 & 0 & 0 & 0 & 0 & 0 & 0 & 0 & 0 & 0 & 0 & 0\\
			0 & 0 & 0 & 0 & 0 & 1 & 0 & 0 & 0 & 0 & 0 & 0 & 1 & 0 & 0 & 0 & 0 & 0 & 0 & 0 & 0 & 0 & 0 & 0 & 0 & 0 & 1 & 0 & 0 & 0\\
			0 & 0 & 0 & 0 & 0 & -1 & 0 & 0 & 0 & 0 & 0 & 0 & -1 & 0 & 0 & 0 & 0 & 0 & 0 & 0 & 0 & 0 & 0 & 0
			&  0 & 0 & -1 & 0 & 0 & 0\\
			0 & 0 & 0 & 0 & 0 & 0 & 0 & 0 & 0 & 0 & 0 & 0 & 0 & 0 & 0 & 0 & 0 & 0 & 0 & 0 & 0 & 0 & 0 & 0 & 0 & 0 & 0 & 0 & 0 & 0\\
			0 & 0 & 0 & 0 & 0 & 0 & 0 & 0 & 0 & 0 & -1 & 0 & 0 & 1 & 0 & 0 & 0 & 0 & 0 & 0 & 0 & 0 & 0 & 0 &
			0 & 0 & 0 & 0 & 0 & -1\\
			0 & 0 & 0 & 0 & 0 & 0 & 0 & 0 & 0 & 1 & 0 & 0 & -1 & 0 & 0 & 0 & 0 & 0 & 0 & 0 & 0 & 0 & 0 & 0
			&  0 & 0 & 0 & 0 & 1 & 0\\
			0 & 0 & 0 & 0 & 0 & 0 & 0 & 0 & 0 & -1 & 0 & 0 & 1 & 0 & 0 & 0 & 0 & 0 & 0 & 0 & 0 & 0 & 0 & 0
			&  0 & 0 & 0 & 0 & -1 & 0\\
			0 & 0 & 0 & 0 & 0 & 0 & 0 & 0 & 0 & 0 & 1 & 0 & 0 & -1 & 0 & 0 & 0 & 0 & 0 & 0 & 0 & 0 & 0 & 0
			&  0 & 0 & 0 & 0 & 0 & 1\\
			0 & 0 & 0 & 0 & 0 & 0 & 0 & 0 & 0 & 0 & 0 & 0 & 0 & 0 & 0 & 0 & 0 & 0 & 0 & 0 & 0 & 0 & 0 & 0 & 0 & 0 & 0 & 0 & 0 & 0\\
			0 & 0 & 0 & 0 & 0 & 0 & 0 & 0 & 0 & 0 & 0 & 0 & 0 & 0 & 0 & 0 & 0 & 0 & 0 & 0 & 0 & 0 & 0 & 0 & 0 & 0 & 0 & 0 & 0 & 0\\
			0 & 0 & 0 & 0 & 0 & 0 & 0 & 0 & 0 & 0 & 0 & 0 & 0 & 0 & 0 & 0 & 0 & 0 & 0 & 0 & 0 & 0 & 0 & 0 & 0 & 0 & 0 & 0 & 0 & 0\\
			0 & 0 & 0 & 0 & 0 & 0 & 0 & 0 & 0 & 0 & 0 & 0 & 0 & 0 & 0 & 0 & 0 & 0 & 0 & 0 & 0 & 0 & 0 & 0 & 0 & 0 & 0 & 0 & 0 & 0\\
			0 & 0 & 0 & 0 & 0 & 0 & 0 & 0 & 0 & 0 & 0 & 0 & 0 & 0 & 1 & 0 & 0 & 0 & -1 & 0 & 0 & 0 & 0 & 1
			&  0 & 0 & 0 & 0 & 0 & 0\\
			0 & 0 & 0 & 0 & 0 & 0 & 0 & 0 & 0 & 0 & 0 & 0 & 0 & 0 & -1 & 0 & 0 & 0 & 1 & 0 & 0 & 0 & 0 & -1 &
			0 & 0 & 0 & 0 & 0 & 0\\
			0 & 0 & 0 & 0 & 0 & 0 & 0 & 0 & 0 & 0 & 0 & 0 & 0 & 0 & 0 & 1 & 0 & 0 & 0 & -1 & 0 & 0 & 0 & 0
			&  1 & 0 & 0 & 0 & 0 & 0\\
			0 & 0 & 0 & 0 & 0 & 0 & 0 & 0 & 0 & 0 & 0 & 0 & 0 & 0 & 0 & -1 & 0 & 0 & 0 & 1 & 0 & 0 & 0 & 0
			& -1 & 0 & 0 & 0 & 0 & 0\\
			0 & 0 & 0 & 0 & 0 & 0 & 0 & 0 & 0 & 0 & 0 & 0 & 0 & 0 & 1 & 1 & -1 & 1 & -1 & -1 & -1 &  1 & 1 & 1 &
			1 & 1 & -1 & -1 & 1 & 1\\
			0 & 0 & 0 & 0 & 0 & 0 & 0 & 0 & 0 & 0 & 0 & 0 & 0 & 0 & 0 & 0 & -1 & 1 & 0 & 0 & -1 & 0 & 0 & 0
			& 0 & 1 & 0 & 0 & 0 & 0\\
			0 & 0 & 0 & 0 & 0 & 0 & 0 & 0 & 0 & 0 & 0 & 0 & 0 & 0 & 1 & 0 & -1 & 1 & -1 & 0 & -1 & 1 & 0 & 1
			&  0 & 1 & -1 & 0 & 1 & 0\\
			0 & 0 & 0 & 0 & 0 & 0 & 0 & 0 & 0 & 0 & 0 & 0 & 0 & 0 & 0 & 1 & -1 & 1 & 0 & -1 & -1 & 0 & 1 & 0
			&  1 & 1 & 0 & -1 & 0 & 1\\
			0 & 0 & 0 & 0 & 0 & 0 & 0 & 0 & 0 & 0 & 0 & 0 & 0 & 0 & 1 & 0 & 0 & 0 & 0 & 0 & 0 & 1 & 0 & 0
			&  0 & 0 & -1 & 0 & 0 & 0\\
			0 & 0 & 0 & 0 & 0 & 0 & 0 & 0 & 0 & 0 & 0 & 0 & 0 & 0 & -1 & 0 & 0 & 0 & 0 & 0 & 0 & -1 & 0 & 0
			&  0 & 0 & 1 & 0 & 0 & 0\\
			0 & 0 & 0 & 0 & 0 & 0 & 0 & 0 & 0 & 0 & 0 & 0 & 0 & 0 & 0 & 1 & 0 & 0 & 0 & 0 & 0 & 0 & 1 & 0
			& 0 & 0 & 0 & -1 & 0 & 0\\
			0 & 0 & 0 & 0 & 0 & 0 & 0 & 0 & 0 & 0 & 0 & 0 & 0 & 0 & 0 & -1 & 0 & 0 & 0 & 0 & 0 & 0 & -1 & 0
			&  0 & 0 & 0 & 1 & 0 & 0\\
			0 & 0 & 0 & 0 & 0 & 0 & 0 & 0 & 0 & 0 & 0 & 0 & 0 & 0 & 0 & 0 & 0 & 0 & 0 & 0 & 0 & 0 & 0 & 0 & 0 & 0 & 0 & 0 & 0 & 0\\
			0 & 0 & 0 & 0 & 0 & 0 & 0 & 0 & 0 & 0 & 0 & 0 & 0 & 0 & 0 & 0 & 0 & 0 & 0 & 0 & 0 & 0 & 0 & 0 & 0 & 0 & 0 & 0 & 0 & 0\\
			0 & 0 & 0 & 0 & 0 & 0 & 0 & 0 & 0 & 0 & 0 & 0 & 0 & 0 & 0 & 0 & 0 & 0 & 1 & 0 & 0 & -1 & 0 & 0
			& 0 & 0 & 0 & 0 & -1 & 0\\
			0 & 0 & 0 & 0 & 0 & 0 & 0 & 0 & 0 & 0 & 0 & 0 & 0 & 0 & 0 & 0 & 0 & 0 & 0 & 1 & 0 & 0 & -1 & 0
			&  0 & 0 & 0 & 0 & 0 & -1\\
			0 & 0 & 0 & 0 & 0 & 0 & 0 & 0 & 0 & 0 & 0 & 0 & 0 & 0 & 0 & 0 & 0 & 0 & -1 & 0 & 0 & 1 & 0 & 0
			& 0 & 0 & 0 & 0 & 1 & 0\\
			0 & 0 & 0 & 0 & 0 & 0 & 0 & 0 & 0 & 0 & 0 & 0 & 0 & 0 & 0 & 0 & 0 & 0 & 0 & 0 & 0 & 0 & 0 & 0 & 0 & 0 & 0 & 0 & 0 & 0\\
			0 & 0 & 0 & 0 & 0 & 0 & 0 & 0 & 0 & 0 & 0 & 0 & 0 & 0 & 0 & 0 & 0 & 0 & 0 & -1 & 0 & 0 & 1 & 0
			&  0 & 0 & 0 & 0 & 0 & 1\\
			0 & 0 & 0 & 0 & 0 & 0 & 0 & 0 & 0 & 0 & 0 & 0 & 0 & 0 & 0 & 0 & 0 & 0 & 0 & 0 & 0 & 0 & 0 & 0 & 0 & 0 & 0 & 0 & 0 & 0\\
			0 & 0 & 0 & 0 & 0 & 0 & 0 & 0 & 0 & 0 & 0 & 0 & 0 & 0 & 0 & 0 & 0 & 0 & 0 & 0 & 0 & 0 & 0 & 0 & 0 & 0 & 0 & 0 & 0 & 0\\
			0 & 0 & 0 & 0 & 0 & 0 & 0 & 0 & 0 & 0 & 0 & 0 & 0 & 0 & 0 & 0 & 0 & 0 & 0 & 0 & 0 & 0 & 0 & -1
			&  0 & 0 & 1 & 0 & -1 & 0\\
			0 & 0 & 0 & 0 & 0 & 0 & 0 & 0 & 0 & 0 & 0 & 0 & 0 & 0 & 0 & 0 & 0 & 0 & 0 & 0 & 0 & 0 & 0 & 0
			& -1 & 0 & 0 & 1 & 0 & -1\\
			0 & 0 & 0 & 0 & 0 & 0 & 0 & 0 & 0 & 0 & 0 & 0 & 0 & 0 & 0 & 0 & 0 & 0 & 0 & 0 & 0 & 0 & 0 & 0
			&  1 & 0 & 0 & -1 & 0 & 1\\
			0 & 0 & 0 & 0 & 0 & 0 & 0 & 0 & 0 & 0 & 0 & 0 & 0 & 0 & 0 & 0 & 0 & 0 & 0 & 0 & 0 & 0 & 0 & 0 & 0 & 0 & 0 & 0 & 0 & 0\\
			0 & 0 & 0 & 0 & 0 & 0 & 0 & 0 & 0 & 0 & 0 & 0 & 0 & 0 & 0 & 0 & 0 & 0 & 0 & 0 & 0 & 0 & 0 & 1
			&  0 & 0 & -1 & 0 & 1 & 0\\
			0 & 0 & 0 & 0 & 0 & 0 & 0 & 0 & 0 & 0 & 0 & 0 & 0 & 0 & 0 & 0 & 0 & 0 & 0 & 0 & 0 & 0 & 0 & 0 & 0 & 0 & 0 & 0 & 0 & 0\\
			0 & 0 & 0 & 0 & 0 & 0 & 0 & 0 & 0 & 0 & 0 & 0 & 0 & 0 & 0 & 0 & 0 & 0 & 0 & 0 & 0 & 0 & 0 & 0 & 0 & 0 & 0 & 0 & 0 & 0\\
		\end{array}
		\right]^{\mathsf{T}}
		$}
\end{figure}

\subsection{Reduction of the chain complex $\LM$}
In this subsection, we will use \autoref{rt} from Section~\ref{red}, to reduce the chain complex $\LM$ to $\LMZ$ based on an acyclic pairing $\z$ on $\LM$. In the latter part of this subsection, we further reduce $\LMZ$ using \autoref{koz}.

In order to use \autoref{rt}, we first assign a gradient vector field on the nerve of $\M(\K)$.

We know that $\Nm= \{\alpha \subseteq [5] \mid \alpha \ne [5]\}$. We construct a gradient vector field $\VN$ on $\Nm$ as follows. For any $\sigma \in \Nm$, if $1 \notin \sigma$ and $\sigma \sqcup\{1\} \in \Nm$, then $(\sigma, \sigma \sqcup \{1\}) \in \VN$. The well-definedness and acyclicity of $\VN$ follow directly from the construction. So, the $\VN$-critical simplices are given by, $\Crit^{\VN}(\Nm)= \{\{1\}, \{2,3,4,5\}\}$.

Now, let us construct a pairing $\z$ on $\LM$ according to the construction given in Section~\ref{red}. We note from the previous section that for each $\alpha \in \Nm$, $\dim(\alpha) \le 2$, there exists a unique critical $0$-simplex in $\Aa$. Therefore, for any $(\alpha^{(q)}, \beta^{(q+1)})$, $q=0,1$, when we pair simplices in $L^{\beta}_{q+1}$ with simplices in $L^{\alpha}_q$, there is no question of choice. Each such pair is uniquely determined by the corresponding pair in $\VN$. The following tables (\autoref{pairs01} and \autoref{pairs12}) list the resulting pairs in $\z$ in such cases, among elements of $L_0, L_1$ and $L_1, L_2$ respectively.
\begin{table}[ht]
	\centering
	\begin{tabular}{ll}
		\toprule
		$\mathbf{\VN}$ & $\mathbf{\z}$ \\
		\midrule
		$(\{2\}, \{1,2\})$ & $(\{(1,2)\}_{\{2\}}, \{(2,3)\}_{\{1,2\}})$ \\
		$(\{3\}, \{1,3\})$ & $(\{(1,3)\}_{\{3\}},\{(2,2)\}_{\{1,3\}})$ \\
		$(\{4\}, \{1,4\})$ & $(\{(1,4)\}_{\{4\}}, \{(2,2)\}_{\{1,4\}})$ \\
		$(\{5\}, \{1,5\})$ & $(\{(1,5)\}_{\{5\}}, \{(2,2)\}_{\{1,5\}})$\\
		\bottomrule
	\end{tabular}
	\caption{Pairings between $L_0^{\alpha}$ and $L_1^{\beta}$ corresponding to each $(\alpha^{(0)}, \beta^{(1)}) \in \VN$.}\label{pairs01}
\end{table}

\begin{table}[ht]
	\centering
	\begin{tabular}{ll}
		\toprule
		$\mathbf{\VN}$ & $\mathbf{\z}$ \\
		\midrule
		$(\{2,3\}, \{1,2,3\})$ & $\{(2,1)\}_{\{2,3\}}, \{(2,4)\}_{\{1,2.3\}})$ \\
		$(\{2,4\}, \{1,2,4\})$ & $(\{(2,1)\}_{\{2,4\}},\{(2,3)\}_{\{1,2,4\}})$ \\
		$(\{2,5\}, \{1,2,5\})$ & $(\{(2,1)\}_{\{2,5\}}, \{(2,3)\}_{\{1,2,5\}})$ \\
		$(\{3,4\}, \{1,3,4\})$ & $(\{(2,1)\}_{\{3,4\}}, \{(2,2)\}_{\{1,3,4\}})$\\
		$(\{3,5\}, \{1,3,5\})$ & $(\{(2,1)\}_{\{3,5\}}, \{(2,2)\}_{\{1,3,5\}})$\\
		$(\{4,5\}, \{1,4,5\})$ & $(\{(2,1)\}_{\{4,5\}}, \{(2,2)\}_{\{1,4,5\}})$\\
		\bottomrule
	\end{tabular}
	\caption{Pairings between $L_1^{\alpha}$ and $L_2^{\beta}$ corresponding to each $(\alpha^{(1)}, \beta^{(2)}) \in \VN$.} \label{pairs12}
\end{table}
In the next step we list the pairs in $\z$ between elements of $L_2$ and $L_3$. Here, for each $(\alpha^{(2)}, \beta^{(3)}) \in \VN$, we order the simplices of $L_3^{\beta}$ in a particular order of our choice. Now, as in Section~\ref{red}, for each $(\alpha^{(2)}, \beta^{(3)}) \in \VN$, we choose the first simplex $\tau_\beta$ in $L_3^{\beta}$ in this ordering and pair it with the unique $\mathbf{t}(\tau_\beta, \alpha)\in L_2^{\alpha}$ associated with it. Accordingly, \autoref{pairs23} lists the pairings in $\z$ between elements of $L_2$ and $L_3$.

\begin{table}[ht]
	\centering
	\begin{tabular}{ll}
		\toprule
		$\mathbf{\VN}$ & $\mathbf{\z}$ \\
		\midrule
		$(\{2,3,4\}, \{1,2,3,4\})$ & $\{(2,1)\}_{\{2,3,4\}}, \{(4,5)\}_{\{1,2.3,4\}})$ \\
		$(\{2,3,5\}, \{1,2,3,5\})$ & $(\{(2,1)\}_{\{2,3,5\}},\{(3,4)_{\{1,2,3,5\}}\})$ \\
		$(\{2,4,5\}, \{1,2,4,5\})$ & $(\{(2,1)\}_{\{2,4,5\}}, \{(3,3)_{\{1,2,4,5\}}\})$ \\
		$(\{3,4,5\}, \{1,3,4,5\})$ & $(\{(2,1)\}_{\{3,4,5\}}, \{(2,2)_{\{1,3,4,5\}}\})$\\
		\bottomrule
	\end{tabular}
	\caption{Pairings between $L_2^{\alpha}$ and $L_3^{\beta}$ corresponding to $(\alpha^{(2)}, \beta^{(3)}) \in \VN$.} \label{pairs23}
\end{table}

Therefore, from \autoref{acyc}, we conclude that $\z$ is an acyclic pairing on $\LM$ and from \autoref{rt}, we deduce that the homology groups of the reduced chain complex $\LMZ$ are isomorphic to those of $\LM$, that is, $H^{\Lq^{\z}}_{\#}(\M(\K)) \cong H^{\Lq}_{\#}(\M(\K))$. We note that after this reduction, $|L_0^{\z}|=1$, $|L_1^{\z}|=0$, $|L_2^{\z}|=20$ and $|L_3^{\z}|=76$. It can be observed here that the elements of $L_2^{\z}$ and $L_3^{\z}$ allow further reduction. So, now we reduce this chain complex further using \autoref{koz} (Subsection~\ref{alg}).

Let us rename $\LMZ$ as $\F$. We represent $\F$ as a ranked poset, where for $y \in \Lz_{q+1}$ and $x \in \Lz_{q}$, $y \le x$  if and only if $\langle y,x \rangle_{\Lz}=\sum_{P \in \mathcal{P}(y,x)}w_L(P) \ne 0$.  Now, we again assign an acyclic pairing $\z'$ on this reduced chain complex $\F$ based on the same gradient vector field on $\N$.  For each $(\alpha^{(1)}, \beta^{(2)}) \in \VN$, we choose a well-defined collection of pairs $\{(\sigma^{(1)}_{\alpha}, \tau^{(1)}_\beta) \mid \langle\tau_\beta, \sigma_\alpha \rangle_{\Lz}= \pm 1\}$. So, we have $\z'= \{(\sigma^{(1)}_{\alpha}, \tau^{(1)}_\beta) \mid \langle\tau_\beta, \sigma_\alpha \rangle_{\Lz}= \pm 1, (\alpha^{(1)}, \beta^{(2)}) \in \VN\}$. We list the pairings of $\z'$ in \autoref{pairz23}.
\begin{table}[ht]
	\centering
	\begin{tabular}{ll}
		\toprule
		$\mathbf{\VN}$ & $\mathbf{\z'}$ \\
		\midrule
		$(\{2,3\}, \{1,2,3\})$ & $\{(3,1),(5,4)\}_{\{2,3\}}, \{(3,4),(5,5)\}_{\{1,2,3\}})$ \\
		$(\{2,3\}, \{1,2,3\})$ & $\{(4,1),(5,4)\}_{\{2,3\}}, \{(4,4),(5,5)\}_{\{1,2,3\}})$ \\
		$(\{2,4\}, \{1,2,4\})$ & $\{(3,1),(5,3)\}_{\{2,4\}}, \{(3,5),(5,3)\}_{\{1,2,4\}})$ \\
		$(\{2,4\}, \{1,2,4\})$ & $\{(4,1),(5,3)\}_{\{2,4\}}, \{(4,5),(5,3)\}_{\{1,2,4\}})$ \\
		$(\{2,5\}, \{1,2,5\})$ & $\{(3,1),(5,3)\}_{\{2,5\}}, \{(3,4),(5,3)\}_{\{1,2,5\}})$ \\
		$(\{2,5\}, \{1,2,5\})$ & $\{(4,1),(5,3)\}_{\{2,5\}}, \{(4,3),(5,4)\}_{\{1,2,5\}})$ \\
		$(\{3,4\}, \{1,3,4\})$ & $\{(3,1),(5,2)\}_{\{3,4\}}, \{(3,2),(5,5)\}_{\{1,3,4\}})$ \\
		$(\{3,4\}, \{1,3,4\})$ & $\{(4,1),(5,2)\}_{\{3,4\}}, \{(4,5),(5,2)\}_{\{1,3,4\}})$ \\
		$(\{3,5\}, \{1,3,5\})$ & $\{(4,1),(5,2)\}_{\{3,5\}}, \{(4,4),(5,2)\}_{\{1,3,5\}})$ \\
		$(\{3,5\}, \{1,3,5\})$ & $\{(3,1),(5,2)\}_{\{3,5\}}, \{(3,4),(5,2)\}_{\{1,3,5\}})$ \\
		$(\{4,5\}, \{1,4,5\})$ & $\{(4,1),(5,2)\}_{\{4,5\}}, \{(4,3),(5,2)\}_{\{1,4,5\}})$ \\
		$(\{4,5\}, \{1,4,5\})$ & $\{(3,1),(5,2)\}_{\{4,5\}}, \{(3,2),(5,3)\}_{\{1,4,5\}})$ \\
		\bottomrule
	\end{tabular}
	\caption{Pairings between $L_2^{\z}$ and $L_3^{\z}$ corresponding to $(\alpha^{(1)}, \beta^{(2)}) \in \VN$.} \label{pairz23}
\end{table}

It can be checked that $\z'$ is a well-defined and acyclic pairing on $\F$. Hence, we can define $\z'$-paths and the boundary map $\partial^{\z'}_{q+1}: F_{q+1}^{\z'} \rightarrow F_q^{\z'}$, for $q \ge 0$, as in Subsection~\ref{alg}. So, $\LMz$ is the reduced chain complex with respect to the acyclic pairing $\z'$. Using \autoref{koz}, we can infer that the chain complexes $\F$ and $\LMz$ have isomorphic homology groups. Hence, $H_{\#}(\M(\K)) \cong H^{F}_{\#}(\M(\K)) \cong H^{F^{\z'}}_{\#}(\M(\K))$. We have computed the boundary matrices of this reduced chain complex $\LMz$ using Python (a similar setup of Algorithm~\ref{alg reduction}). We observe that the only non-trivial boundary map in this reduced chain complex is $\partial^{\mathcal{F}^{\z'}}_3$ whose matrix is given as follows.

\begin{figure}[H]
	\centering
	\setlength{\arraycolsep}{4pt}
	
	\resizebox{0.37\textwidth}{!}{$
		\partial^{\mathcal{F}^{\z'}}_3=\left[
		\begin{array}{*{32}{c}}
			0 & 0 & 0 & 0 & 0 & 0 & 0 & 0\\
			0 & 0 & 0 & 0 & 0 & 0 & 0 & 0\\
			0 & 0 & 0 & 0 & 0 & 0 & 0 & 0\\
			0 & 0 & 0 & 0 & 0 & 0 & 0 & 0\\
			1  & -1 & -1 & 1 & 1 & -1 & 0 & 0\\
			1 &  0 & -1 & 0 & 1 & 0 & 0 & 0\\
			0 & -1 & 0 & 1 & 0 & -1 & 0 & 0\\
			0 & -1 & 0 & 1 & 0 & 0 & 0 & 1\\
			1 & 0 & -1 & 0 & 0 & 0 & -1 & 0\\
			1 & -1 & -1 & 1 & 0 &  0 & -1 & 1\\
			0 & 0 & 0 & 0 & 0 & 0 & 0 & 0\\
			0 & 0 & 0 & 0 & 0 & 0 & 0 & 0\\
			0 & 0 & 0 & 0 & 0 & 0 & 0 & 0\\
			0 &  1 & 0 & 0 & 0 & 1 & 0 & -1\\
			-1 & 0 & 0 & 0 & -1 & 0 & 1 & 0\\
			-1 & 1 & 0 & 0 & -1 & 1 & 1 & -1\\
			0 & 0 & 0 & 0 & 0 & 0 & 0 & 0\\
			0 & 0 & 0 & 0 & 0 & 0 & 0 & 0\\
			0 & 0 & 0 & 0 & 0 & 0 & 0 & 0\\
			0 & 0 & 0 & 0 & 0 & 0 & 0 & 0\\
			0 & 0 & 0 & 0 & 0 & 0 & 0 & 0\\
			0 & 0 & 0 & 0 & 0 & 0 & 0 & 0\\
			0 & 0 & 0 & 0 & 0 & 0 & 0 & 0\\
			0 & 0 & -1 & -1 & 1 & 1 & -1 & -1\\
			0 & 0 & -1 & 0 & 1 & 0 & -1 & 0\\
			0 & 0 & 0 & -1 & 0 & 1 & 0 & -1\\
			0 & 0 & 0 & 0 & 0 & 0 & 0 & 0\\
			0 & 0 & 0 & 0 & 0 & 0 & 0 & 0\\
			0 & 0 & 0 & 0 & 0 & 0 & 0 & 0\\
			0 & 0 & 0 & 0 & 0 & 0 & 0 & 0\\
			0 & 0 & 0 & 0 & 0 & 0 & 0 & 0\\
			0 & 0 & 0 & 0 & 0 & 0 & 0 & 0\\
			0 & 0 & 0 & 0 & 0 & 0 & 0 & 0\\
			0 & 0 & 0 & 0 & 0 & 0 & 0 & 0\\
			0 & 0 & 0 & 0 & 0 & 0 & 0 & 0\\
			0 & 0 & 0 & 0 & 0 & 0 & 0 & 0\\
			-2 & 0 & 2 & 0 & -2 & 0 & 0 & 0\\
			2 & 0 & -2 & 0 & 2 & 0 & 0 & 0\\
			0 & -2 & 0 & 2 & 0 & -2 & 0 & 0\\
			0 & 2 & 0 & -2 & 0 & 2 & 0 & 0\\
			-2 & -3 & 2 & 2 & -2 & -2 & 1 & 3\\
			1 & 0 & -1 & -1 & 1 & 1 & -2 & 0\\
			-2 & 0 & 2 & -1 & -2 & 1 & 1 & 0\\
			1 & -3 & -1 & 2 & 1 & -2 & -2 & 3\\
			-2 & 0 & 2 & 0 & 0 & 0 & 2 & 0\\
			2 & 0 & -2 & 0 & 0 & 0 & -2 & 0\\
			0 & -2 & 0 & 2 & 0 & 0 & 0 & 2\\
			0 & 2 & 0 & -2 & 0 & 0 & 0 & -2\\
			0 & 0 & 0 & 0 & 0 & 0 & 0 & 0\\
			0 & 0 & 0 & 0 & 0 & 0 & 0 & 0\\
			2 & 0 & 0 & 0 & 2 & 0 & -2 & 0\\
			0 & 2 & 0 & 0 & 0 & 2 & 0 & -2\\
			-2 & 0 & 0 & 0 & -2 & 0 & 2 & 0\\
			0 & 0 & 0 & 0 & 0 & 0 & 0 & 0\\
			0 & -2 & 0 & 0 & 0 & -2 & 0 & 2\\
			0 & 0 & 0 & 0 & 0 & 0 & 0 & 0\\
			0 & 0 & 0 & 0 & 0 & 0 & 0 & 0\\
			0 & 0 & -2 & 0 & 2 & 0 & -2 & 0\\
			0 & 0 & 0 & -2 & 0 & 2 & 0 & -2\\
			0 & 0 & 0 &  2 & 0 & -2 & 0 & 2\\
			0 & 0 & 0 & 0 & 0 & 0 & 0 & 0\\
			0 & 0 & 2 & 0 & -2 & 0 & 2 & 0\\
			0 & 0 & 0 & 0 & 0 & 0 & 0 & 0\\
			0 & 0 & 0 & 0 & 0 & 0 & 0 & 0\\
		\end{array}
		\right]^{\mathsf{T}}
		$}
\end{figure}

We note here that if we consider the simplicial chain complex of $\M(\K)$, then the number of generators of the chain groups are, $|S_0|= 25$, $|S_1|=200$, $|S_2|=600$, $|S_3|=600$, $|S_4|=120$, whereas the number of generators in the case of  our reduced chain complex $\LMz$ are given by, $|\Crit_0^{\z'}(\Lz)|=1$, $|\Crit_1^{\z'}(\Lz)|= 0$, $|\Crit_2^{\z'}(\Lz)|=8$, $|\Crit_3^{\z'}(\Lz)|=64$. Thus, we observe a significant reduction in the number of generators of the chain groups.
Next, we have computed the homology groups of $\LMz$ using the aforementioned boundary matrices through SageMath. These are given by,

$$ H^{F^{\z'}}_0 \cong \Z, H^{F^{\z'}}_1 \cong 0, H^{F^{\z'}}_2 \cong \Z_3, H^{F^{\z'}}_3 \cong \Z^{56}.$$
which are consistent with the values in existing literature (see \cite{Wachs2003}).
\bibliographystyle{plain}
\bibliography{ref.bib}

\end{document}